\documentclass[12pt]{amsart}

\usepackage{epic}

\usepackage[active]{srcltx}
\usepackage[parfill]{parskip}
\usepackage[e]{esvect}

\usepackage{
   a4wide, 
  amsmath,
  amsfonts,
  amssymb,
  mathtools,
  graphicx, 
  times,
  color,
  hyphenat, 
  stmaryrd, datetime, bbm, thmtools, verbatim
}
\usepackage[dvipsnames]{xcolor}
\usepackage{tikz-cd}

\usepackage[shortlabels]{enumitem}

\usepackage{mathdots}
\usepackage[all]{xy}
\usepackage{euscript,mathrsfs}
\usepackage{bbm,xspace}

\usepackage[final]{hyperref, bookmark}

\usepackage[modulo]{lineno}
\usepackage{dirtytalk}

\usepackage{tikz}
\usetikzlibrary{cd}
\usetikzlibrary{decorations}
\usetikzlibrary{decorations.markings}
\usetikzlibrary{decorations.pathreplacing}
\usetikzlibrary{decorations.pathmorphing}
\usetikzlibrary{arrows.meta,shapes,positioning,matrix,calc}
\usetikzlibrary{shapes.callouts}
\tikzstyle{densely dotted}=[dash pattern=on \pgflinewidth off 0.5pt]
\tikzset{anchorbase/.style={baseline={([yshift=-0.5ex]current bounding box.center)}},
tinynodes/.style={font=\tiny,text height=0.25ex,text depth=0.05ex},
smallnodes/.style={font=\scriptsize,text height=0.75ex,text depth=0.15ex},
usual/.style={line width=0.9,color=black},
dusual/.style={line width=0.9,color=spinach,densely dashed},
pole/.style={line width=3.0,color=specialgray},
crossline/.style={preaction={draw=white,line width=8.0pt,-},preaction={draw=black,line width=0.9pt,-}},
crosslinedashed/.style={preaction={draw=white,line width=3.0pt,-},preaction={draw=white,line width=0.6pt,-}},
crosspole/.style={preaction={draw=white,line width=6.0pt,-},preaction={draw=specialgray,line width=3.0pt,-}},
mor/.style={line width=0.75,color=black,fill=cream},
blob/.style={circle,fill,minimum size=5.0pt,inner sep=0pt,outer sep=0pt},
blobed/.style n args={3}{decoration={markings,post length=0.5mm,pre length=0.5mm,
mark=at position #1 with {\node[blob,#3,label=left:$#2\!$]at (0,0){};}
},postaction={decorate}},
rblobed/.style n args={3}{decoration={markings,post length=0.5mm,pre length=0.5mm,
mark=at position #1 with {\node[blob,#3,label=right:$\!#2$]at (0,0){};}
},postaction={decorate}},
cutline/.style={line width=2.25,color=black!50,densely dotted},
}
\tikzstyle directed=[postaction={decorate,decoration={markings,mark=at position #1 with {\arrow[line width=0.25mm, black]{>}}}}]
\tikzstyle rdirected=[postaction={decorate,decoration={markings,mark=at position #1 with {\arrow[line width=0.25mm, black]{<}}}}]

\newcommand{\tikzdiagh}[2][]{\tikz[#1,thick,baseline={([yshift=1ex+#2]current bounding box.center)}]}
\newcommand{\tikzdiagc}[1][]{\tikzdiagh[#1]{-1ex}}

\tikzstyle{tikzdot}=[fill, circle, inner sep=2pt]
\tikzstyle{smoltikzdot}=[fill=white, draw=black, circle, inner sep=1pt]

\tikzset{
    partial ellipse/.style args={#1:#2:#3}{
        insert path={+ (#1:#3) arc (#1:#2:#3)}
    }
}

\input xy
\xyoption{all}
\title{Induction for extended affine type $A$ Soergel bimodules: first steps}

\author{Marco Mackaay}
\address{M.M.: Departamento de Matemática, FCT, Universidade do Algarve, Campus de Gambelas,
8005-139 Faro, Portugal \&
Center for Research and Development in Mathematics and Applications (CIDMA), 
Department of Mathematics, University of Aveiro, 3810-193 Aveiro, Portugal,  
\newline \href{https://fct.ualg.pt/bio/mmackaay}{https://fct.ualg.pt/bio/mmackaay}, \href{https://orcid.org/0000-0001-9807-6991}{ORCID 0000-0001-9807-6991}}
\email{mmackaay@ualg.pt}
\author{Vanessa Miemietz}
\address{V.M.: School of Mathematics, University of East Anglia, Norwich NR4 7TJ, United Kingdom,  \newline \href{https://archive.uea.ac.uk/~byr09xgu/}{https://archive.uea.ac.uk/~byr09xgu/}}
\email{v.miemietz@uea.ac.uk}
\author{Pedro Vaz}
\address{P.V.: Institut de Recherche en Math{\'e}matique et Physique, 
Universit{\'e} catholique de Louvain, Chemin du Cyclotron 2,  
1348 Louvain-la-Neuve, Belgium, \newline \href{https://perso.uclouvain.be/pedro.vaz}{https://perso.uclouvain.be/pedro.vaz}, \href{https://orcid.org/0000-0001-9422-4707}{ORCID 0000-0001-9422-4707}}
\email{pedro.vaz@uclouvain.be}
\newcommand{\bfM}{\mathbf{M}}

\newcommand{\ti}{\mathtt{i}}
\newcommand{\tj}{\mathtt{j}}
\newcommand{\tk}{\mathtt{k}}
\newcommand{\tl}{\mathtt{l}}

\newcommand{\rA}{\mathrm{A}}
\newcommand{\rB}{\mathrm{B}}
\newcommand{\rC}{\mathrm{C}}

\newcommand{\rF}{\mathrm{F}}
\newcommand{\rG}{\mathrm{G}}

\newcommand{\rI}{\mathrm{I}}

\newcommand{\rT}{\mathrm{T}}
\newcommand{\rX}{\mathrm{X}}
\newcommand{\rY}{\mathrm{Y}}

\DeclareMathOperator{\ind}{ind}

\newcommand{\circh}{\circ_{\mathsf{h}}}
\newcommand{\circv}{\circ_{\mathsf{v}}}

\newcommand{\rmod}{\operatorname{mod}}
\newcommand{\brmod}{\boldsymbol{\operatorname{mod}}}

\newcommand{\bfm}{\mathbf{m}}
\newcommand{\bfu}{\mathbf{u}}

\newcommand{\cpl}[1]{{#1}^\circ}
\newcommand{\copr}[1]{{#1}^\sqcup}

\newcommand{\Vect}{\operatorname{Vect}}
\newcommand{\one}{\mathbbm{1}}

\newcommand{\Hom}{\mathrm{Hom}}

\newcommand{\End}{\mathrm{End}}

\newcommand{\add}{\operatorname{add}}

\definecolor{myblue}{rgb}{0,.5,1}
\definecolor{myred}{rgb}{0.9,0,0}
\definecolor{mygreen}{rgb}{0,0.7,0}

\newcommand{\cop}[1]{{#1}^\diamond}

\newcommand{\Bi}{\textcolor{blue}{\rB_i}}

\newcommand{\Brhopm}{\rB_{\rho}^{\pm 1}}

\newcommand{\Ti}{\textcolor{blue}{\rT_i}}
\newcommand{\Tim}{\textcolor{blue}{\rT_i^{-1}}}

\newcommand{\Sext}{\widehat{\eS}^{\mathrm{ext}}}
\newcommand{\Sextdiamond}{\widehat{\eS}^{\mathrm{ext}, \diamond}}
\newcommand{\Sextstar}{\widehat{\eS}^{\mathrm{ext},*}}
\newcommand{\BSext}{\widehat{\eBS}^{\mathrm{ext}}}
\newcommand{\BSextstar}{\widehat{\eBS}^{\mathrm{ext},*}}
\newcommand{\fSext}{\overline{\eS}^{\,\mathrm{ext}}}

\newcommand{\rhoL}{{\rho_{\color{blue}L}}}
\newcommand{\rhoR}{{\rho_{\textcolor{purple}R}}}

\newcommand{\affh}{\widehat{H}_n}
\newcommand{\eaffh}{\widehat{H}^{ext}_n}
\newcommand{\Sy}{\mathfrak{S}}
\newcommand{\Syaff}{\widehat{\mathfrak{S}}}

\newcommand{\ybox}{
\xy (0,0)*{
\tikzdiagc[scale=.375]{
\draw[very thick] (0,0) to (1,0) to (1,1) to (0,1) -- cycle;
\node at (.5,.5) {$y$};
}}\endxy
}

\newcommand{\eah}[1]{\widehat{H}^{\mathrm{ext}}_{#1}}
\newcommand{\eahCq}[1]{\widehat{H}^{\mathrm{ext},\mathbb{C}(q)}_{#1}}
\newcommand{\feah}[1]{\overline{H}^{\,\mathrm{ext}}_{#1}}

\newcommand{\ib}{\textcolor{blue}{i}}
\newcommand{\jr}{\textcolor{myred}{j}}
\newcommand{\ko}{\textcolor{orange}{k}}

\newcommand{\gi}{\raisebox{-4.5pt}{$i$}}

\newcommand{\vastl}[1]{\left(\rule{#1 cm}{.9cm}\right.}
\newcommand{\vastr}[1]{\left.\rule{#1 cm}{.9cm}\right)}

\newcommand{\xra}[1]{\xrightarrow{#1}}

\newcommand{\raltseq}[2]{\widehat{#1}_{\,#2}}
\newcommand{\laltseq}[2]{{}_{#2}\widehat{#1}}
\newcommand{\lraltseq}[3]{{}_{#2}\widehat{#1}_{\,#3}}

\newcommand\F{{\sf{F}}}
\newcommand\K{{\sf{K}}}

\DeclareMathOperator{\id}{Id}

\newtheorem{thm}{Theorem}[section]
\newtheorem{lem}[thm]{Lemma}
\newtheorem{cor}[thm]{Corollary}
\newtheorem{prop}[thm]{Proposition}
\newtheorem{exe}[thm]{Example}
\newtheorem{conj}[thm]{Conjecture}
\theoremstyle{definition}
\newtheorem{defn}[thm]{Definition}

\newtheorem{rem}[thm]{Remark}

\def\F{{\sf{F}}}

\newcommand{\bR}{\mathbb{R}}

\newcommand{\eS}{\EuScript{S}}

\newcommand{\eBS}{\EuScript{BS}}

\newcommand{\cA}{\mathcal{A}}
\newcommand{\cB}{\mathcal{B}}
\newcommand{\cC}{\mathcal{C}}
\newcommand{\cD}{\mathcal{D}}
\newcommand{\cE}{\mathcal{E}}

\catcode`\@=11
\long\def\@makecaption#1#2{%
    \vskip 10pt
    \setbox\@tempboxa\hbox{%
\small{#1: }\ignorespaces #2}%
    \ifdim \wd\@tempboxa >\captionwidth {%
        \rightskip=\@captionmargin\leftskip=\@captionmargin
        \unhbox\@tempboxa\par}%
      \else
        \hbox to\hsize{\hfil\box\@tempboxa\hfil}%
    \fi}
\newdimen\@captionmargin\@captionmargin=2\parindent
\newdimen\captionwidth\captionwidth=\hsize
\catcode`\@=12

\let\fullref\autoref
\def\makeautorefname#1#2{\expandafter\def\csname#1autorefname\endcsname{#2}}
\makeautorefname{lem}{Lemma}%
\makeautorefname{prop}{Proposition}%
\makeautorefname{rem}{Remark}%
\makeautorefname{section}{Section}%
\makeautorefname{subsection}{Section}%
\makeautorefname{subsubsection}{Section}%

\begin{document}
%
%
%
\begin{abstract}
In this paper we take the first steps towards the categorification of the Zelevinsky tensor product of finite dimensional representations of 
extended affine type $A$ Hecke algebras. 
\end{abstract}
\maketitle

{\hypersetup{hidelinks}
\tableofcontents 
}
%
%
\pagestyle{myheadings}
\markboth{\em\small M.~Mackaay, V.~Miemietz, P.~Vaz}{\em\small  Parabolic induction}
%
%

\section{Introduction}\label{sec:intro}

This paper is part of our ongoing study of the birepresentation theory of extended affine type $A_{n-1}$ Soergel bimodules in characteristic zero, which 
we started in~\cite{mmv-evalfunctor}. The monoidal category of these Soergel bimodules, denoted $\Sext_n$ in this paper, was studied both algebraically and diagrammatically in~\cite{mackaay-thiel} and~\cite{elias2018}, and categorifies the extended affine type $A$ Hecke algebra $\eah{n}$. 
This algebra is infinite dimensional, which on the categorical level corresponds to the fact that 
$\Sext_n$ is not finitary but {\em wide finitary}, in the terminology of~\cite{macpherson22a}. 
Whereas finitary birepresentation theory of finitary monoidal categories/bicategories, such as Soergel bimodules of finite Coxeter type, is fairly well developed, see e.g. \cite{mmmtz2019} and 
\cite{mmmtz2020} and references therein, birepresentation theory of wide finitary monoidal categories/bicategories is still in its infancy. Some fundamental results in finitary birepresentation theory 
were generalized to the wide finitary setting by Macpherson in~\cite{macpherson22a}, but his framework does not cover triangulated birepresentations, which play a prominent role in the birepresentation 
theory of $\Sext_n$. For example, the evaluation birepresentations of $\Sext_n$ in~\cite{mmv-evalfunctor} are triangulated. In this paper, we continue our study of the birepresentation theory of $\Sext_n$, taking a first step towards the categorification of parabolic induction. As the reader will see, this also involves triangulated birepresentations. 

Parabolic induction and restriction play an important role in the finite dimensional representation theory of $\eah{n}$. 
In particular, for any integers $1\leq k< n$, there is a well-known embedding of algebras  
\begin{equation}\label{eq:embedding}
\psi_{k,n-k}\colon \eah{k}\otimes \eah{n-k}\to \eah{n},
\end{equation}
which is the analog of the embedding of finite symmetric groups $\Sy_k\times \Sy_{n-k}\to \Sy_{n}$. 
Given two finite dimensional representations $M_1$ and $M_2$ of $\eah{k}$ and $\eah{n-k}$, respectively, 
one can use $\psi_{k,n-k}$ to define their {\em Zelevinsky tensor product} 
\begin{equation}\label{eq:Zelevinskytensor}
M_1\odot M_2:=\mathrm{Ind}_{\eah{k}\otimes \eah{n-k}}^{\eah{n}} M_1\otimes M_2,
\end{equation}
which is a finite dimensional $\eah{n}$-representation. Zelevinsky~\cite{zelevinsky} classified the 
finite dimensional irreducible representations of the extended affine type $A$ Hecke algebra in terms of combinatorial 
objects called multisegments. Leclerc, Nazarov and Thibon~\cite{LNTh} gave a necessary and 
sufficient condition for irreducibility of $M_1\odot M_2$, based on the multisegments corresponding  to two finite dimensional irreducible representations $M_1$ and $M_2$. 

Now, let $k_1,k_2,k_3\in \mathbb{Z}_{\geq 1}$ such that $k_1+k_2+k_3=n$ (assuming that $n\geq 3$, of course). Then one can show that 
\begin{equation}\label{eq:associativity1}
\psi_{k_1+k_2, k_3}(\psi_{k_1, k_2}\otimes \mathrm{id}_{\eah{k_3}})=
\psi_{k_1, k_2+k_3}(\mathrm{id}_{\eah{k_1}}\otimes \psi_{k_2,k_3}), 
\end{equation}
which implies that there is a canonical isomorphism 
\begin{equation}\label{eq:associativity2}
\left(M_1\odot M_2\right) \odot M_3 \cong M_1\odot \left(M_2 \odot M_3\right),
\end{equation}
with $M_i$ being a finite dimensional representation of $\eah{k_i}$ for $i=1,2,3$.
If we define the embedding of algebras 
\[
\psi_{k_1,k_2,k_3}\colon \eah{k_1}\otimes \eah{k_2}\otimes\eah{k_3} \to \eah{n}
\]
as either one of the two composite maps in~\eqref{eq:associativity1}, then $M_1\odot M_2 \odot M_3$ can also be 
defined directly as 
\[
\mathrm{Ind}_{\eah{k_1}\otimes \eah{k_2}\otimes \eah{k_3}}^{\eah{n}} M_1\otimes M_2\otimes M_3,
\]
which by the above is canonically isomorphic to either one of the two isomorphic representations in~\eqref{eq:associativity2}.

More generally, let $k_1, \ldots, k_m\in \mathbb{Z}_{\geq 1}$ such that $1\leq m\leq n$ and $k_1+\ldots +k_m=n$. Then there is an 
embedding of algebras 
\begin{equation}\label{eq:embedding2}
\psi_{k_1,\ldots, k_m}\colon \eah{k_1}\otimes \dotsb \otimes\eah{k_m} \to \eah{n},
\end{equation}
which can be used to define the corresponding Zelevinsky tensor product
\[
M_1 \odot \dotsb \odot M_m,
\]
with $M_i$ being a finite dimensional representation of $\eah{k_i}$ for $i=1,\ldots, m$, 
and this tensor product is associative up to canonical isomorphism. For more information on the Zelevinsky tensor product 
and its role in representation theory, see~\cite{LNTh} and references therein. 

In this paper, we initiate the categorification of the Zelevinsky tensor product. Specifically, we define a linear, monoidal functor 
\begin{equation}\label{eq:catembedding}
\Psi_{k,n-k}\colon \Sext_k\boxtimes \Sext_{n-k} \to K^b(\Sext_n)
\end{equation}
which categorifies the embedding $\psi_{k,n-k}$ in~\eqref{eq:embedding}. Here $K^b(\Sext_n)$ denotes the homotopy category of bounded 
complexes in $\Sext_n$. Before continuing, we should note that $\psi_{k,n-k}$ is 
usually defined in terms of the Bernstein presentation of the extended affine type $A$ Hecke algebra. Unfortunately, there is currently 
no categorification of $\eah{n}$ based on that presentation. The decategorification of $\Sext_n$ is naturally associated to the Kazhdan-Lusztig 
presentation of $\eah{n}$, so we define $\psi_{k,n-k}$ in terms of that presentation and use it as the starting point for the definition of $\Psi_{k,n-k}$ in this paper. 
Note that in \cite{StWe2024}, the authors use the categorification of parabolic induction for finite type $A$ Soergel bimodules, where the analog of $\Psi_{k,n-k}$ is much easier to define and does not involve Rouquier complexes.

The way to categorify the Zelevinsky tensor product using $\Psi_{k,n-k}$ is a bit roundabout, because we do not know how to make sense of 
something like 
\[
\Sext_n \boxtimes_{\,\Sext_k\,\boxtimes\, \Sext_{n-k}} \mathbf{M}_1\boxtimes \mathbf{M}_2,
\]
where $\mathbf{M}_1$ and $\mathbf{M}_2$ are finitary birepresentations of $\Sext_k$ and $\Sext_{n-k}$, respectively. This would require an 
analog of the balanced box tensor product for module categories over finite tensor categories, defined in~\cite[Section 2.7]{DaNi}, which does not (yet) exist in our setting. Therefore, we follow a different approach, explained below. We first sketch the general idea and then point out some technical hurdles. Recall that, for a fiat bicategory $\mathcal{C}$ (e.g. a linear additive pivotal category satisfying certain finiteness conditions), there is a correspondence between transitive finitary birepresentations of $\mathcal{C}$ and algebra $1$-morphisms in (the abelianization of) $\mathcal{C}$, see~\cite[Corollary 4.8]{mmmt2019}. Note that in that same paper, and subsequent ones like~\cite{mmmtz2020}, the correspondence was first formulated 
with coalgebra $1$-morphisms, instead of algebra $1$-morphisms. For technical reasons, 
we prefer algebra $1$-morphisms in this paper, see the comments below. Moreover, the bicategories in this paper have only one object and are hence monoidal categories, thus we speak about algebra objects. Any monoidal functor $\mathcal{C}\to \mathcal{D}$ between monoidal categories maps algebra objects to algebra objects and can, therefore, be used to induce birepresentations of $\mathcal{C}$ to birepresentations of $\mathcal{D}$, which is how we would like to go about categorifying the Zelevinsky tensor product. Unfortunately, $\Sext_n$ is not 
finitary but wide finitary, and $\Psi_{k,n-k}$ does not take values in $\Sext_n$ but in $K^b(\Sext_n)$. 

In~\cite{macpherson22a}, a correspondence between wide finitary birepresentations and (co)algebra 
objects in some completion is given, but there is currently no analog for triangulated birepresentations. In the examples of finitary or triangulated birepresentations of $\Sext_n$ that we fully understand, the algebra objects are typically given by countable coproducts of tensor products of Soergel bimodules and Rouquier complexes. Hence we introduce certain cocompletions of additive and triangulated monoidal categories/bicategories containing these 
kinds of algebra objects/1-morphisms in \autoref{sec:completions}, which are smaller than 
the completions considered in~\cite{macpherson22a}. To explain our preference for algebra objects over 
coalgebra objects in this paper, note that algebra objects typically belong to cocompletions  
whereas coalgebra objects typically belong to completions. As can be seen from our calculations in \autoref{sec:excatstory}, cocompletions are easier to work with in practice. In \autoref{sec:excatstory} we work out the explicit example of the categorified Zelevinsky tensor product $\mathbf{W}:=\mathbf{V}\boxdot\mathbf{V}$ of the trivial finitary birepresentation $\mathbf{V}$ of $\Sext_1$ with itself. Its construction is technically quite involved. We first define and analyse a wide finitary $\Sext_2$-birepresentation $\mathbf{U}$ corresponding to the algebra object $\rY:=\Psi_{1,1}(\rX\boxtimes \rX))$ in the cocompletion of $K^b(\Sext_2)$, where $\rX$ is the algebra object in the cocompletion of $\Sext_1$ corresponding to $\mathbf{V}$. Following the terminology introduced in \cite{mmv-evalfunctor}, we call $\mathbf{U}$ a 
{\em wide finitary cover} of $\mathbf{W}$. To obtain a triangulated birepresentation, 
we then observe that $\rY$ induces a right action of $\mathbb{Z}^2$ on $K^b(\Sext_2)$ which commutes with the left action of $\Sext_2$, so the 
orbit category $\Omega$ inherits the structure of a $\Sext_2$ birepresentation. We prove that $\mathbf{U}$ is equivalent to the full subcategory 
of $\Omega$ whose objects can be identified with those of $\Sext_2$ under the usual embedding into $K^b(\Sext_2)$. In general, 
the orbit category of a triangulated category under a group action does not inherit a natural triangulated structure, but in this case we can relate it (via a technical detour which we will explain in \autoref{sec:orbitcats}) to a triangulated orbit category using a dg-enhancement as 
in \cite[Theorem/Definition 1.1]{FKQ}. By construction, this triangulated orbit category is a $\Sext_2$ birepresentation and we conjecture that its triangulated Grothendieck group is isomorphic to $W$, see \autoref{conj:example}. It seems likely that this approach can be generalized, but we do not develop a general theory of these kinds of birepresentations and the corresponding algebra objects/1-morphisms in this paper, leaving that for future work.

Something else beyond the scope of this paper is the generalization of $\Psi_{k,n-k}$ to a linear monoidal functor  
\[
\Psi_{k_1,\ldots, k_m}\colon \Sext_{k_1}\boxtimes \dotsb \boxtimes \Sext_{k_m} \to K^b(\Sext_n)
\]
for $k_1, \ldots, k_m\in \mathbb{Z}_{\geq 1}$ such that $1\leq m\leq n$ and $k_1+\ldots +k_m=n$, categorifying \eqref{eq:embedding2}. Conceptually it is clear how to define such a generalization, but we postpone the lengthy technical details to a future paper. Finally, to prove the categorical analog of~\eqref{eq:associativity1}, i.e., the existence of a natural isomorphism 
\[
\Psi_{k_1+k_2, k_3}(\Psi_{k_1, k_2}\boxtimes \mathrm{Id}_{\Sext_{k_3}}) \cong 
\Psi_{k_1, k_2+k_3}(\mathrm{Id}_{\Sext_{k_1}}\boxtimes \Psi_{k_2,k_3}) 
\]
for $1\leq k_1, k_2, k_3\leq n$ such that $k_1+k_2+k_3=n$, we would first need to extend 
$\Psi_{k,n-k}$ to a monoidal functor (with the same notation)
\[
\Psi_{k,n-k}\colon K^b(\Sext_k) \boxtimes K^b(\Sext_{n-k}) \to K^b(\Sext_n).
\]
Proving the existence of such an extension is a non-trivial problem and is related to similar extension problems in other contexts, see 
\cite[Conjecture 1.2]{allr2024}, \cite[Section 1.6]{elias2018}, \cite{ElHo} and~\cite[Section 1]{mmv-evalfunctor}. The solution of this problem is also a topic for future research.


\subsubsection*{Acknowledgements} 
M.M. is partially supported by CIDMA under the FCT (Portuguese Foundation for Science and Technology) grant UID/04106. 
V.M. is supported by EPSRC grant EP/Z533750/1.
P.V. is supported by the Fonds de la Recherche Scientifique-FNRS (Belgium) under Grant no. J.0189.23.

%
%

\section{Parabolic embedding: decategorified story}\label{sec:decatreminders}

Throughout the paper, let $n\in\mathbb{Z}_{\geq 1}$. For $n=1$, the {\em affine Weyl group}  $\widehat{\Sy}_1$ of type $\widehat{A}_0$ is the trivial group. Similarly, the finite Weyl group $\Sy_1$ of type $A_0$ is the trivial group. 

For $n\geq 2$, let $\widehat{I}:=\mathbb{Z}/n\mathbb{Z}$ and $I:=\{1,\ldots, n-1\}$. By a slight abuse of notation, we will often identify $\widehat{I}$ with the set of representatives $\{0,1,\ldots, n-1\}$ and consider $I$ as a subset of $\widehat{I}$. 

For $n=2$, the {\em affine Weyl group $\widehat{\Sy}_2$} of type $\widehat{A}_1$ is generated by the simple reflections $s_0, s_1,$ subject to the relations
\begin{equation*}
  s_i^2 = 1, 
\end{equation*}
for $i\in \widehat{I}$.

For $n>2$, the {\em affine Weyl group} $\widehat{\Sy}_n$ of type $\widehat{A}_{n-1}$ is generated by the simple reflections $s_i, \; i\in \widehat{I},$ subject to the relations
\begin{equation*}
  s_i^2 = 1, \mspace{40mu} s_is_j =s_js_i \mspace{10mu}\text{if}\mspace{10mu}
  \vert i-j\vert>1, \mspace{40mu} s_is_{i+1}s_i = s_{i+1}s_is_{i+1}, 
\end{equation*}
for $i\in \widehat{I}$. 

A {\em reduced expression (rex)} for an element $w\in \widehat{\Sy}_n$ is a finite word 
$(s_{i_1}, \ldots, s_{i_m})$ of simple reflections, for some $m\in \mathbb{Z}_{\geq 0}$, such that $w=s_{i_1}\cdots s_{i_m}$ and there is no word consisting of fewer simple reflections with that property. All rexes for $w$ contain the same number of simple reflections, which is called the length of $w$ and denoted by $\ell(w)$. By definition, the rex for the neutral element is the empty word of length zero (i.e. $m=0$). 

The {\em extended} affine Weyl group $\widehat{\Sy}_n^{\mathrm{ext}}$, using the weight lattice of $\mathrm{GL}_n$, is the semidirect product 
\[
\langle \rho \rangle \ltimes \widehat{\Sy}_n,
\]
where $\langle \rho \rangle$ is an infinite cyclic group generated by $\rho$ and 
\[
\rho s_i \rho^{-1} =s_{i+1},
\]
for $i\in \widehat{I}$. In particular, when $n=1$, the extended affine Weyl group $\widehat{\Sy}_1^{\mathrm{ext}}$ is just the infinite cyclic group $\langle \rho \rangle$. 

The finite Weyl group of type $A_{n-1}$ is the symmetric group on $n$ letters, $\Sy_n$, corresponding to the subgroup of $\widehat{\Sy}_n$ generated by $s_i,\; i\in I$.

\subsection{Hecke algebras}\label{sec:hecke}
Let $q$ be a formal parameter. For $n=1$, the {\em extended affine Hecke algebra} $\eah{1}$ is simply the group algebra of $\langle \rho\rangle$ over $\mathbb{Z}[q,q^{-1}]$. 

For $n=2$, the \emph{extended affine Hecke algebra} $\eaffh$ is the $\mathbb{Z}[q,q^{-1}]$-algebra generated by $T_0, T_1$, and $\rho^{\pm 1}$, subject to the relations
\begin{gather}
  (T_i+q)(T_i-q^{-1}) = 0, \mspace{40mu} \rho\rho^{-1}=1=\rho^{-1}\rho,\mspace{40mu} \rho T_i \rho^{-1} = T_{i+1},  
    \label{eq:affHeckeRrelsn2} 
    \end{gather}
for $i\in \widehat{I}$.

For $n > 2$, the \emph{extended affine Hecke algebra} $\eaffh$ is the $\mathbb{Z}[q,q^{-1}]$-algebra generated by $T_i,\; i\in \widehat{I}$, and $\rho^{\pm 1}$, subject to the relations
\begin{gather}
  (T_i+q)(T_i-q^{-1}) = 0, \mspace{40mu} T_iT_j =T_jT_i \mspace{10mu}\text{if}\mspace{10mu}
  \vert i-j\vert>1, \mspace{40mu} T_iT_{i+1}T_i = T_{i+1}T_iT_{i+1}, 
 \label{eq:affHeckeSnrels}   \\
\rho\rho^{-1}=1=\rho^{-1}\rho,\mspace{40mu} \rho T_i \rho^{-1} = T_{i+1},  
    \label{eq:affHeckeRrels} 
    \end{gather}
for $i,j\in \widehat{I}$. Note that $T_i$ is invertible 
for every $i\in \widehat{I}$: 
\[
T_i^{-1}=T_i+q-q^{-1}.
\] 
As is well-known, $\eaffh$ is a $q$-deformation of the group 
algebra $\mathbb{Z}[\widehat{\Sy}_n^{\mathrm{ext}}]$ with the {\em standard basis} given by 
$\{\rho^m T_w\mid m\in \mathbb{Z}, w\in \widehat{\Sy}_n\}$, where $T_w:=T_{i_1} \cdots T_{i_{\ell}}$ 
for any {\em reduced expression} (rex) $s_{i_1}\cdots s_{i_{\ell}}$ of $w$.

\smallskip

Another presentation of $\eaffh$ is given in terms of the \emph{Kazhdan--Lusztig generators} $b_i:=T_i+q$, for $i\in \widehat{I}$, 
and $\rho^{\pm 1}$, subject to the relations
\begin{gather}
  b_i^{2}= [2] b_i,\mspace{40mu} \rho\rho^{-1}=1=\rho^{-1}\rho,\mspace{40mu} \rho b_i \rho^{-1} = b_{i+1}, 
    \label{eq:affHeckebrelsnistwo} 
    \end{gather}
for $n=2$, and     
\begin{gather}
  b_i^{2}= [2] b_i, \mspace{40mu} b_ib_j =b_jb_i \mspace{10mu}\text{if}\mspace{10mu}
  \vert i-j\vert>1, \mspace{40mu} b_ib_{i+1}b_i + b_{i+1} = b_{i+1}b_ib_{i+1} + b_{i}, 
 \label{eq:fdHeckebrels}   \\
\rho\rho^{-1}=1=\rho^{-1}\rho,\mspace{40mu} \rho b_i \rho^{-1} = b_{i+1}, 
    \label{eq:affHeckebrels} 
    \end{gather}
for $n>2$, where $i\in \widehat{I}$ and $[2]:=q+q^{-1}$. Note that $T_i=b_i-q$ and $T_i^{-1}=b_i-q^{-1}$, for every 
$i\in \widehat{I}$. The {\em Kazhdan--Lusztig basis} is given by $\{\rho^m b_w \mid 
m\in \mathbb{Z}, w \in \widehat{\Sy}_n \}$, where the definition of $b_w$ requires the choice of a rex for $w$ but is independent of that choice. 
\smallskip 

For $n=1$, the (non-extended) \emph{affine Hecke algebra} $\widehat{H}_1$ is the trivial one-dimensional algebra isomorphic to $\mathbb{Z}[q,q^{-1}]$. For $n\geq 2$, the (non-extended) \emph{affine Hecke algebra} $\affh$ is the subalgebra of $\eaffh$ generated by either $T_i$ or $b_i,$ for $ i\in \widehat{I}$. 

For $n=1$, the {\em finite Hecke algebra} $H_1$ is also the trivial one-dimensional algebra isomorphic to $\mathbb{Z}[q,q^{-1}]$. For $n\geq 2$, 
the \emph{finite Hecke algebra} $H_n$ is the $\mathbb{Z}[q,q^{-1}]$-subalgebra of $\affh$ generated by either $T_i$ or $b_i$, for $i\in I$.


\vspace{0.1in}

There is a third presentation of $\eaffh$, called the \emph{Bernstein presentation}. In that presentation, $\eaffh$ is defined as a twisted tensor product of $H_n$ and the algebra of Laurent polynomials in $n$ indeterminates with coefficients in $\mathbb{Z}[q,q^{-1}]$. There are several possible choices for defining the commutation relations. The one we use here has indeterminates $y_1,\dotsc , y_n$, with relations given by 
\begin{equation}
T_i^{-1}y_iT_i^{-1} = y_{i+1}, 
\end{equation}  
for $i\in I$. 

The relation between this Bernstein presentation and our first presentation of $\eaffh$ is 
given by 
\begin{align}
y_1 &= \rho T_{n-1}\dotsm T_{2}T_1 , \label{E:ybernstein}\\
y_i &= T_{i-1}^{-1}\dotsm T_2^{-1}T_1^{-1}\rho T_{n-1}\dotsm T_{i+1}T_i, \quad i=2,
\ldots, n-1. 
\end{align}

\begin{rem}
Some remarks about the various conventions in the literature are in order.
We try to follow conventions close to those in~\cite{elias2018}. 
Our presentation of the extended affine Hecke algebra in~\fullref{sec:hecke}
agrees with~\cite{elias2018}, as does the relation between the standard generators and the 
Kazhdan--Lusztig generators.  
\end{rem}

The {\em standard trace} $\epsilon\colon \eah{n}\to \mathbb{Z}[q,q^{-1}]$, which also plays an important role in this paper, is the 
$\mathbb{Z}[q,q^{-1}]$-linear map defined by 
\begin{equation}\label{eq:standardtrace}
\epsilon(\rho^r T_w)=\delta_{r,0} \delta_{w,e}, 
\end{equation}
for $r\in \mathbb{Z}$ and $w\in \widehat{\Sy}_n$, where $\delta_{-,-}$ is the Kronecker delta. The trace induces a $q$-sesquilinear form $(-,-)$ on $\eah{n}$ defined by 
\begin{equation}\label{eq:sesquilinform}
(x,y):=\epsilon(\omega(x)y),
\end{equation}
where $\omega$ is the $q$-antilinear antiinvolution on  $\eah{n}$ defined by $\omega(\rho)=\rho^{-1}$ and $\omega(T_w)=T_{w}^{-1}$, for 
$w\in \widehat{\Sy}_n$. By definition, $q$-sesquilinear means that $(-,-)$ is $\mathbb{Z}$-bilinear and satisfies $(qx,y)=q^{-1}(x,y)=(x,q^{-1}y)$, for all $x,y\in \eah{n}$. The above definitions imply 
that $\omega(b_w)=b_{w^{-1}}$ and that the Kazhdan-Lusztig basis is {\em asymptotically orthonormal} w.r.t. $(-,-)$, see 
e.g. \cite[Theorem 3.21]{e-m-t-w} for non-extended affine type A. 
\begin{thm}\label{thm:asymportho}
For all $k,l\in \mathbb{Z}$ and $u,v\in \widehat{\Sy}_n$, we have $\left(\rho^k b_u,\rho^l b_v \right)=\delta_{k,l}\left(b_u,b_v \right)$ and 
\[
\left(b_u,b_v \right) \in 
\begin{cases}
1+q\mathbb{Z}[q], & \text{if}\; u=v;\\
q\mathbb{Z}[q], &\text{else}.  
\end{cases}
\]
\end{thm}

\subsection{The embedding}\label{sec:par-hecke}
Let $1\leq k\leq n-1$. There are two unital embeddings of $\mathbb{Z}[q,q^{-1}]$-algebras $\psi_L\colon\eah{k}\to \eah{n}$ (the left embedding) and 
$\psi_R\colon \eah{n-k} \to \eah{n}$ (the right embedding), which are easily defined in terms of the Bernstein presentation:
\begin{alignat*}{5}
\psi_L \colon 
\begin{dcases}
T_i \mapsto T_i,  \quad i=1,\dotsc, k-1 ,
\\
y_i \mapsto y_i , \quad i=1,\dotsc, k,
\end{dcases}
& 
\qquad & 
\psi_R \colon  
\begin{dcases}
T_j \mapsto T_{k+j} , \quad j=1,\dotsc, n-k-1 ,
\\
y_j \mapsto y_{k+j} ,\quad j=1,\dotsc , n-k .
\end{dcases}
\end{alignat*}
The two embeddings 
give rise to a $\mathbb{Z}[q,q^{-1}]$-linear map  
\[
\psi_{k,n-k}\colon\eah{k}\otimes \eah{n-k}\to \eah{n},
\]
defined by 
\[
\psi_{k,n-k}(a \otimes b):=\psi_L(a)\psi_R(b).
\]
By definition, $\psi_{k,n-k}(a \otimes 1)=\psi_L(a)$ and $\psi_{k,n-k}(1 \otimes b)=\psi_R(b)$, where we denote the identity element of the various Hecke algebras by the same symbol $1$.

It is easy to see that this map is a homomorphism of algebras. Note that $\psi_L$ and $\psi_R$ form a {\em commuting pair of 
algebra homomorphisms}, in the sense that 
\[
\psi_L(a)\psi_R(b)=\psi_R(b)\psi_L(a)
\]
holds for all $a\in \eah{k}, b\in \eah{n-k}$, because the $y_m$ for $1\leq m\leq n$ all commute with each other and we have $T_i T_j=T_j T_i$, $T_iy_j=y_jT_i$ and $T_jy_i=y_iT_j$ for all $1\leq i\leq k-1$ and $k+1\leq j\leq n-1$.
Therefore, for all $a_1,a_2\in \eah{k}$ and $b_1,b_2\in \eah{n-k}$, we have 
\[
\psi_L(a_1)\psi_L(a_2)\psi_R(b_1)\psi_R(b_2)=\psi_L(a_1)\psi_R(b_1)\psi_L(a_2)\psi_R(b_2), 
\]
which means that  
\[
\psi_{k,n-k}(a_1a_2,b_1b_2)=\psi_{k,n-k}(a_1,b_1)\psi_{k,n-k}(a_2,b_2).
\] 

Using~\eqref{E:ybernstein} to translate to the presentation in terms of the standard generators and $\rho$ given in~\eqref{eq:affHeckeSnrels} and~\eqref{eq:affHeckeRrels}, yields 
\begin{align}
    \psi_L(\rhoL) &= y_1 T_1^{-1}T_2^{-1}\dotsm T_{k-1}^{-1}  = \rho T_{n-1}\dotsm T_k ,\label{eq:psiLrhoL}\\
     \psi_L(T_0) &= \psi_L(\rhoL^{\!\!\!\!-1}T_1\rhoL) = T_k^{-1}\dotsm T_{n-1}^{-1}T_0 T_{n-1}\dotsm T_k ,\label{eq:psiLT0}
  \intertext{and} 
    \psi_R(\rhoR) &= y_{k+1} T_{k+1}^{-1}T_{k+2}^{-1}\dotsm T_{n-k-1}^{-1} = T_k^{-1}\dotsm T_1^{-1}\rho,\label{eq:psiRrhoR} \\
    \psi_R(T_0) &= \psi_R(\rhoR^{\!\!\!\!-1}T_1\rhoR) = T_0 \dotsm T_{k-1}T_k T_{k-1}^{-1}\dotsm T_0^{-1}, \label{eq:psiRT0}
\end{align}
where we use the notation $\rhoL$ and $\rhoR$ for the twist generators of $\eah{k}$ and $\eah{n-k}$, with $\rho$ being used for the generator of $\eah{n}$. The underscript letters $\color{blue}L$ and $\color{purple}R$ are meant to signify \say{{\color{blue}left}} and \say{{\color{purple}right}} to match their position in the tensor product $\eah{k}\otimes \eah{n-k}$. 

Since $\psi_L(\rhoL)$ and $\psi_R(\rhoR)$ commute, we have 
\begin{equation}\label{eq:psirhoLrhoR}
\psi_{k,n-k}(\rhoL,\rhoR) = T_k^{-1}\dotsm T_1^{-1}\rho \rho T_{n-1}\dotsm T_k = 
\rho T_{n-1}\dotsm T_{k+1} T_{k-1}^{-1}\dotsm T_1^{-1}\rho.
\end{equation}

The image of the Kazhdan--Lusztig generators under $\psi_L$ and $\psi_R$ is 
\begin{align*}
\psi_L(b_i) &=b_i ,\quad  i = 1 , \dotsc ,k-1 ,
  \\
\psi_R(b_j) & =b_{k+j} , \quad j = 1 , \dotsc ,n-k-1 .
\end{align*}
Using~\eqref{eq:psiLT0} and \eqref{eq:psiRT0} it follows at once that  
\begin{align*}
\psi_L(b_0) &= T_k^{-1}\dotsm T_{n-1}^{-1}b_0 T_{n-1}\dotsm T_k,
\\
\psi_R(b_0) &= T_0 \dotsm T_{k-1}b_k T_{k-1}^{-1}\dotsm T_0^{-1}.
\end{align*}

\subsection{Diagrammatics}\label{sec:diagrams}
As is well known, the various Hecke algebras described above also have a diagrammatic incarnation, since all of them can be defined as a quotient of the group algebra of the corresponding braid group. In this incarnation, the standard generators correspond to crossings between two neighboring strands. In the affine case, we have to use braid diagrams on a cylinder, which we depict as a rectangle where the two vertical boundary components are identified. In this diagrammatic presentation for $\eah{n}$ we have
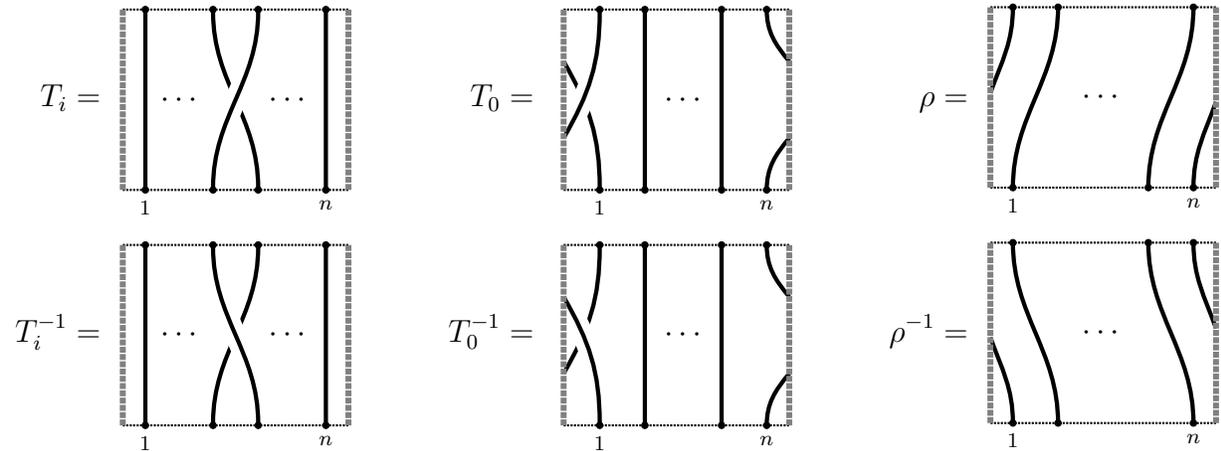
\begin{align*}
  T_i &=\ 
\xy (0,-2)*{
\begin{tikzpicture}[scale=.6]
\draw[thick,densely dotted] (0, 2) to (5, 2);
\draw[thick,densely dotted] (0,-2) to (5,-2);
\draw[black,ultra thick] (.5,-2) to (.5,2);
\draw[black,crossline,ultra thick] (3,-1.95) to[out=90,in=-90] (2,1.95);
\draw[black,crossline,ultra thick] (2,-1.95) to[out=90,in=-90] (3,1.95);
\draw[black,ultra thick] (4.5,-2) to (4.5,2); 
\draw[black,fill=black] (.5,2) circle (2.0pt);
\draw[black,fill=black] (2,2) circle (2.0pt);
\draw[black,fill=black] (3,2) circle (2.0pt);
\draw[black,fill=black] (4.5,2) circle (2.0pt);
\node at (1.3,0) {$\dotsm$};
\node at (3.7,0) {$\dotsm$};
\draw[black,fill=black] (.5,-2) circle (2.0pt)node[black,below] {\tiny  $1$};
\draw[black,fill=black] (2,-2) circle (2.0pt);
\draw[black,fill=black] (3,-2) circle (2.0pt);
\draw[black,fill=black] (4.5,-2) circle (2.0pt)node[black,below] {\tiny  $n$};
\draw[cutline] (0,-2) to (0,2);
\draw[cutline] (5,-2) to (5,2);
\end{tikzpicture}
}\endxy \mspace{40mu}
&
%
  T_0 &=\ 
\xy (0,-2)*{
\begin{tikzpicture}[scale=.6]
\draw[thick,densely dotted] (0, 2) to (5, 2);
\draw[thick,densely dotted] (0,-2) to (5,-2);
\draw[black,ultra thick] (1.8,-2) to (1.8,2);
\draw[black,ultra thick] (3.5,-2) to (3.5,2);
\draw[black,ultra thick] (4.5,-2) to[out=90,in=-130] (5,-.85);
\draw[black,ultra thick] (5,.85) to[out=130,in=-90] (4.5,2);
\draw[black,crossline,ultra thick] (.8,-1.95) to[out=90,in=-60] (0,.85);
\draw[black,crossline,ultra thick] (0,-.85) to[out=60,in=-90] (.8,1.95);
\draw[black,fill=black] (0.8,2) circle (2.0pt);
\draw[black,fill=black] (1.8,2) circle (2.0pt);
\draw[black,fill=black] (3.5,2) circle (2.0pt);
\draw[black,fill=black] (4.5,2) circle (2.0pt);
\node at (2.7,0) {$\dotsm$};
\draw[black,fill=black] (.8,-2) circle (2.0pt)node[black,below] {\tiny  $1$};
\draw[black,fill=black] (1.8,-2) circle (2.0pt);
\draw[black,fill=black] (3.5,-2) circle (2.0pt);
\draw[black,fill=black] (4.5,-2) circle (2.0pt)node[black,below] {\tiny  $n$};
\draw[cutline] (0,-2) to (0,2);
\draw[cutline] (5,-2) to (5,2);
\end{tikzpicture}
}\endxy \mspace{40mu}
&
  \rho &=\ 
\xy (0,-2)*{
\begin{tikzpicture}[scale=.6]
\draw[thick,densely dotted] (0, 2) to (5, 2);
\draw[thick,densely dotted] (0,-2) to (5,-2);
\draw[black,fill=black] (0.5,2) circle (2.0pt);
\draw[black,fill=black] (1.5,2) circle (2.0pt);
\draw[black,fill=black] (4.5,2) circle (2.0pt);
\node at (2.5,0) {$\dotsm$};
\draw[black,fill=black] (.5,-2) circle (2.0pt)node[black,below] {\tiny  $1$};
\draw[black,fill=black] (3.5,-2) circle (2.0pt);
\draw[black,fill=black] (4.5,-2) circle (2.0pt)node[black,below] {\tiny  $n$};
\draw[black,ultra thick] (.5,-2) to[out=90,in=-90] (1.5,2);
\draw[black,ultra thick] (3.5,-2) to[out=90,in=-90] (4.5,2);
\draw[black,ultra thick] (4.5,-2) to[out=90,in=-110] (5,-.1);
\draw[black,ultra thick] (0,.1) to[out=70,in=-90] (.5,2);
\draw[cutline] (0,-2) to (0,2);
\draw[cutline] (5,-2) to (5,2);
\end{tikzpicture}
}\endxy
\\
  T_i^{-1} &=\ 
\xy (0,-2)*{
\begin{tikzpicture}[scale=.6]
\draw[thick,densely dotted] (0, 2) to (5, 2);
\draw[thick,densely dotted] (0,-2) to (5,-2);
\draw[black,ultra thick] (.5,-2) to (.5,2);
\draw[black,crossline,ultra thick] (2,-1.95) to[out=90,in=-90] (3,1.95);
\draw[black,crossline,ultra thick] (3,-1.95) to[out=90,in=-90] (2,1.95);
\draw[black,ultra thick] (4.5,-2) to (4.5,2); 
\draw[black,fill=black] (.5,2) circle (2.0pt);
\draw[black,fill=black] (2,2) circle (2.0pt);
\draw[black,fill=black] (3,2) circle (2.0pt);
\draw[black,fill=black] (4.5,2) circle (2.0pt);
\node at (1.3,0) {$\dotsm$};
\node at (3.7,0) {$\dotsm$};
\draw[black,fill=black] (.5,-2) circle (2.0pt)node[black,below] {\tiny  $1$};
\draw[black,fill=black] (2,-2) circle (2.0pt);
\draw[black,fill=black] (3,-2) circle (2.0pt);
\draw[black,fill=black] (4.5,-2) circle (2.0pt)node[black,below] {\tiny  $n$};
\draw[cutline] (0,-2) to (0,2);
\draw[cutline] (5,-2) to (5,2);
\end{tikzpicture}
}\endxy
&
%
  T_0^{-1} &=\ 
\xy (0,-2)*{
\begin{tikzpicture}[scale=.6]
\draw[thick,densely dotted] (0, 2) to (5, 2);
\draw[thick,densely dotted] (0,-2) to (5,-2);
\draw[black,ultra thick] (1.8,-2) to (1.8,2);
\draw[black,ultra thick] (3.5,-2) to (3.5,2);
\draw[black,ultra thick] (4.5,-2) to[out=90,in=-130] (5,-.85);
\draw[black,ultra thick] (5,.85) to[out=130,in=-90] (4.5,2);
\draw[black,crossline,ultra thick] (0,-.85) to[out=60,in=-90] (.8,1.95);
\draw[black,crossline,ultra thick] (.8,-1.95) to[out=90,in=-60] (0,.85);
\draw[black,fill=black] (0.8,2) circle (2.0pt);
\draw[black,fill=black] (1.8,2) circle (2.0pt);
\draw[black,fill=black] (3.5,2) circle (2.0pt);
\draw[black,fill=black] (4.5,2) circle (2.0pt);
\node at (2.7,0) {$\dotsm$};
\draw[black,fill=black] (.8,-2) circle (2.0pt)node[black,below] {\tiny  $1$};
\draw[black,fill=black] (1.8,-2) circle (2.0pt);
\draw[black,fill=black] (3.5,-2) circle (2.0pt);
\draw[black,fill=black] (4.5,-2) circle (2.0pt)node[black,below] {\tiny  $n$};
\draw[cutline] (0,-2) to (0,2);
\draw[cutline] (5,-2) to (5,2);
\end{tikzpicture}
}\endxy
&
  \rho^{-1} &=\ 
\xy (0,-2)*{
\begin{tikzpicture}[scale=.6,xscale=-1]
\draw[thick,densely dotted] (0, 2) to (5, 2);
\draw[thick,densely dotted] (0,-2) to (5,-2);
\draw[black,fill=black] (0.5,2) circle (2.0pt);
\draw[black,fill=black] (1.5,2) circle (2.0pt);
\draw[black,fill=black] (4.5,2) circle (2.0pt);
\node at (2.5,0) {$\dotsm$};
\draw[black,fill=black] (.5,-2) circle (2.0pt)node[black,below] {\tiny  $n$};
\draw[black,fill=black] (3.5,-2) circle (2.0pt);
\draw[black,fill=black] (4.5,-2) circle (2.0pt)node[black,below] {\tiny  $1$};
\draw[black,ultra thick] (.5,-2) to[out=90,in=-90] (1.5,2);
\draw[black,ultra thick] (3.5,-2) to[out=90,in=-90] (4.5,2);
\draw[black,ultra thick] (4.5,-2) to[out=90,in=-110] (5,-.1);
\draw[black,ultra thick] (0,.1) to[out=70,in=-90] (.5,2);
\draw[cutline] (0,-2) to (0,2);
\draw[cutline] (5,-2) to (5,2);
\end{tikzpicture}
}\endxy
\end{align*}
where $i=1,\dotsc,n-1$. Note that we don't draw all the vertical lines and that the boundary points at the top, resp. at the bottom, of the braid diagram are cyclically ordered.
In our planar presentation we always order them from left to right starting with $1$ and ending with $n$. In our conventions, 
the product $XY$ consists of gluing the diagram of $X$ atop the one of $Y$.

\vspace{0.1in}

The embedding $\psi_{k,n-k}$ can also be described diagrammatically in a natural way. The diagram for $\psi_{k,n-k}(X\otimes Y)$ consists of placing the cyclinder with the diagram for $Y\in\eah{n-k}$ inside the cylinder with the diagram for $X\in\eah{k}$ and then projecting the outer cylinder over the inner in such a way that the boundary points of the diagram for $X$ are sent to themselves, while the boundary points of the diagram for $Y$ are sent to the points labelled $k+1,\dotsc , n$.
This is easily visualized using colored diagrams: we use \textcolor{blue}{blue} for the diagrams of the elements of $\eah{k}$ and \textcolor{purple}{purple} for the ones of $\eah{n-k}$. In this convention, the cylinder with a purple diagram goes inside the one with blue diagram.

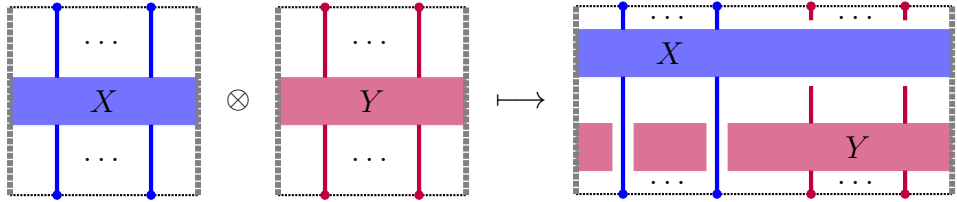
\begin{equation*}
\begin{tikzpicture}[anchorbase,scale=1.25]
\draw[blue!55,fill=blue!55] (-0.5,-0.25) to (-0.5,0.25) to (1.5,0.25) to (1.5,-0.25) to (-0.5,-0.25);
\draw[thick,densely dotted] (-0.5, 1) to (1.5, 1);
\draw[thick,densely dotted] (-0.5,-1) to (1.5,-1);
\draw[ultra thick,blue] (0,-1) to (0,-0.25);
\draw[ultra thick,blue] (0,0.25) to (0,1);
\draw[ultra thick,blue] (1,-1) to (1,-0.25);
\draw[ultra thick,blue] (1,0.25) to (1,1);
\draw[blue,fill=blue] (0,-1) circle (1.25pt);
\draw[blue,fill=blue] (0, 1) circle (1.25pt);
\draw[blue,fill=blue] (1,-1) circle (1.25pt);
\draw[blue,fill=blue] (1, 1) circle (1.25pt);
\node at (0.5,-0.625) {$\dots$};
\node at (0.5,0.625) {$\dots$};
\draw[cutline] (-0.5,-1) to (-0.5,1);
\draw[cutline] (1.5,-1) to (1.5,1);
\node at (0.5,0) {$X$};
\end{tikzpicture}	
\mspace{10mu}\otimes\mspace{10mu}
\begin{tikzpicture}[anchorbase,scale=1.25]
\draw[purple!55,fill=purple!55] (-0.5,-0.25) to (-0.5,0.25) to (1.5,0.25) to (1.5,-0.25) to (-0.5,-0.25);
\draw[thick,densely dotted] (-0.5, 1) to (1.5, 1);
\draw[thick,densely dotted] (-0.5,-1) to (1.5,-1);
\draw[ultra thick,purple] (0,-1) to (0,-0.25);
\draw[ultra thick,purple] (0,0.25) to (0,1);
\draw[ultra thick,purple] (1,-1) to (1,-0.25);
\draw[ultra thick,purple] (1,0.25) to (1,1);
\draw[purple,fill=purple] (0,-1) circle (1.25pt);
\draw[purple,fill=purple] (0, 1) circle (1.25pt);
\draw[purple,fill=purple] (1,-1) circle (1.25pt);
\draw[purple,fill=purple] (1, 1) circle (1.25pt);
\node at (0.5,-0.625) {$\dots$};
\node at (0.5,0.625) {$\dots$};
\draw[cutline] (-0.5,-1) to (-0.5,1);
\draw[cutline] (1.5,-1) to (1.5,1);
\node at (0.5,0) {$Y$};
\end{tikzpicture}
\mspace{10mu}\longmapsto\mspace{10mu}
\begin{tikzpicture}[anchorbase,scale=1.25]
\draw[blue!55,fill=blue!55] (-0.5,0.25) to (3.5,0.25) to (3.5,0.75) to (-0.5,0.75) to (-0.5,0.25);
\draw[purple!55,fill=purple!55] (-0.5,-0.25) to (3.5,-0.25) to (3.5,-0.75) to (-0.5,-0.75) to (-0.5,-0.25);
\draw[thick,densely dotted] (-0.5, 1) to (3.5, 1);
\draw[thick,densely dotted] (-0.5,-1) to (3.5,-1);
\draw[ultra thick,blue,crossline] (0,-0.96) to (0,0);
\draw[ultra thick,blue] (0,0) to (0,0.25);
\draw[ultra thick,blue] (0,0.75) to (0,1);
\draw[ultra thick,blue,crossline] (1,-0.96) to (1,0);
\draw[ultra thick,blue] (1,0) to (1,0.25);
\draw[ultra thick,blue] (1,0.75) to (1,1);
\draw[ultra thick,purple] (2,-1) to (2,-.75);
\draw[ultra thick,purple] (2,-0.25) to (2,.15);
\draw[ultra thick,purple] (2,0.85) to (2,1);
\draw[ultra thick,purple] (3,-1) to (3,-0.75);
\draw[ultra thick,purple] (3,-0.25) to (3,.15);
\draw[ultra thick,purple] (3,0.85) to (3,1);
\draw[blue,fill=blue] (0,-1) circle (1.25pt);
\draw[blue,fill=blue] (0, 1) circle (1.25pt);
\draw[blue,fill=blue] (1,-1) circle (1.25pt);
\draw[blue,fill=blue] (1, 1) circle (1.25pt);
\draw[purple,fill=purple] (2,-1) circle (1.25pt);
\draw[purple,fill=purple] (2, 1) circle (1.25pt);
\draw[purple,fill=purple] (3,-1) circle (1.25pt);
\draw[purple,fill=purple] (3, 1) circle (1.25pt);
\node at (0.5,-0.875) {$\dots$};
\node at (0.5,0.875) {$\dots$};
\node at (2.5,-0.875) {$\dots$};
\node at (2.5,0.875) {$\dots$};
\draw[cutline] (-0.5,-1) to (-0.5,1);
\draw[cutline] (3.5,-1) to (3.5,1);
\node at (0.5, 0.5) {$X$};
\node at (2.5,-0.5) {$Y$};
\end{tikzpicture}
\end{equation*}

\begin{exe}\label{ex:rhobrhor}
The diagrams for $\psi_L(\rhoL)=\psi_{k,n-k}(\rhoL\otimes 1)$ and $\psi_R(\rhoR)=\psi_{k,n-k}(1\otimes\rhoR)$:   
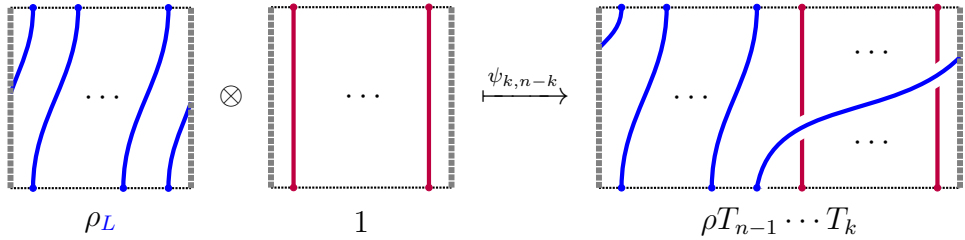
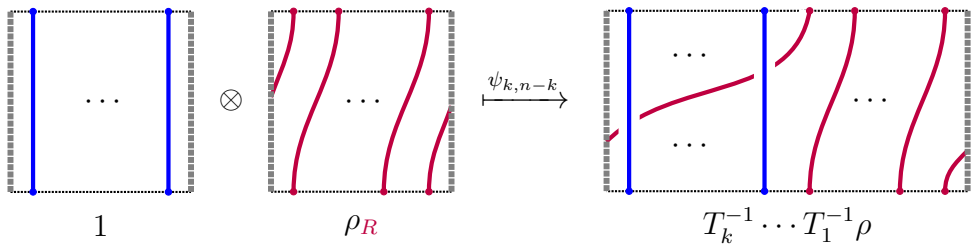
\begin{gather*}
\xy (0,-2.5)*{
\begin{tikzpicture}[scale=.6]
\draw[thick,densely dotted] (0, 2) to (4, 2);
\draw[thick,densely dotted] (0,-2) to (4,-2);
\draw[blue,ultra thick] (.5,-2) to[out=90,in=-90] (1.5,2);
\draw[blue,ultra thick] (2.5,-2) to[out=90,in=-90] (3.5,2);
\draw[blue,ultra thick] (3.5,-2) to[out=90,in=-110] (4,-.1);
\draw[blue,ultra thick] (0,.1) to[out=70,in=-90] (.5,2);
\draw[blue,fill=blue] (0.5,2) circle (2.0pt);
\draw[blue,fill=blue] (1.5,2) circle (2.0pt);
\draw[blue,fill=blue] (3.5,2) circle (2.0pt);
\node at (2.1,0) {$\dotsm$};
\draw[blue,fill=blue] (.5,-2) circle (2.0pt);
\draw[blue,fill=blue] (2.5,-2) circle (2.0pt);
\draw[blue,fill=blue] (3.5,-2) circle (2.0pt);
\draw[cutline] (0,-2) to (0,2);
\draw[cutline] (4,-2) to (4,2);
\node at (2,-2.75) {$\rhoL$};
\end{tikzpicture}
}\endxy
\mspace{10mu}\otimes\mspace{10mu}
\xy (0,-3.5)*{
\begin{tikzpicture}[scale=.6]
\draw[thick,densely dotted] (0, 2) to (4, 2);
\draw[thick,densely dotted] (0,-2) to (4,-2);
\draw[purple,ultra thick] ( .5,-2) to ( .5,2);
\draw[purple,ultra thick] (3.5,-2) to (3.5,2);
\draw[purple,fill=purple] (0.5,2) circle (2.0pt);
\draw[purple,fill=purple] (3.5,2) circle (2.0pt);
\node at (2.1,0) {$\dotsm$};
\draw[purple,fill=purple] (.5,-2) circle (2.0pt);
\draw[purple,fill=purple] (3.5,-2) circle (2.0pt);
\draw[cutline] (0,-2) to (0,2);
\draw[cutline] (4,-2) to (4,2);
\node at (2,-2.75) {$1$};
\end{tikzpicture}
}\endxy
\mspace{10mu}\xmapsto{\psi_{k,n-k}}\mspace{10mu}
\xy (0,-3.5)*{
\begin{tikzpicture}[scale=.6]
\draw[thick,densely dotted] (0, 2) to (8, 2);
\draw[thick,densely dotted] (0,-2) to (8,-2);
\draw[purple,crossline,ultra thick] (4.5,-1.95) to (4.5,1.95);
\draw[purple,crossline,ultra thick] (7.5,-1.95) to (7.5,1.95);
\node at (6.1,-1) {$\dotsm$};
\node at (6.1,1) {$\dotsm$};
\draw[purple,fill=purple] (4.5,2) circle (2.0pt);
\draw[purple,fill=purple] (7.5,2) circle (2.0pt);
\draw[purple,fill=purple] (4.5,-2) circle (2.0pt);
\draw[purple,fill=purple] (7.5,-2) circle (2.0pt);
\draw[blue,ultra thick] (.5,-2) to[out=90,in=-90] (1.5,2);
\draw[blue,ultra thick] (2.5,-2) to[out=90,in=-90] (3.5,2);
\draw[blue,crossline,ultra thick] (3.5,-2) to[out=80,in=-135] (8,.9);
\draw[blue,ultra thick] (0,1.1) to[out=45,in=-90] (.5,2);
\node at (2.1,0) {$\dotsm$};
\draw[blue,fill=blue] (0.5,2) circle (2.0pt);
\draw[blue,fill=blue] (1.5,2) circle (2.0pt);
\draw[blue,fill=blue] (3.5,2) circle (2.0pt);
\draw[blue,fill=blue] (.5,-2) circle (2.0pt);
\draw[blue,fill=blue] (2.5,-2) circle (2.0pt);
\draw[blue,fill=blue] (3.5,-2) circle (2.0pt);
\draw[cutline] (0,-2) to (0,2);
\draw[cutline] (8,-2) to (8,2);
\node at (4,-2.75) {$\rho T_{n-1}\dotsm T_k$}; 
\end{tikzpicture}
}\endxy
\end{gather*}

\begin{gather*}
\xy (0,-2.5)*{
\begin{tikzpicture}[scale=.6]
\draw[thick,densely dotted] (0, 2) to (4, 2);
\draw[thick,densely dotted] (0,-2) to (4,-2);
\draw[blue,ultra thick] ( .5,-2) to ( .5,2);
\draw[blue,ultra thick] (3.5,-2) to (3.5,2);
\draw[blue,fill=blue] (0.5,2) circle (2.0pt);
\draw[blue,fill=blue] (3.5,2) circle (2.0pt);
\node at (2.1,0) {$\dotsm$};
\draw[blue,fill=blue] (.5,-2) circle (2.0pt);
\draw[blue,fill=blue] (3.5,-2) circle (2.0pt);
\draw[cutline] (0,-2) to (0,2);
\draw[cutline] (4,-2) to (4,2);
\node at (2,-2.75) {$1$};
\end{tikzpicture}
}\endxy
\mspace{10mu}\otimes\mspace{10mu}
\xy (0,-3.5)*{
\begin{tikzpicture}[scale=.6]
\draw[thick,densely dotted] (0, 2) to (4, 2);
\draw[thick,densely dotted] (0,-2) to (4,-2);
\draw[purple,ultra thick] (.5,-2) to[out=90,in=-90] (1.5,2);
\draw[purple,ultra thick] (2.5,-2) to[out=90,in=-90] (3.5,2);
\draw[purple,ultra thick] (3.5,-2) to[out=90,in=-110] (4,-.1);
\draw[purple,ultra thick] (0,.1) to[out=70,in=-90] (.5,2);
\draw[purple,fill=purple] (0.5,2) circle (2.0pt);
\draw[purple,fill=purple] (1.5,2) circle (2.0pt);
\draw[purple,fill=purple] (3.5,2) circle (2.0pt);
\node at (2.1,0) {$\dotsm$};
\draw[purple,fill=purple] (.5,-2) circle (2.0pt);
\draw[purple,fill=purple] (2.5,-2) circle (2.0pt);
\draw[purple,fill=purple] (3.5,-2) circle (2.0pt);
\draw[cutline] (0,-2) to (0,2);
\draw[cutline] (4,-2) to (4,2);
\node at (2,-2.75) {$\rhoR$};
\end{tikzpicture}
}\endxy
\mspace{10mu}\xmapsto{\psi_{k,n-k}}\mspace{10mu}
\xy (0,-3.5)*{
\begin{tikzpicture}[scale=.6,xscale=-1,yscale=-1]
\draw[thick,densely dotted] (0, 2) to (8, 2);
\draw[thick,densely dotted] (0,-2) to (8,-2);
\draw[purple,ultra thick] (.5,-2) to[out=90,in=-90] (1.5,2);
\draw[purple,ultra thick] (2.5,-2) to[out=90,in=-90] (3.5,2);
\draw[purple,crossline,ultra thick] (3.5,-2) to[out=80,in=-135] (8,.9);
\draw[purple,ultra thick] (0,1.1) to[out=45,in=-90] (.5,2);
\node at (2.1,0) {$\dotsm$};
\draw[purple,fill=purple] (0.5,2) circle (2.0pt);
\draw[purple,fill=purple] (1.5,2) circle (2.0pt);
\draw[purple,fill=purple] (3.5,2) circle (2.0pt);
\draw[purple,fill=purple] (.5,-2) circle (2.0pt);
\draw[purple,fill=purple] (2.5,-2) circle (2.0pt);
\draw[purple,fill=purple] (3.5,-2) circle (2.0pt);
\draw[blue,crossline,ultra thick] (4.5,-1.95) to (4.5,1.95);
\draw[blue,crossline,ultra thick] (7.5,-1.95) to (7.5,1.95);
\node at (6.1,-1) {$\dotsm$};
\node at (6.1,1) {$\dotsm$};
\draw[blue,fill=blue] (4.5,2) circle (2.0pt);
\draw[blue,fill=blue] (7.5,2) circle (2.0pt);
\draw[blue,fill=blue] (4.5,-2) circle (2.0pt);
\draw[blue,fill=blue] (7.5,-2) circle (2.0pt);
\draw[cutline] (0,-2) to (0,2);
\draw[cutline] (8,-2) to (8,2);
\node at (4,2.75) {$T_k^{-1}\dotsm T_1^{-1}\rho$}; 
\end{tikzpicture}
}\endxy
\end{gather*}
\end{exe}

\begin{exe}[Diagrams on a cylinder]
To visualize this product, it might help to draw the inclusions from~\autoref{ex:rhobrhor} using diagrams on a cylinder.
For the curious reader we have also drawn the diagram for the Bernstein generator $y_1$.
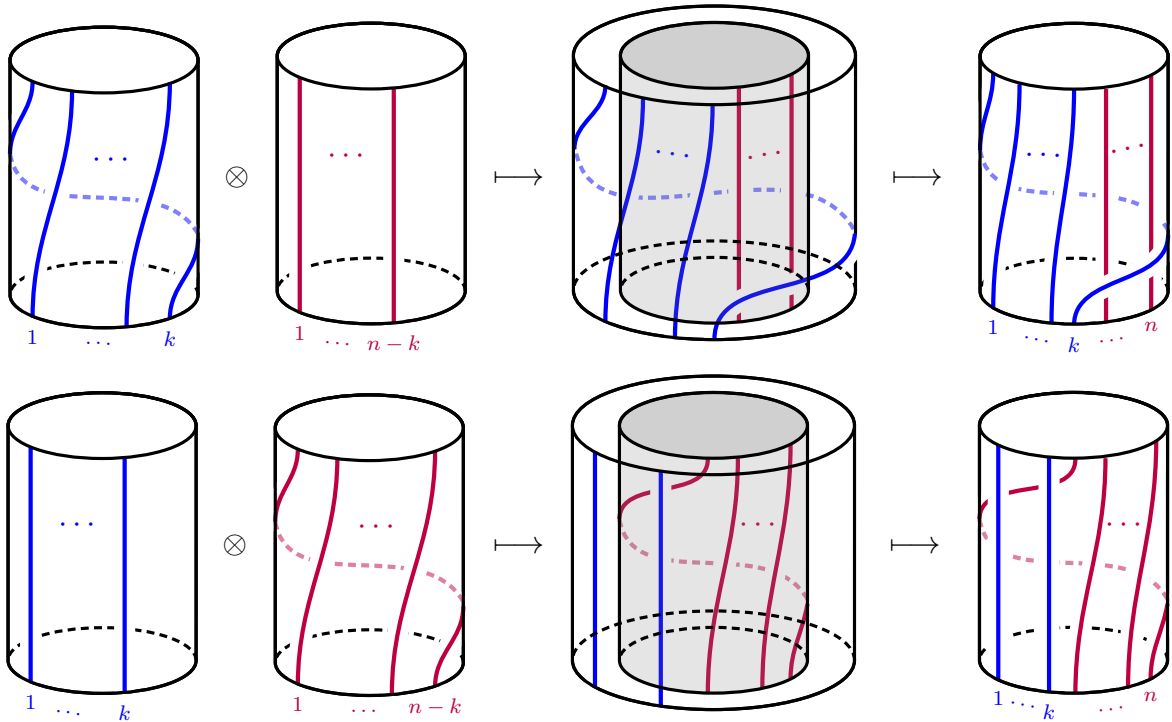
\begingroup\allowdisplaybreaks
\begin{align*}
\xy (0,-1)*{
\begin{tikzpicture}[scale=2.5]
\draw [very thick]   (180:5mm) coordinate (a) -- ++(0,-12.5mm) coordinate (b) arc (180:360:5mm and 1.75mm) coordinate (d)  -- (a -| d) coordinate (c) arc (0:180:5mm and 1.75mm);
\draw[very thick, densely dashed] (d) arc (0:180:5mm and 1.75mm);
\draw[crosslinedashed,blue,densely dashed,ultra thick,opacity=0.5]  (0.5,-0.925) to[out=120,in=-80] (-0.5,-0.5);
\draw[crossline,blue,ultra thick] (-0.38,-1.37)node[below]{\tiny $1$} to[out=90,in=-90] (-0.17,-0.16);
\draw[crossline,blue,ultra thick] (0.12,-1.42) to[out=90,in=-90] (0.35,-0.12);
\draw[crossline,blue,ultra thick] (-0.5,-0.5) to[out=90,in=-90] (-0.38,-0.117);
\draw[crossline,blue,ultra thick] (0.35,-1.375)node[below]{\tiny $k$} to[out=90,in=-90] (0.5,-0.95);
\node at (-0.02,-1.51) {\color{blue}\rotatebox{-02}{\tiny $\dotsm$}};
\node[blue] at (0.05,-0.53) {$\dotsm$};
  \draw [very thick]  (0,0) coordinate (t) circle (5mm and 1.75mm);
    \draw [very thick]   (180:5mm) coordinate (a) -- ++(0,-12.5mm) coordinate (b) arc (180:360:5mm and 1.75mm) coordinate (d)  -- (a -| d) coordinate (c) arc (0:180:5mm and 1.75mm);
 \end{tikzpicture}
}\endxy
\mspace{5mu}\otimes\mspace{11mu}
  \xy (0,-2)*{
\begin{tikzpicture}[scale=2.5]
\draw [very thick]   (180:5mm) coordinate (a) -- ++(0,-12.5mm) coordinate (b) arc (180:360:5mm and 1.75mm) coordinate (d)  -- (a -| d) coordinate (c) arc (0:180:5mm and 1.75mm);
\draw[very thick, densely dashed] (d) arc (0:180:5mm and 1.75mm);
\draw[crossline,purple,ultra thick] (-0.38,-1.37)node[below]{\tiny $1$} to[out=90,in=-90] (-0.38,-0.117);
\draw[crossline,purple,ultra thick] (0.12,-1.42)node[below]{\tiny $n-k$} to[out=90,in=-90] (0.12,-0.17);
\node at (-0.17,-1.53) {\color{purple}\rotatebox{-07}{\tiny $\dotsm$}};
\node[purple] at (-0.12,-0.53) {$\dotsm$};
\draw [very thick]  (0,0) coordinate (t) circle (5mm and 1.75mm);
\draw [very thick]   (180:5mm) coordinate (a) -- ++(0,-12.5mm) coordinate (b) arc (180:360:5mm and 1.75mm) coordinate (d)  -- (a -| d) coordinate (c) arc (0:180:5mm and 1.75mm);
 \end{tikzpicture}
}\endxy
\mspace{10mu}&\longmapsto\mspace{6mu}
  \xy (0,0)*{
\begin{tikzpicture}[scale=2.5]
\draw[crosslinedashed,blue,densely dashed,ultra thick,opacity=0.5]  (0.74,-0.925) to[out=120,in=-80] (-0.74,-0.5);
\draw[crossline,blue,ultra thick] (-0.745,-0.5) to[out=90,in=-90] (-0.58,-0.16);
\draw [very thick]  (180:7.5mm) coordinate (A)  -- ++(0,-12.5mm) coordinate (B)  arc (180:360:7.5mm and 2.625mm) coordinate (D) -- (A -| D) coordinate (C) arc (0:180:7.5mm and 2.625mm);
\draw[crossline,purple,ultra thick] (0.13,-1.42) to[out=90,in=-90] (0.13,-0.176);
\draw[crossline,purple,ultra thick] (0.41,-1.35) to[out=90,in=-90] (0.41,-0.10);
\node[purple] at (0.28,-0.53) {\rotatebox{10}{$\dotsm$}};  
\draw[crossline,blue,ultra thick] (-0.58,-1.42) to[out=90,in=-90] (-0.38,-0.23);  
\draw[crossline,blue,ultra thick] (-0.21,-1.51) to[out=90,in=-90] (-0.01,-0.27);  
\draw[crossline,blue,ultra thick] (0,-1.51) to[out=90,in=-95] (0.747,-0.95);
\node[blue] at (-0.21,-0.53) {\rotatebox{-07}{$\dotsm$}};  
 \draw [very thick,fill=gray, fill opacity=.2]  (180:5mm) coordinate (a) -- ++(0,-12.5mm) coordinate (b)  arc (180:360:5mm and 1.75mm) coordinate (d)  -- (a -| d) coordinate (c) arc (0:180:5mm and 1.75mm);
 \draw [very thick,densely dashed] (d) arc (0:180:5mm and 1.75mm);
\draw [very thick, fill=gray, fill opacity=.2]  (0,0) coordinate (t) circle (5mm and 1.75mm);
\draw [very thick] (0,0) coordinate (T) circle (7.5mm and 2.625mm);
\draw [very thick, densely dashed] (D) arc (0:180:7.5mm and 2.625mm);
\draw [very thick] (D) arc (0:-180:7.5mm and 2.625mm);
\end{tikzpicture}
}\endxy
\mspace{10mu}\longmapsto\mspace{10mu}
\xy (0,-2)*{
\begin{tikzpicture}[scale=2.5]
\draw [very thick]   (180:5mm) coordinate (a) -- ++(0,-12.5mm) coordinate (b) arc (180:360:5mm and 1.75mm) coordinate (d)  -- (a -| d) coordinate (c) arc (0:180:5mm and 1.75mm);
\draw[very thick, densely dashed] (d) arc (0:180:5mm and 1.75mm);
\draw[crosslinedashed,blue,densely dashed,ultra thick,opacity=0.5]  (0.5,-0.925) to[out=120,in=-80] (-0.5,-0.5);
\draw[crossline,purple,ultra thick] (0.17,-1.42) to[out=90,in=-90] (0.17,-0.175);
\draw[crossline,purple,ultra thick] (0.41,-1.35)node[below]{\tiny $n$} to[out=90,in=-90] (0.41,-0.10);
\draw[crossline,blue,ultra thick] (-0.43,-1.34)node[below]{\tiny $1$} to[out=90,in=-90] (-0.29,-0.15);
\draw[crossline,blue,ultra thick] (-0.12,-1.42) to[out=90,in=-90] (0,-0.17);
\draw[crossline,blue,ultra thick] (-0.5,-0.5) to[out=90,in=-90] (-0.43,-0.09);
\draw[crossline,blue,ultra thick] (0.0,-1.43)node[below]{\tiny $k$} to[out=90,in=-90] (0.5,-0.95);
\node[blue] at (-0.19,-1.51) {\rotatebox{-09}{\tiny $\dotsm$}};
\node[blue] at (-0.15,-0.53) {\rotatebox{-03}{$\dotsm$}};
\node[purple] at (0.3,-0.49) {\rotatebox{13}{$\dotsm$}};
\node[purple] at (0.21,-1.51) {\rotatebox{09}{\tiny $\dotsm$}};
  \draw [very thick]  (0,0) coordinate (t) circle (5mm and 1.75mm);
    \draw [very thick]   (180:5mm) coordinate (a) -- ++(0,-12.5mm) coordinate (b) arc (180:360:5mm and 1.75mm) coordinate (d)  -- (a -| d) coordinate (c) arc (0:180:5mm and 1.75mm);
 \end{tikzpicture}
}\endxy
  \\ 
\xy (0,-1)*{
\begin{tikzpicture}[scale=2.5]
\draw [very thick]   (180:5mm) coordinate (a) -- ++(0,-12.5mm) coordinate (b) arc (180:360:5mm and 1.75mm) coordinate (d)  -- (a -| d) coordinate (c) arc (0:180:5mm and 1.75mm);
\draw[very thick, densely dashed] (d) arc (0:180:5mm and 1.75mm);
\draw[crossline,blue,ultra thick] (-0.38,-1.37)node[below]{\tiny $1$} to[out=90,in=-90] (-0.38,-0.117);
\draw[crossline,blue,ultra thick] (0.12,-1.42)node[below]{\tiny $k$} to[out=90,in=-90] (0.12,-0.17);
\node at (-0.17,-1.53) {\color{blue}\rotatebox{-07}{\tiny $\dotsm$}};
\node[blue] at (-0.12,-0.53) {$\dotsm$};
\draw [very thick]  (0,0) coordinate (t) circle (5mm and 1.75mm);
\draw [very thick]   (180:5mm) coordinate (a) -- ++(0,-12.5mm) coordinate (b) arc (180:360:5mm and 1.75mm) coordinate (d)  -- (a -| d) coordinate (c) arc (0:180:5mm and 1.75mm);
 \end{tikzpicture}
}\endxy
\mspace{10mu}\otimes\mspace{6mu}
\xy (0,-2)*{
\begin{tikzpicture}[scale=2.5]
\draw [very thick]   (180:5mm) coordinate (a) -- ++(0,-12.5mm) coordinate (b) arc (180:360:5mm and 1.75mm) coordinate (d)  -- (a -| d) coordinate (c) arc (0:180:5mm and 1.75mm);
\draw[very thick, densely dashed] (d) arc (0:180:5mm and 1.75mm);
\draw[crosslinedashed,purple,densely dashed,ultra thick,opacity=0.5]  (0.5,-0.925) to[out=120,in=-80] (-0.5,-0.5);
\draw[crossline,purple,ultra thick] (-0.38,-1.37)node[below]{\tiny $1$} to[out=90,in=-90] (-0.17,-0.16);
\draw[crossline,purple,ultra thick] (0.12,-1.42) to[out=90,in=-90] (0.35,-0.12);
\draw[crossline,purple,ultra thick] (-0.5,-0.5) to[out=90,in=-90] (-0.38,-0.117);
\draw[crossline,purple,ultra thick] (0.35,-1.375)node[below]{\tiny $n-k$} to[out=90,in=-90] (0.5,-0.95);
\node at (-0.02,-1.51) {\color{purple}\rotatebox{-02}{\tiny $\dotsm$}};
\node[purple] at (0.05,-0.53) {$\dotsm$};
  \draw [very thick]  (0,0) coordinate (t) circle (5mm and 1.75mm);
    \draw [very thick]   (180:5mm) coordinate (a) -- ++(0,-12.5mm) coordinate (b) arc (180:360:5mm and 1.75mm) coordinate (d)  -- (a -| d) coordinate (c) arc (0:180:5mm and 1.75mm);
 \end{tikzpicture}
}\endxy
\mspace{6mu}&\longmapsto\mspace{10mu}
  \xy (0,0)*{
\begin{tikzpicture}[scale=2.5]
  \draw [very thick]  (180:7.5mm) coordinate (A)  -- ++(0,-12.5mm) coordinate (B)  arc (180:360:7.5mm and 2.625mm) coordinate (D) -- (A -| D) coordinate (C) arc (0:180:7.5mm and 2.625mm);
\draw[crosslinedashed,purple,densely dashed,ultra thick,opacity=0.5]  (0.5,-0.925) to[out=120,in=-80] (-0.5,-0.5);
\draw[crossline,purple,ultra thick] (-0.03,-1.43) to[out=90,in=-90] (0.13,-0.176);
\draw[crossline,purple,ultra thick] (0.26,-1.40) to[out=90,in=-90] (0.41,-0.10);
\draw[crossline,purple,ultra thick] (-0.5,-0.5) to[out=90,in=-90] (-0.03,-0.18);
\draw[crossline,purple,ultra thick] (0.41,-1.35) to[out=90,in=-100] (0.50,-0.95);
\node[purple] at (0.25,-0.53) {$\dotsm$};  
\draw[crossline,blue,ultra thick] (-0.63,-1.39) to[out=90,in=-90] (-0.63,-0.14);  
\draw[crossline,blue,ultra thick] (-0.28,-1.49) to[out=90,in=-90] (-0.28,-0.24);  
 \draw [very thick,fill=gray, fill opacity=.2]  (180:5mm) coordinate (a) -- ++(0,-12.5mm) coordinate (b)  arc (180:360:5mm and 1.75mm) coordinate (d)  -- (a -| d) coordinate (c) arc (0:180:5mm and 1.75mm);
 \draw [very thick,densely dashed] (d) arc (0:180:5mm and 1.75mm);
\draw [very thick, fill=gray, fill opacity=.2]  (0,0) coordinate (t) circle (5mm and 1.75mm);
\draw [very thick] (0,0) coordinate (T) circle (7.5mm and 2.625mm);
\draw [very thick, densely dashed] (D) arc (0:180:7.5mm and 2.625mm);
\draw [very thick] (D) arc (0:-180:7.5mm and 2.625mm);
\end{tikzpicture}
}\endxy
\mspace{14mu}\longmapsto\mspace{11mu}
\xy (0,-2)*{
\begin{tikzpicture}[scale=2.5]
\draw [very thick]   (180:5mm) coordinate (a) -- ++(0,-12.5mm) coordinate (b) arc (180:360:5mm and 1.75mm) coordinate (d)  -- (a -| d) coordinate (c) arc (0:180:5mm and 1.75mm);
\draw[very thick, densely dashed] (d) arc (0:180:5mm and 1.75mm);
\draw[crosslinedashed,purple,densely dashed,ultra thick,opacity=0.5]  (0.5,-0.925) to[out=120,in=-80] (-0.5,-0.5);
\draw[crossline,purple,ultra thick] (0.01,-1.43) to[out=90,in=-90] (0.17,-0.176);
\draw[crossline,purple,ultra thick] (0.29,-1.40) to[out=90,in=-90] (0.44,-0.08);
\draw[crossline,purple,ultra thick] (-0.5,-0.5) to[out=90,in=-90] (0.01,-0.18);
\draw[crossline,purple,ultra thick] (0.41,-1.35)node[below] {\tiny $n$} to[out=90,in=-100] (0.50,-0.95);
\node[purple] at (0.28,-0.53) {$\dotsm$};  
\draw[crossline,blue,ultra thick] (-0.40,-1.35)node[below] {\tiny $1$} to[out=90,in=-90] (-0.40,-0.10);  
\draw[crossline,blue,ultra thick] (-0.13,-1.42)node[below] {\tiny $k$} to[out=90,in=-90] (-0.13,-0.17);  
\node[blue] at (-0.26,-1.49) {\rotatebox{-13}{\tiny $\dotsm$}};
\node[purple] at (0.21,-1.51) {\rotatebox{09}{\tiny $\dotsm$}};
\draw [very thick]  (0,0) coordinate (t) circle (5mm and 1.75mm);
\draw [very thick]   (180:5mm) coordinate (a) -- ++(0,-12.5mm) coordinate (b) arc (180:360:5mm and 1.75mm) coordinate (d)  -- (a -| d) coordinate (c) arc (0:180:5mm and 1.75mm);
 \end{tikzpicture}
}\endxy
\end{align*}
\endgroup
The Bernstein generator $y_1$ can be drawn as 
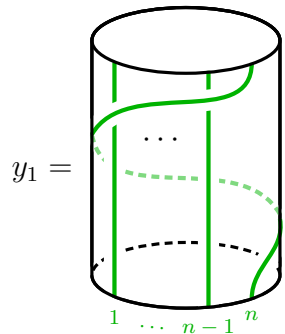
\[
y_1 =
\xy (0,-1)*{
\begin{tikzpicture}[scale=2.5]
  \draw [very thick]   (180:5mm) coordinate (a) -- ++(0,-12.5mm) coordinate (b) arc (180:360:5mm and 1.75mm) coordinate (d)  -- (a -| d) coordinate (c) arc (0:180:5mm and 1.75mm);
\draw [crosslinedashed,very thick, densely dashed] (d) arc (0:180:5mm and 1.75mm);
\draw[crosslinedashed,mygreen,densely dashed,ultra thick,opacity=0.5]  (0.5,-0.925) to[out=120,in=-80] (-0.5,-0.5);
\draw[crossline,mygreen,ultra thick] (-0.38,-1.37)node[below]{\tiny $1$} to
(-0.38,-0.117);
\draw[crossline,mygreen,ultra thick] (0.12,-1.42)node[below]{\tiny $n-1$} to
(0.12,-0.162);
\draw[crossline,mygreen,ultra thick] (-0.5,-0.5) to[out=60,in=-90] (0.35,-0.12);
\draw[crossline,mygreen,ultra thick] (0.35,-1.375)node[below]{\tiny $n$} to[out=90,in=-80] (0.5,-0.95);
\node at (-0.17,-1.53) {\color{mygreen}\rotatebox{-07}{\tiny $\dotsm$}};
\node at (-0.12,-0.53) {$\dotsm$};
\draw [very thick]  (0,0) coordinate (t) circle (5mm and 1.75mm);
\draw [very thick]   (180:5mm) coordinate (a) -- ++(0,-12.5mm) coordinate (b) arc (180:360:5mm and 1.75mm) coordinate (d)  -- (a -| d) coordinate (c) arc (0:180:5mm and 1.75mm);
 \end{tikzpicture}
}\endxy
\]
\end{exe}

Note that the equation $\psi_{k,n-k}(\rhoL\otimes 1)\psi_{k,n-k}(1\otimes\rhoR)=\psi_{k,n-k}(1\otimes\rhoR)\psi_{k,n-k}(\rhoL\otimes 1)$ 
translates to \say{\textcolor{blue}{blue} and \textcolor{purple}{purple} sliding past each other} in the corresponding compositions of the diagrams above.

%
%

\section{Rouquier-Soergel calculus}\label{s:rouquier-Soergel}

\subsection{Soergel calculus in extended affine type \texorpdfstring{$A$}{A} and Rouquier complexes}\label{sec:catreminders}

In this section, we initially revisit the definition of the diagrammatic Soergel category for the extended affine type $A$ as presented in \cite{mackaay-thiel} (see also \cite{mmv-evalfunctor} and \cite{elias2018}). 

Let $\Bbbk$ be an arbitrary field. Our conventions for graded ($\Bbbk$-linear additive) categories and shifts are as in~\cite[\S 3.1]{mmv-evalfunctor}, which match those of \cite[Section 4.1]{e-m-t-w}. Since those conventions are important, let us briefly recall them. 
Throughout this paper, graded will always mean $\mathbb{Z}$-graded. For any $t\in \mathbb{Z}$, the grading shift $\langle t\rangle$ of a graded vector space $M\cong \oplus_{a\in \mathbb{Z}} M_a$ is defined by $(M\langle t\rangle)_a:=M_{a+t}$ for all $a\in \mathbb{Z}$. This implies that
$\mathrm{hom}(M\langle r\rangle, N\langle t\rangle)$ consists of all homogeneous linear maps between $M$ and $N$ of degree $t-r$. The graded
morphism space of all morphisms from $M$ to $N$ is defined as $\mathrm{Hom}(M,N):=\oplus_{t\in \mathbb{Z}} \mathrm{hom}(M,N\langle t\rangle)$.
In the categories of this paper, the objects need not be vector spaces, in which case the grading shifts are formal, but the morphism spaces are 
always vector spaces consisting of (equivalence classes of) linear combinations of diagrams. Each diagram has some degree $s\in \mathbb{Z}$ and can therefore be seen as a morphism between two objects with shifts $M\langle r\rangle$ and $N\langle t\rangle$, such that $s=t-r$, or as a homogeneous morphism between $M$ and $N$ of degree $s=t-r$. To distinguish a non-graded category $\cC$ with 
shift and lower case morphism spaces of the form $\mathrm{hom}(M\langle r \rangle,N\langle t\rangle)$ from the associated graded category with morphism spaces of the form $\mathrm{Hom}(M,N)$, we 
denote the latter by $\cC^*$. 

Recall that the \emph{additive envelope} of a $\Bbbk$-linear category $\cC$ is the 
additive, linear category $\cC_{\oplus}$ whose objects are formal direct sums of objects in $\cC$ and whose morphisms are 
matrices of morphisms in $\cC$, such that composition is given by matrix multiplication. The 
{\em Karoubi envelope} of an additive, linear category is a formal enhancement which results in an idempotent complete category we denote $\cC_{\oplus, e}$. 
For more information on these formal constructions, see e.g.~\cite[Sections 11.2.2 -- 11.2.4]{e-m-t-w}.

\subsubsection{Soergel calculus in extended affine type A}\label{sec:soergeldiagrammatics-eaff}

The \emph{diagrammatic Bott-Samelson category} of extended type $\widehat{A}_{n-1}$, denoted 
$\BSext_n$, is the $\mathbb{Z}$-graded, $\mathbb{R}$-linear, pivotal category whose objects are grading shifts of finite words in the alphabet $\widehat{I}\cup \{\pm\}$, and whose vector spaces of morphisms are defined below in terms of generating diagrams and relations. 
Recall that a pivotal category is a monoidal category with duals satisfying certain conditions, see \cite[Definition 4.7.8]{EGNO}.
We will 
often identify that alphabet with the corresponding elementary Soergel bimodules $\{\mathrm{B}_i,i\in \widehat{I}\}\cup\{\Brhopm \}$, which is justified by~\cite[Theorem 2.10]{mackaay-thiel}. These are bimodules over a polynomial algebra $R$, defined below \eqref{eq:orslide6vertex}. Under this identification, a word in the alphabet corresponds to a tensor product of $\rB_i$ and $\rB_\rho$, which is called a {\em Bott-Samelson bimodule}. The empty word is the identity object, which is identified with $R$. The tensor product is taken over $R$, but we will often suppress $\otimes_R$ in our notation, e.g., we will simply write $\rB_1 \rB_2$ for $\rB_1\otimes_R \rB_2$. The diagrams then correspond to $R$-$R$-bimodule maps, but we will not use those maps explicitly in this paper. As explained above, the diagrams can be seen as morphisms between degree-shifted objects or as homogeneous morphisms between unshifted objects. In this definition, we take the former point of view. As explained above, the graded version of $\BSext_n$ will be denoted by $\BSextstar_n$.

The identity morphism on $i\in \widehat{I}$ is given by a vertical non-oriented strand colored by $i$, whereas the identity morphism on $+/-$ is given by a vertical upward/downward oriented black strand. As usual, we will color the unoriented strands to facilitate the reading of the diagrams. When there are too many different colors in a diagram, the colors are sometimes indicated by labels next to the strands. We say that two colors $i,j\in \widehat{I}$ are {\em adjacent} if $i \equiv j\pm 1\bmod n$ and that they are {\em distant} otherwise. 

The number of different types of diagram increases with $n$, so below we first give 
the generating diagrams for $n\geq 1$, then the additional ones for $n\geq 2$ and, finally, 
the additional ones that only show up for $n\geq 3$.

\begin{itemize}
\item For $n\geq 1$, 
\begin{equation}\label{eq:gensn1}
\xy (0,.6)*{
\tikzdiagc[yscale=0.9]{
\draw[ultra thick,black,to-] (1,.4) .. controls (1.2,-.4) and (1.8,-.4) .. (2,.4);
\node at (-1,-.75) {Degree};
\node at (1.5,-.75) {$0$};
}}\endxy
\mspace{40mu}
  \xy (0,.6)*{
\tikzdiagc[yscale=0.9]{
\draw[ultra thick,black,-to] (1,.4) .. controls (1.2,-.4) and (1.8,-.4) .. (2,.4);
\node at (1.5,-.75) {$0$};
    }}\endxy
\end{equation}
and the corresponding oriented black caps. 
\item For $n\geq 2$, with $i\in \widehat{I}$, 
\begin{equation}\label{eq:coloredgensn2}
\xy (0,0)*{
\tikzdiagc[scale=1]{
\begin{scope}[yscale=-.5,xscale=.5,shift={(5,-2)}] 
  \draw[ultra thick,blue] (-1,0)node[above]{\tiny $i$} -- (-1, 1)node[pos=0, tikzdot]{};
\end{scope}
\begin{scope}[yscale=.5,xscale=.5,shift={(8,2)}] 
   \draw[ultra thick,blue] (0,0)-- (0, 1)node[above]{\tiny $i$}; \draw[ultra thick,blue] (-1,-1) -- (0,0); \draw[ultra thick,blue] (1,-1) -- (0,0);
\end{scope}
\node at (0,0) {Degree};
\node at (2,0) {$1$};
\node at (4,0) {$-1$};
}}\endxy
\mspace{60mu}
  \xy (0,0)*{
\tikzdiagc[scale=1]{
\draw[ultra thick,myred] (.5,-.5)node[below] {\tiny $i-1$} -- (0,0);
\draw[ultra thick,blue] (-.5,.5) node[above] {\tiny $i$} -- (0,0);
\draw[ultra thick,black,-to] (-.5,-.5) -- (.5,.5);
\node at (0,-1) {$0$};
  }}\endxy  
\mspace{60mu}  
  \xy (0,0)*{
\tikzdiagc[scale=1,xscale=-1]{
\draw[ultra thick,blue] (.5,-.5)node[below] {\tiny $i$} -- (0,0);
\draw[ultra thick,myred] (-.5,.5) node[above] {\tiny $i-1$}-- (0,0);
\draw[ultra thick,black,-to] (-.5,-.5) -- (.5,.5);
\node at (0,-1) {$0$};
  }}\endxy
\end{equation}  
and the diagrams obtained from these by a rotation of $180$ degrees (which have the 
same degrees).
\item For $n\geq 3$, with $i,j,k\in \widehat{I}$ such that $\vert i-j\vert =1$ and $\vert i-k\vert >1$, 
\begin{equation}\label{eq:gensngeq3}
\xy (0,0)*{
\tikzdiagc[scale=1]{
\begin{scope}[scale=1, shift={(2,1)}] 
\draw[ultra thick,blue] (.5,-.5)  -- (-.5,.5)node[above]{\tiny $i$};
\draw[ultra thick,mygreen] (-.5,-.5) -- (.5,.5)node[above]{\tiny $k$};
\end{scope}
\begin{scope}[scale=.5,shift={(0,2)}] 
  \draw[ultra thick,myred] (0,-1) -- (0,0);\draw[ultra thick,myred] (0,0) -- (-1, 1)node[above]{\tiny $j$};\draw[ultra thick,myred] (0,0) -- (1, 1)node[above]{\tiny $j$};
  \draw[ultra thick,blue] (0,0)-- (0, 1)node[above]{\tiny $i$}; \draw[ultra thick,blue] (-1,-1) -- (0,0); \draw[ultra thick,blue] (1,-1) -- (0,0);
\end{scope}
\node at (-2,0) {Degree};
\node at (0,0) {$0$};
\node at (2,0) {$0$};
}}\endxy
\end{equation}
and the diagrams obtained from these by a rotation of $180$ degrees (which have the 
same degrees).
\end{itemize}

We read diagrams from bottom to top as morphisms, i.e., their source is at the bottom and their 
target at the top.

Diagrams can be stacked vertically (composition of morphisms) and juxtaposed horizontally (monoidal product of morphisms), while adding the degrees, and are subject to the relations 
below. The list contains all relations for all $n\geq 1$, but for $n=1$ and $n=2$ only 
the relations involving the respective generators should be considered. For $n=2$, there is an additional subtlety in \eqref{eq:forcingdumbel-i-iminus}, as indicated.  
We also assume isotopy invariance and cyclicity, meaning that closed parts of the 
diagrams can be moved around freely in the plane as long as they do not cross any other strands 
and the boundary is fixed, and all diagrams can be bent and rotated, and the bent and rotated 
versions of the relations also hold. 
\begin{itemize}
\item Relations involving only one color:

\begingroup\allowdisplaybreaks
\begin{gather}\label{eq:onecolorfirst}
\xy (0,0)*{
\tikzdiagc[scale=.4,yscale=.7]{
  \draw[ultra thick,blue] (0,-1.75)node[below]{\tiny $i$} -- (0,1.75)node[above]{\phantom{\tiny $i$}};
  \draw[ultra thick,blue] (0,0) -- (1,0)node[pos=1, tikzdot]{};  
}}\endxy
=\ 
\xy (0,0)*{
\tikzdiagc[scale=.4,yscale=.7]{
  \draw[ultra thick,blue] (0,-1.75)node[below]{\tiny $i$} -- (0,1.75)node[above]{\phantom{\tiny $i$}};
}}\endxy
\\[1ex] \label{eq:onecolorsecond}
  \xy (0,.05)*{
\tikzdiagc[scale=.4,yscale=.7]{
  \draw[ultra thick,blue] (0,-1) -- (0,1);
  \draw[ultra thick,blue] (0,1) -- (-1,2);  \draw[ultra thick,blue] (0,1) -- (1,2)node[above]{\phantom{\tiny $i$}};
  \draw[ultra thick,blue] (0,-1) -- (-1,-2)node[below]{\tiny $i$};  \draw[ultra thick,blue] (0,-1) -- (1,-2);
}}\endxy
=
  \xy (0,.05)*{
\tikzdiagc[yscale=.4,xscale=.3]{
  \draw[ultra thick,blue] (-1,0) -- (1,0);
  \draw[ultra thick,blue] (-2,-1)node[below]{\tiny $i$} -- (-1,0);  \draw[ultra thick,blue] (-2,1)node[above]{\phantom{\tiny $i$}} -- (-1,0);
  \draw[ultra thick,blue] ( 2,-1) -- ( 1,0);  \draw[ultra thick,blue] ( 2,1) -- ( 1,0);
}}\endxy
\\[1ex] \label{eq:onecolorthird}
  \xy (0,0)*{
\tikzdiagc[scale=0.9]{
\draw[ultra thick,blue] (0,0) circle  (.3);
\draw[ultra thick,blue] (0,-.8)node[below]{\tiny $i$} --  (0,-.3);
\node at (0,.3){\phantom{\tiny $i$}};
}}\endxy
  \ = 0
\\[1ex] \label{eq:onecolorfourth}
\xy (0,0)*{
  \tikzdiagc[yscale=0.5,xscale=.5]{
    \draw[ultra thick,blue] (0,-.35) -- (0,.35)node[pos=0, tikzdot]{} node[pos=1, tikzdot]{};
  \draw[ultra thick,blue] (.6,-1)node[below]{\tiny $i$}-- (.6, 1)node[above]{\phantom{\tiny $i$}}; 
}}\endxy
\ +\ 
\xy (0,0)*{
  \tikzdiagc[yscale=0.5,xscale=-.5]{
    \draw[ultra thick,blue] (0,-.35) -- (0,.35)node[pos=0, tikzdot]{} node[pos=1, tikzdot]{};
  \draw[ultra thick,blue] (.6,-1)node[below]{\tiny $i$}-- (.6, 1)node[above]{\phantom{\tiny $i$}}; 
}}\endxy\
=
2\,\  
\xy (0,0)*{
  \tikzdiagc[yscale=0.5,xscale=-.5]{
\draw[ultra thick,blue] (0,-1)node[below]{\tiny $i$} -- (0,-.4)node[pos=1, tikzdot]{};
\draw[ultra thick,blue] (0,.4) -- (0,1)node[above]{\phantom{\tiny $i$}}node[pos=0, tikzdot]{};
}}\endxy
\end{gather}
\endgroup

\item Relations involving two distant colors:
\begingroup\allowdisplaybreaks  
\begin{gather}
\label{eq:Scat-Rtwo}
  \xy (0,0)*{
  \tikzdiagc[yscale=1.7,xscale=1.1]{
    \draw[ultra thick,blue] (0,0)node[below]{\tiny $i$} ..controls (0,.25) and (.65,.25) .. (.65,.5) ..controls (.65,.75) and (0,.75) .. (0,1)node[above]{\phantom{\tiny $i$}};
\begin{scope}[shift={(.65,0)}]
    \draw[ultra thick,mygreen] (0,0)node[below]{\tiny $k$} ..controls (0,.25) and (-.65,.25) .. (-.65,.5) ..controls (-.65,.75) and (0,.75) .. (0,1)node[above]{\phantom{\tiny $i$}};
\end{scope}
}}\endxy
= \ 
\xy (0,0)*{
  \tikzdiagc[yscale=1.7,xscale=-1.1]{
\draw[ultra thick, blue] (.65,0)node[below]{\tiny $i$} -- (.65,1)node[above]{\phantom{\tiny $i$}};
\draw[ultra thick,mygreen] (0,0)node[below]{\tiny $k$} -- (0,1)node[above]{\phantom{\tiny $k$}};
  }}\endxy
\\[1ex]\label{eq:Scat-dotslide}
\xy (0,0)*{
\tikzdiagc[scale=1]{
\draw[ultra thick,blue] (.3,-.3)node[below]{\tiny $i$} -- (-.5,.5)node[pos=0, tikzdot]{};
\draw[ultra thick,mygreen] (-.5,-.5)node[below]{\tiny $k$} -- (.5,.5)node[above]{\phantom{\tiny $k$}};
}}\endxy
=
\xy (0,0)*{
\tikzdiagc[scale=1]{
\draw[ultra thick,blue] (-.5,.5) -- (-.2,.2)node[pos=1, tikzdot]{};
\draw[ultra thick,mygreen] (-.5,-.5)node[below]{\tiny $k$} -- (.5,.5)node[above]{\phantom{\tiny $k$}};
}}\endxy
\\[1ex]\label{eq:Scat-trivslide}
\xy (0,0)*{
  \tikzdiagc[yscale=0.5,xscale=.5]{
  \draw[ultra thick,blue] (0,0)-- (0, 1)node[above]{\tiny $i$}; \draw[ultra thick,blue] (-1,-1)node[below]{\tiny $i$} -- (0,0); \draw[ultra thick,blue] (1,-1) -- (0,0);
\draw[ultra thick,mygreen] (-1,0)node[left]{\tiny $k$} ..controls (-.25,.75) and (.25,.75) .. (1,0);
}}\endxy
=
\xy (0,0)*{
  \tikzdiagc[yscale=0.5,xscale=.5]{
  \draw[ultra thick,blue] (0,0)-- (0, 1)node[above]{\tiny $i$}; \draw[ultra thick,blue] (-1,-1)node[below]{\tiny $i$} -- (0,0); \draw[ultra thick,blue] (1,-1) -- (0,0);
\draw[ultra thick,mygreen] (-1,0)node[left]{\tiny $k$} ..controls (-.25,-.75) and (.25,-.75) .. (1,0);
}}\endxy
\\[1ex]\label{eq:distantdumbbells}
\xy (0,0)*{
  \tikzdiagc[yscale=0.5,xscale=.5]{
    \draw[ultra thick,mygreen] (0,-.35)node[below]{\tiny $k$} -- (0,.35)node[pos=0, tikzdot]{} node[pos=1, tikzdot]{};
  \draw[ultra thick,blue] (.6,-1)node[below]{\tiny $i$}-- (.6, 1)node[above]{\phantom{\tiny $i$}}; 
}}\endxy
-
\xy (0,0)*{
  \tikzdiagc[yscale=0.5,xscale=-.5]{
    \draw[ultra thick,mygreen] (0,-.35)node[below]{\tiny $k$} -- (0,.35)node[pos=0, tikzdot]{} node[pos=1, tikzdot]{};
  \draw[ultra thick,blue] (.6,-1)node[below]{\tiny $i$}-- (.6, 1)node[above]{\phantom{\tiny $i$}}; 
}}\endxy
\ = 0
\end{gather}
\endgroup
Note that \eqref{eq:distantdumbbells} follows from \eqref{eq:Scat-dotslide}.

\item Relations involving two adjacent colors:

\begingroup\allowdisplaybreaks  
\begin{gather}\label{eq:6vertexdot}
\xy (0,0)*{
  \tikzdiagc[yscale=0.5,xscale=.5]{
  \draw[ultra thick,myred] (0,-.75) -- (0,0)node[pos=0, tikzdot]{};\draw[ultra thick,myred] (0,0) -- (-1, 1)node[above]{\tiny $j$};\draw[ultra thick,myred] (0,0) -- (1, 1);
  \draw[ultra thick,blue] (0,0)-- (0, 1); \draw[ultra thick,blue] (-1,-1)node[below]{\tiny $i$} -- (0,0); \draw[ultra thick,blue] (1,-1) -- (0,0);
}}\endxy
\ =\ \ 
\xy (0,0)*{
  \tikzdiagc[yscale=0.5,xscale=.5]{
  \draw[ultra thick,myred] (-1,0) -- (-1, 1)node[above]{\tiny $j$}node[pos=0, tikzdot]{};\draw[ultra thick,myred] (1,0) -- (1, 1)node[pos=0, tikzdot]{};
  \draw[ultra thick,blue] (0,0)-- (0, 1); \draw[ultra thick,blue] (-1,-1)node[below]{\tiny $i$} -- (0,0); \draw[ultra thick,blue] (1,-1) -- (0,0);
}}\endxy
\ \ + \
\xy (0,0)*{
\tikzdiagc[yscale=0.6,xscale=.8]{
 \draw[ultra thick,blue] (1.5,0.1) -- (1.5, 0.4)node[pos=0, tikzdot]{};
 \draw[ultra thick,myred] (.9,.4)node[above]{\tiny $j$} .. controls (1.2,-.45) and (1.8,-.45) .. (2.1,.4);
 \draw[ultra thick,blue] ( .9,-1.2)node[below]{\tiny $i$} .. controls (1.2,-.35) and (1.8,-.35) .. (2.1,-1.2);
}}\endxy
\\[1ex]\label{eq:braidmoveB}
\xy (0,.05)*{
\tikzdiagc[scale=0.4,yscale=1]{
  \draw[ultra thick,myred] (-1,-2)node[below]{\tiny $j$} -- (-1,2)node[above]{\phantom{\tiny $j$}};
  \draw[ultra thick,myred] ( 1,-2)node[below]{\tiny $j$} -- ( 1,2);
  \draw[ultra thick,blue] (0,-2)node[below]{\tiny $i$}  --  (0,2)node[above]{\phantom{\tiny $i$}}; 
}}\endxy
=
\xy (0,.05)*{
\tikzdiagc[scale=0.4,yscale=1]{
  \draw[ultra thick,myred] (0,-1) -- (0,1);
  \draw[ultra thick,myred] (0,1) -- (-1,2);  \draw[ultra thick,myred] (0,1) -- (1,2);
  \draw[ultra thick,myred] (0,-1) -- (-1,-2)node[below]{\tiny $j$};  \draw[ultra thick,myred] (0,-1) -- (1,-2)node[below]{\tiny $j$};
  \draw[ultra thick,blue] (0,1)  --  (0,2)node[above]{\phantom{\tiny $j$}};  \draw[ultra thick,blue] (0,-1)  --  (0,-2)node[below]{\tiny $i$};
  \draw[ultra thick,blue] (0,1)  ..controls (-.95,.25) and (-.95,-.25) ..  (0,-1);
  \draw[ultra thick,blue] (0,1)  ..controls ( .95,.25) and ( .95,-.25) ..  (0,-1);
}}\endxy
-
\xy (0,.05)*{
\tikzdiagc[scale=0.4,yscale=1]{
  \draw[ultra thick,myred] (0,-.6) -- (0,.6);
  \draw[ultra thick,myred] (0,.6) .. controls (-.25,.6) and (-1,1).. (-1,2);
  \draw[ultra thick,myred] (0,.6) .. controls (.25,.6) and (1,1) .. (1,2)node[above]{\phantom{\tiny $j$}};
  \draw[ultra thick,myred] (0,-.6) .. controls (-.25,-.6) and (-1,-1).. (-1,-2)node[below]{\tiny $j$};
  \draw[ultra thick,myred] (0,-.6) .. controls (.25,-.6) and (1,-1) .. (1,-2);
  \draw[ultra thick,blue] (0,1.25)  --  (0,2)node[pos=0, tikzdot]{};  \draw[ultra thick,blue] (0,-1.25)  --  (0,-2)node[below]{\tiny $i$}node[pos=0, tikzdot]{};
 }}\endxy
\\[1ex]
\label{eq:stroman}
\xy (0,.05)*{
\tikzdiagc[scale=0.4,yscale=1]{
  \draw[ultra thick,myred] (0,-1) -- (0,1);
  \draw[ultra thick,myred] (0,1) -- (-1,2);  \draw[ultra thick,myred] (0,1) -- (1,2)node[above]{\phantom{\tiny $j$}};
  \draw[ultra thick,myred] (0,-1) -- (-1,-2)node[below]{\tiny $j$};  \draw[ultra thick,myred] (0,-1) -- (1,-2)node[below]{\tiny $j$};
  \draw[ultra thick,blue] (0,1)  --  (0,2)node[above]{\phantom{\tiny $i$}};  \draw[ultra thick,blue] (0,-1)  --  (0,-2)node[below]{\tiny $i$};
  \draw[ultra thick,blue] (0,1)  ..controls (-.95,.25) and (-.95,-.25) ..  (0,-1);
  \draw[ultra thick,blue] (0,1)  ..controls ( .95,.25) and ( .95,-.25) ..  (0,-1);
  \draw[ultra thick,blue] (-2,0)node[left]{\tiny $i$}  --  (-.75,0);  \draw[ultra thick,blue] (2,0)node[right]{\tiny $i$}  --  (.75,0);
}}\endxy
=
  \xy (0,.05)*{
\tikzdiagc[scale=0.4,yscale=1]{
  \draw[ultra thick,myred] (-1,0) -- (1,0);
  \draw[ultra thick,myred] (-2,-1)node[below]{\tiny $j$} -- (-1,0);  \draw[ultra thick,myred] (-2,1) -- (-1,0);
  \draw[ultra thick,myred] ( 2,-1)node[below]{\tiny $j$} -- ( 1,0);  \draw[ultra thick,myred] ( 2,1) -- ( 1,0);
  \draw[ultra thick,blue] (-1,0)  ..controls (-.5,-.8) and (.5,-.8) ..  (1,0);
  \draw[ultra thick,blue] (-1,0)  ..controls (-.5, .8) and (.5, .8) ..  (1,0);
  \draw[ultra thick,blue] (0,.65)  --  (0,1.8)node[above]{\phantom{\tiny $i$}};  \draw[ultra thick,blue] (0,-.65)  --  (0,-1.8)node[below]{\tiny $i$};
  \draw[ultra thick,blue] (-2,0)node[left]{\tiny $i$}  --  (-1,0);  \draw[ultra thick,blue] (2,0)node[right]{\tiny $i$}  --  (1,0);
}}\endxy
\\[1ex]\label{eq:forcingdumbel-i-iminus}
\xy (0,0)*{
  \tikzdiagc[yscale=0.5,xscale=.5]{
    \draw[ultra thick,myred] (0,-.35)node[below]{\tiny $j$} -- (0,.35)node[pos=0, tikzdot]{} node[pos=1, tikzdot]{};
  \draw[ultra thick,blue] (.6,-1)node[below]{\tiny $i$}-- (.6, 1)node[above]{\phantom{\tiny $i$}}; 
}}\endxy
-
\xy (0,0)*{
  \tikzdiagc[yscale=0.5,xscale=-.5]{
    \draw[ultra thick,myred] (0,-.35)node[below]{\tiny $j$} -- (0,.35)node[pos=0, tikzdot]{} node[pos=1, tikzdot]{};
  \draw[ultra thick,blue] (.6,-1)node[below]{\tiny $i$}-- (.6, 1)node[above]{\phantom{\tiny $i$}}; 
}}\endxy
\ =
\begin{cases}
\xy (0,.8)*{
  \tikzdiagc[yscale=0.5,xscale=-.5]{
    \draw[ultra thick,blue] (0,-.35)node[below]{\tiny $i$} -- (0,.35)node[pos=0, tikzdot]{} node[pos=1, tikzdot]{};
  \draw[ultra thick,blue] (.6,-1)node[below]{\tiny $i$}-- (.6, 1)node[above]{\phantom{\tiny $i$}}; 
}}\endxy
-
\xy (0,0)*{
  \tikzdiagc[yscale=0.5,xscale=.5]{
    \draw[ultra thick,blue] (0,-.35)node[below]{\tiny $i$} -- (0,.35)node[pos=0, tikzdot]{} node[pos=1, tikzdot]{};
  \draw[ultra thick,blue] (.6,-1)node[below]{\tiny $i$}-- (.6, 1)node[above]{\phantom{\tiny $i$}}; 
}}\endxy & \text{for}\;n=2
\\
\frac{1}{2}
\biggl(\ 
\xy (0,0)*{
  \tikzdiagc[yscale=0.5,xscale=-.5]{
    \draw[ultra thick,blue] (0,-.35)node[below]{\tiny $i$} -- (0,.35)node[pos=0, tikzdot]{} node[pos=1, tikzdot]{};
  \draw[ultra thick,blue] (.6,-1)node[below]{\tiny $i$}-- (.6, 1)node[above]{\phantom{\tiny $i$}}; 
}}\endxy
-
\xy (0,0)*{
  \tikzdiagc[yscale=0.5,xscale=.5]{
    \draw[ultra thick,blue] (0,-.35)node[below]{\tiny $i$} -- (0,.35)node[pos=0, tikzdot]{} node[pos=1, tikzdot]{};
  \draw[ultra thick,blue] (.6,-1)node[below]{\tiny $i$}-- (.6, 1)node[above]{\phantom{\tiny $i$}}; 
}}\endxy\ 
\biggr)& \text{for}\;n>2
\end{cases}
\end{gather}
\endgroup

\item Relation involving three distant colors:
\begin{equation}\label{eq:crossingslidek}
\xy (0,0)*{
  \tikzdiagc[yscale=0.5,xscale=.5]{
 \draw[ultra thick,blue] (-1,-1)node[below]{\tiny $i$} -- (1,1)node[above]{\phantom{\tiny $i$}};
 \draw[ultra thick,mygreen] (1,-1)node[below]{\tiny $j$} -- (-1,1)node[above]{\phantom{\tiny $j$}};
 \draw[ultra thick,orange] (-1,0)node[left]{\tiny $k$} ..controls (-.25,.75) and (.25,.75) .. (1,0);
}}\endxy
=
\xy (0,0)*{
  \tikzdiagc[yscale=0.5,xscale=.5]{
 \draw[ultra thick,blue] (-1,-1)node[below]{\tiny $i$} -- (1,1)node[above]{\phantom{\tiny $i$}};
 \draw[ultra thick,mygreen] (1,-1)node[below]{\tiny $j$} -- (-1,1)node[above]{\phantom{\tiny $j$}};
\draw[ultra thick,orange] (-1,0)node[left]{\tiny $k$} ..controls (-.25,-.75) and (.25,-.75) .. (1,0);
}}\endxy
\end{equation}

\item Relation involving two adjacent colors and one distant from the other two:
\begin{equation}\label{eq:sixv-dist}
\xy (0,0)*{
  \tikzdiagc[yscale=0.5,xscale=.5]{
  \draw[ultra thick,myred] (0,-1)node[below]{\tiny $j$} -- (0,0);\draw[ultra thick,myred] (0,0) -- (-1, 1)node[above]{\phantom{\tiny $j$}};\draw[ultra thick,myred] (0,0) -- (1, 1);
  \draw[ultra thick,blue] (0,0)-- (0, 1)node[above]{\phantom{\tiny $i$}}; \draw[ultra thick,blue] (-1,-1)node[below]{\tiny $i$} -- (0,0); \draw[ultra thick,blue] (1,-1) -- (0,0);
\draw[ultra thick,mygreen] (-1,0)node[left]{\tiny $k$} ..controls (-.25,.75) and (.25,.75) .. (1,0);  
}}\endxy
\ =\ \ 
\xy (0,0)*{
  \tikzdiagc[yscale=0.5,xscale=.5]{
  \draw[ultra thick,myred] (0,-1)node[below]{\tiny $j$} -- (0,0);\draw[ultra thick,myred] (0,0) -- (-1, 1)node[above]{\phantom{\tiny $j$}};\draw[ultra thick,myred] (0,0) -- (1, 1);
  \draw[ultra thick,blue] (0,0)-- (0, 1)node[above]{\phantom{\tiny $i$}}; \draw[ultra thick,blue] (-1,-1)node[below]{\tiny $i$} -- (0,0); \draw[ultra thick,blue] (1,-1) -- (0,0);
\draw[ultra thick,mygreen] (-1,0)node[left]{\tiny $k$} ..controls (-.25,-.75) and (.25,-.75) .. (1,0);  
}}\endxy
\end{equation}

 \item Relation involving three colors such that $i$ is adjacent to $j$, and $j$ is adjacent to $k$:
\begin{equation}\label{eq:relhatSlast}
\xy (0,.05)*{
\tikzdiagc[scale=0.5,yscale=1]{
  \draw[ultra thick,myred] (0,-1) -- (0,1);
  \draw[ultra thick,myred] (0,1) -- (-1,2)node[above]{\tiny $i$};
  \draw[ultra thick,myred] (0,1) -- (2,2)node[above]{\tiny $i$};
  \draw[ultra thick,myred] (0,-1) -- (-2,-2)node[below]{\tiny $i$};  
  \draw[ultra thick,myred] (0,-1) -- (1,-2)node[below]{\tiny $i$};
  \draw[ultra thick,blue] (0,1)  --  (0,2)node[above]{\tiny $j$}; 
  \draw[ultra thick,blue] (0,-1)  --  (0,-2)node[below]{\tiny $j$};
  \draw[ultra thick,blue] (0,1) -- (1,0); \draw[ultra thick,blue] (0,1) -- (-1,0);
  \draw[ultra thick,blue] (0,-1) -- (1,0); \draw[ultra thick,blue] (0,-1) -- (-1,0);
  \draw[ultra thick,blue] (-2,0)node[left]{\tiny $j$} -- (-1,0);\draw[ultra thick,blue] (2,0)node[right]{\tiny $j$} -- (1,0);
\draw[ultra thick,mygreen] (-1,0) -- (1,0);
\draw[ultra thick,mygreen] (-1,0) -- (-1,-2)node[below]{\tiny $k$};
\draw[ultra thick,mygreen] (-1,0) -- (-2,2)node[above]{\tiny $k$};   
\draw[ultra thick,mygreen] (1,0) -- (2,-2)node[below]{\tiny $k$};
\draw[ultra thick,mygreen] (1,0) -- (1,2)node[above]{\tiny $k$};   
}}\endxy
=
\xy (0,.05)*{
\tikzdiagc[scale=0.5,yscale=1]{
  \draw[ultra thick,mygreen] (0,-1) -- (0,1);
  \draw[ultra thick,mygreen] (0,1) -- (1,2)node[above]{\tiny $k$};
  \draw[ultra thick,mygreen] (0,1) -- (-2,2)node[above]{\tiny $k$};
  \draw[ultra thick,mygreen] (0,-1) -- (2,-2)node[below]{\tiny $k$};
  \draw[ultra thick,mygreen] (0,-1) -- (-1,-2)node[below]{\tiny $k$};
  \draw[ultra thick,blue] (0,1)  --  (0,2)node[above]{\tiny $j$};
  \draw[ultra thick,blue] (0,-1)  --  (0,-2)node[below]{\tiny $j$};
  \draw[ultra thick,blue] (0,1) -- (1,0);
  \draw[ultra thick,blue] (0,1) -- (-1,0);
  \draw[ultra thick,blue] (0,-1) -- (1,0);
  \draw[ultra thick,blue] (0,-1) -- (-1,0);
  \draw[ultra thick,blue] (-2,0)node[left]{\tiny $j$} -- (-1,0);
  \draw[ultra thick,blue] (2,0)node[right]{\tiny $j$} -- (1,0);
\draw[ultra thick,myred] (1,0) -- (-1,0);
\draw[ultra thick,myred] (1,0) -- (1,-2)node[below]{\tiny $i$};
\draw[ultra thick,myred] (1,0) -- (2,2)node[above]{\tiny $i$};   
\draw[ultra thick,myred] (-1,0) -- (-2,-2)node[below]{\tiny $i$};
\draw[ultra thick,myred] (-1,0) -- (-1,2)node[above]{\tiny $i$};   
}}\endxy
\end{equation}

\end{itemize}

\begin{itemize}
\item Relations involving only oriented strands:
\begingroup\allowdisplaybreaks
  \begin{gather}\label{eq:orloop}
  \xy (0,0)*{
\tikzdiagc[yscale=-0.9]{
\draw[ultra thick,black] (0,0) circle  (.65);\draw [ultra thick,black,-to] (.65,0) --(.65,0);
  }}\endxy
  \ = 1 =\ 
    \xy (0,0)*{
\tikzdiagc[yscale=-0.9]{
\draw[ultra thick,black] (0,0) circle  (.65);\draw [ultra thick,black,to-] (-.65,0) --(-.65,0);
  }}\endxy 
\\ \label{eq:orinv}
\xy (0,0)*{
\tikzdiagc[yscale=0.8]{
\draw[ultra thick,black,-to] (.5,-.75) -- (.5,.75);
\draw[ultra thick,black,to-] (-.5,-.75) -- (-.5,.75);
}}\endxy\ 
=
\xy (0,0)*{
  \tikzdiagc[yscale=0.8]{
\draw[ultra thick,black,-to] (1,.75) .. controls (1.2,-.05) and (1.8,-.05) .. (2,.75);
\draw[ultra thick,black,to-] (1,-.75) .. controls (1.2,.05) and (1.8,.05) .. (2,.-.75);
            }}\endxy 
          \mspace{80mu}
\xy (0,0)*{
\tikzdiagc[yscale=-0.8]{
\draw[ultra thick,black,-to] (.5,-.75) -- (.5,.75);
\draw[ultra thick,black,to-] (-.5,-.75) -- (-.5,.75);
}}\endxy\ 
=
\xy (0,0)*{
  \tikzdiagc[yscale=-0.8]{
\draw[ultra thick,black,-to] (1,.75) .. controls (1.2,-.05) and (1.8,-.05) .. (2,.75);
\draw[ultra thick,black,to-] (1,-.75) .. controls (1.2,.05) and (1.8,.05) .. (2,.-.75);
}}\endxy
\end{gather}
\endgroup

\item Relation involving oriented strands and two distant colored strands:
\begin{equation}\label{eq:orthru4vertex}
\xy (0,0)*{
  \tikzdiagc[yscale=0.5,xscale=.5]{
    \draw[ultra thick,blue] (-1,-1)node[below]{\tiny $i$} -- (.45,.45);
    \draw[ultra thick,myred] (.45,.45) -- (1,1)node[above]{\tiny $i-1$};
    \draw[ultra thick,mygreen] (1,-1)node[below]{\tiny $j$} -- (-.45,.45);
    \draw[ultra thick,olive] (-.45,.45) -- (-1,1)node[above]{\tiny $j-1$};
 \draw[ultra thick,black,to-] (-1,0) ..controls (-.25,.75) and (.25,.75) .. (1,0);
}}\endxy
=
\xy (0,0)*{
  \tikzdiagc[yscale=0.5,xscale=.5]{
    \draw[ultra thick,blue] (-1,-1)node[below]{\tiny $i$} -- (-.45,-.45);
    \draw[ultra thick,myred] (-.45,-.45) -- (1,1)node[above]{\tiny $i-1$};
    \draw[ultra thick,mygreen] (1,-1)node[below]{\tiny $j$} -- (.45,-.45);
    \draw[ultra thick,olive] (.45,-.45) -- (-1,1)node[above]{\tiny $j-1$};
 \draw[ultra thick,black,to-] (-1,0) ..controls (-.25,-.75) and (.25,-.75) .. (1,0);
}}\endxy
\end{equation}

 \item Relations involving oriented strands and two adjacent colored strands:
\begingroup\allowdisplaybreaks
\begin{gather}\label{eq:orReidII}
\xy (0,0)*{
  \tikzdiagc[yscale=2.1,xscale=1.1]{
\draw[ultra thick, blue] (1,0)node[below]{\tiny $i$} .. controls (1,.15) and  (.7,.24) .. (.5,.29);
\draw[ultra thick,myred] (.5,.29) .. controls (-.1,.4) and (-.1,.6) .. (.5,.71);
\draw[ultra thick, blue] (1,1)node[above]{\tiny $i$} .. controls (1,.85) and  (.7,.74) .. (.5,.71) ;
\node[myred] at (-.35,.5) {\tiny $i-1$};
\draw[ultra thick,black,to-] (0,0) ..controls (0,.35) and (1,.25) .. (1,.5) ..controls (1,.75) and (0,.65) .. (0,1);
  }}\endxy
= \ 
\xy (0,0)*{
  \tikzdiagc[yscale=2.1,xscale=1.1]{
\draw[ultra thick, blue] (1,0)node[below]{\tiny $i$} -- (1,1)node[above]{\phantom{\tiny $i$}};
    \draw[ultra thick,black,to-] (0,0) -- (0,1);
  }}\endxy
\mspace{80mu}
\xy (0,0)*{
  \tikzdiagc[yscale=2.1,xscale=-1.1]{
\draw[ultra thick,myred] (1,0)node[below]{\tiny $i-1$} .. controls (1,.15) and  (.7,.24) .. (.5,.29);
\draw[ultra thick,blue] (.5,.29) .. controls (-.1,.4) and (-.1,.6) .. (.5,.71);
\draw[ultra thick,myred] (1,1)node[above]{\tiny $i-1$} .. controls (1,.85) and  (.7,.74) .. (.5,.71) ;
\node[blue] at (-.15,.5) {\tiny $i$};
\draw[ultra thick,black,to-] (0,0) ..controls (0,.35) and (1,.25) .. (1,.5) ..controls (1,.75) and (0,.65) .. (0,1);
  }}\endxy
= \ 
\xy (0,0)*{
  \tikzdiagc[yscale=2.1,xscale=-1.1]{
\draw[ultra thick,myred] (1,0)node[below]{\tiny $i-1$} -- (1,1)node[above]{\phantom{\tiny $i-1$}};
    \draw[ultra thick,black,to-] (0,0) -- (0,1);
  }}\endxy
\\[1ex] \label{eq:dotrhuor}
\xy (0,1.2)*{
\tikzdiagc[scale=1]{
\draw[ultra thick,blue] (.3,-.3)node[below]{\tiny $i$} -- (0,0)node[pos=0, tikzdot]{};
\draw[ultra thick,myred] (-.5,.5)node[above]{\tiny $i-1$} -- (0,0);
\draw[ultra thick,black,to-] (-.5,-.5) -- (.5,.5);
}}\endxy
=
\xy (0,2.5)*{
\tikzdiagc[scale=1]{
\draw[ultra thick,myred] (-.5,.5)node[above]{\tiny $i-1$} -- (-.2,.2)node[pos=1, tikzdot]{};
\draw[ultra thick,black,to-] (-.5,-.5) -- (.5,.5);
}}\endxy
\mspace{100mu}
\xy (0,-.85)*{
\tikzdiagc[xscale=-1,yscale=-1]{
\draw[ultra thick,myred] (.3,-.3)node[above]{\tiny $i-1$} -- (0,0)node[pos=0, tikzdot]{};
\draw[ultra thick,blue] (-.5,.5)node[below]{\tiny $i$} -- (0,0);
\draw[ultra thick,black,-to] (-.5,-.5) -- (.5,.5);
}}\endxy
=
\xy (0,-2.3)*{
\tikzdiagc[xscale=-1,yscale=-1]{
\draw[ultra thick,blue] (-.5,.5)node[below]{\tiny $i$} -- (-.2,.2)node[pos=1, tikzdot]{};
\draw[ultra thick,black,-to] (-.5,-.5) -- (.5,.5);
}}\endxy
\\[1ex] \label{eq:orpitchfork}
\xy (0,0)*{
  \tikzdiagc[yscale=-0.5,xscale=.5]{
    \draw[ultra thick,myred] (0,0)-- (0,.5);
    \draw[ultra thick,myred] (-1,-1)node[above]{\tiny $i-1$} -- (0,0);
    \draw[ultra thick,myred] (1,-1) -- (0,0);
\draw[ultra thick,blue] (0,.5) -- (0,1)node[below]{\tiny $i$};
  \draw[ultra thick,black,to-] (-1,0) ..controls (-.25,.7) and (.25,.7) .. (1,0);
}}\endxy
\ \ =
\xy (0,0)*{
  \tikzdiagc[yscale=-0.5,xscale=.5]{
    \draw[ultra thick,myred] (-1,-1)node[above]{\tiny $i-1$} -- (-.45,-.45);
    \draw[ultra thick,myred] (1,-1) -- (.45,-.45);
    \draw[ultra thick,blue] (-.45,-.45)-- (0,0);
    \draw[ultra thick,blue] (.45,-.45)-- (0,0);
  \draw[ultra thick,blue] (0,0)-- (0,1)node[below]{\tiny $i$};
 \draw[ultra thick,black,to-] (-1,0) ..controls (-.25,-.75) and (.25,-.75) .. (1,0);
}}\endxy
\end{gather}
\endgroup

\item Relations involving oriented strands and three adjacent colored strands:
\begin{equation} \label{eq:orslide6vertex}
\xy (0,1)*{
\tikzdiagc[yscale=0.5,xscale=.5]{
  \draw[ultra thick,mygreen] (0,0)-- (0,.6);
  \draw[ultra thick,mygreen] (-1,-1)node[below]{\tiny $i+1$} -- (0,0);
  \draw[ultra thick,mygreen] (1,-1) -- (0,0);
  \draw[ultra thick,blue] (0,.6) -- (0,1);
  \draw[ultra thick,blue] (0,-1)node[below]{\tiny $i$} -- (0,0);
  \draw[ultra thick,blue] (0,0) -- (-.45, .45);
  \draw[ultra thick,myred] (-.45,.45) -- (-1,1)node[above]{\tiny $i-1$};
  \draw[ultra thick,blue] (0,0) -- (.45,.45);
  \draw[ultra thick,myred] (.45,.45) -- (1, 1);
\draw[ultra thick,black,to-] (-1,0) ..controls (-.25,.75) and (.25,.75) .. (1,0);  
}}\endxy
\ = 
\xy (0,0)*{
  \tikzdiagc[yscale=0.5,xscale=.5]{
    \draw[ultra thick,mygreen] (-1,-1)node[below]{\tiny $i+1$} -- (-.45,-.45);
    \draw[ultra thick,mygreen] (1,-1) -- (.45,-.45);
    \draw[ultra thick,blue] (0,-1)node[below]{\tiny $i$} -- (0,-.6);
    \draw[ultra thick,myred] (0,-.6) -- (0,0);
  \draw[ultra thick,myred] (0,0) -- (-1,1)node[above]{\tiny $i-1$};
  \draw[ultra thick,myred] (0,0) -- (1, 1);
   \draw[ultra thick,blue] (0,0) -- (0,1);
  \draw[ultra thick,blue] (-.45,-.45) -- (0,0); \draw[ultra thick,blue] (.45,-.45) -- (0,0);
\draw[ultra thick,black,to-] (-1,0) ..controls (-.25,-.75) and (.25,-.75) .. (1,0);  
}}\endxy
\mspace{120mu}
\xy (0,1)*{
  \tikzdiagc[yscale=0.5,xscale=.5]{
    \draw[ultra thick,blue] (0,0)-- (0,.6);
    \draw[ultra thick,blue] (-1,-1)node[below]{\tiny $i$} -- (0,0);
    \draw[ultra thick,blue] (1,-1) -- (0,0);
    \draw[ultra thick,myred] (0,.6) -- (0,1)node[above]{\tiny $i-1$};
  \draw[ultra thick,mygreen] (0,-1)node[below]{\tiny $i+1$} -- (0,0);
  \draw[ultra thick,mygreen] (0,0) -- (-.45, .45);  \draw[ultra thick,blue] (-.45,.45) -- (-1,1);
  \draw[ultra thick,mygreen] (0,0) -- (.45,.45);  \draw[ultra thick,blue] (.45,.45) -- (1, 1);
\draw[ultra thick,black,to-] (-1,0) ..controls (-.25,.75) and (.25,.75) .. (1,0);  
}}\endxy
\ = 
\xy (0,0)*{
  \tikzdiagc[yscale=0.5,xscale=.5]{
    \draw[ultra thick,blue] (-1,-1)node[below]{\tiny $i$} -- (-.45,-.45);
    \draw[ultra thick,blue] (1,-1) -- (.45,-.45);
  \draw[ultra thick,mygreen] (0,-1)node[below]{\tiny $i+1$} -- (0,-.6); \draw[ultra thick,blue] (0,-.6) -- (0,0);
  \draw[ultra thick,blue] (0,0) -- (-1,1);
  \draw[ultra thick,blue] (0,0) -- (1, 1);
   \draw[ultra thick,myred] (0,0) -- (0,1)node[above]{\tiny $i-1$};
  \draw[ultra thick,myred] (-.45,-.45) -- (0,0); \draw[ultra thick,myred] (.45,-.45) -- (0,0);
\draw[ultra thick,black,to-] (-1,0) ..controls (-.25,-.75) and (.25,-.75) .. (1,0);  
}}\endxy
\end{equation}
\end{itemize}

\begin{rem} Further relations that are consequences of \eqref{eq:forcingdumbel-i-iminus} 
usually have two cases too, e.g. 
\[
\xy (0,0)*{
\tikzdiagc[yscale=0.5,xscale=.5]{
\draw[ultra thick,violet] (0,-.35) -- (0,.35)node[pos=0, tikzdot]{} node[pos=1, tikzdot]{};
\draw[ultra thick,blue] (0,-1) to[out=155,in=-155] (0, 1); 
\draw[ultra thick,blue] (0,-1) to[out=25,in=-25] (0, 1); 
\draw[ultra thick,blue] (0,-1.5) to (0,-1); 
\draw[ultra thick,blue] (0,1) to (0,1.5); 
}}\endxy
= 
\begin{cases}
- 2\;\;
\xy (0,0)*{
\tikzdiagc[yscale=0.5,xscale=.5]{
\draw[ultra thick,blue] (0,-1.5) to (0,1.5); 
}}\endxy & \text{for}\; n=2\\[4ex]
\;\;\;-\;
\xy (0,0)*{
\tikzdiagc[yscale=0.5,xscale=.5]{
\draw[ultra thick,blue] (0,-1.5) to (0,1.5); 
}}\endxy & \text{for}\; n>2
\end{cases}
\]
and similarly with the two colors switched.
\end{rem}

Note that the empty word is the identity object in $\BSext_n$ and its endomorphisms are 
the closed diagrams, which by the relations above are equal to polynomials 
in the colored dumbbells 
\begin{equation*}
\xy (0,0)*{
\tikzdiagc[scale=1]{
\begin{scope}[yscale=-.5,xscale=.5] 
  \draw[ultra thick,blue] (-1,-.4) -- (-1, .4)node[pos=0, tikzdot]{} node[pos=1, tikzdot]{};
\end{scope}
}}\endxy
\end{equation*}
As each dumbbell has degree $2$, the degree of any polynomial in these dumbbells,  
as a morphism in $\BSext_n$, is twice its polynomial degree. From now on, 
we denote this polynomial algebra by $R$.  

Note further that, by relations~\eqref{eq:onecolorfourth},~\eqref{eq:forcingdumbel-i-iminus} 
and~\eqref{eq:distantdumbbells}, the morphism $\fbox{y}$, defined in~\cite{mackaay-thiel} by
\begin{equation}\label{eq:dumbbellsum}
\fbox{y} := \sum_{i=0}^{n-1} \,\xy (0,0)*{
\tikzdiagc[scale=1]{
\begin{scope}[yscale=-.5,xscale=.5] 
  \draw[ultra thick,blue] (-1,-.4) -- (-1, .4)node[pos=0, tikzdot]{} node[pos=1, tikzdot]{} 
node at (-0.5,0) {\small $i$} ;
\end{scope}
}}\endxy,
\end{equation}
is central, in the sense that it can be slid through all diagrams 
(i.e. it commutes horizontally with all morphisms), because it is  
equal to the sum of all simple roots.

\begin{defn}
The {\em extended Soergel category} $\Sext_n$ is the Karoubi envelope of the additive envelope of $\BSext_n$. 
\end{defn}

\noindent Note that in our definition shifts are already part of $\BSext_n$ and recall that the idempotent complete category 
$\Sext_n$ is Krull-Schmidt, see e.g.~\cite[Section 11.2.3]{e-m-t-w}. It therefore makes sense to consider the split Grothendieck group $[\Sext_n]_{\oplus}$, which has a natural $\mathbb{Z}[q,q^{-1}]$-module structure defined by $[\mathrm{X}\langle 1\rangle]=:q[\mathrm{X}]$, for every object $\mathrm{X}\in \Sext_n$. It also has a natural algebra structure, inherited from the monoidal structure on $\Sext_n$. 

Finally, let us recall the {\em Soergel categorification theorem}, see e.g. \cite[Theorem 2.5]{mackaay-thiel} for extended affine type $A$ (which was based on previous results by H\"arterich) 
and \cite[Theorem 11.1]{e-m-t-w} for general Coxeter type. For every $w\in \widehat{\Sy}_n$, choose a rex  
$\underline{w}=(s_{i_1},\ldots, s_{i_l})$ and define the Bott-Samelson bimodule 
\[
\mathrm{BS}(\underline{w}):=\rB_{i_1}\cdots \rB_{i_l}\in \BSext_n.
\]
There is an essentially unique indecomposable summand $\rB_w\in \Sext_n$ of $\mathrm{BS}(\underline{w})$ that is not a summand of 
$\mathrm{BS}(\underline{u})$ for any $u\prec w$ in $\widehat{\Sy}_n$, where $\prec$ is the Bruhat order. The isomorphism class of $\rB_w$ 
does not depend on the choice of rex for $w$. More generally, for any $k\in \mathbb{Z}$ 
and any $w\in \widehat{\Sy}_n$, the object $\rB_\rho^k \rB_w$ is indecomposable in $\Sext_n$ and the set of these indecomposables is complete and irredundant, meaning that they are all mutually non-isomorphic and every indecomposable in $\Sext_n$ is isomorphic to one of them up to grading shift. 
Recall further that $\eah{n}$ has a $q$-sesquilinear form $(-,-)$, see~\eqref{eq:sesquilinform}. Similarly, $[\Sext_n]_{\oplus}$ has a 
$q$-sesquilinear form $\langle -,-\rangle$, called the {\em graded Euler form}. To define it, one has to recall that the morphism spaces in $\Sextstar_n$ 
are free as left and as right graded $R$-modules and that their graded rank is finite and does not depend on whether we consider them as left or as right $R$-modules. By definition, 
the value of $\langle [X], [Y] \rangle$, for $X,Y\in \Sext_n$, is equal to the graded rank of $\mathrm{Hom}_{\Sextstar_n}(X,Y)$.
The following theorem combines Soergel's categorification theorem and Soergel's conjecture (which has been proved, see~\cite{mackaay-thiel} and references therein).
\begin{thm}\label{thm:extcategorification}
The $\mathbb{Z}[q,q^{-1}]$-linear map 
\begin{gather*}
\gamma\colon \eah{n} \to [\Sext_n]_{\oplus}\\
\gamma(\rho^k b_w)=[\rB_{\rho^k} \mathrm{B}_w] \quad (k\in \mathbb{Z}, w\in \widehat{\Sy}_n)
\end{gather*}
is an isomorphism of $\mathbb{Z}[q,q^{-1}]$-algebras intertwining the two $q$-sesquilinear forms. 
\end{thm}

\noindent This intertwining property goes under the name of {\em Soergel's hom formula} 
and will be used several times in Section~\ref{sec:excatstory}, so let us state it explicitly here: 
\begin{equation}\label{eq:Soergelhomformula}
(\rho^k b_u, \rho^l b_v )=\langle [\rB_{\rho^k} \rB_u], [\rB_{\rho^l} \rB_v]\rangle = \underline{\mathrm{rk}}\left(
\mathrm{Hom}_{\Sextstar_n}(\rB_{\rho^k} \rB_u, \rB_{\rho^l} \rB_v)\right)
\end{equation}
for all $k,l\in \mathbb{Z}$ and all $u,v\in \widehat{\Sy}_n$, where $\underline{\mathrm{rk}}$ denotes the graded rank.

\subsubsection{Rouquier complexes}\label{sec:diagrouquier-old}
We give a summary of the material in~\cite[\S 4.1]{mmv-evalfunctor} that is needed in the sequel. In particular, we introduced a diagrammatic calculus for certain morphisms between tensor products of Rouquier complexes and Bott--Samelson bimodules in finite type $A$. In the summary 
below, we recall the basic features of that calculus, generalizing it to extended affine type $A$ as well. Note that it only describes part of the morphism spaces in $K^b(\Sext_n)$, which are not yet understood in full generality. Describing the homotopy categories of Rouquier complexes in terms of generators and relations is a hard problem, even in finite type $A$, see for example~\cite{libedinsky-williamson}.

Throughout the rest of the paper, if $\cC$ a $\Bbbk$-linear, additive category, for any field $\Bbbk$, we write $C^b(\cC)$ for the category of bounded complexes in $\cC$ and $K^b(\cC)$ for its homotopy category. If $\cC$ is also monoidal, then the usual monoidal product of chain complexes equips 
$K^b(\cC)$ with a monoidal structure as well. If $\cC$ is also graded, then $K^b(\cC)$ is bigraded 
and we denote the shift inherited from $\cC$ by $\langle \cdot \rangle$ and 
the homological shift by $[\cdot ]$. 

For $n\geq 2$, the Rouquier complexes $\rT_i^{\pm 1}\in K^b(\Sext_n)$ for $i\in\widehat{I}$ are defined by 
\begin{equation}
  \Ti := \Bi\xra{
\mspace{15mu}
    \xy (0,0)*{
  \tikzdiagc[yscale=-.3,xscale=.25]{
  \draw[ultra thick,blue] (-1,0) -- (-1, 1)node[pos=0, tikzdot]{};
}}\endxy
\mspace{15mu}
} R\langle 1\rangle ,
\mspace{60mu}
\Tim  := R \langle -1 \rangle \xra{
\mspace{15mu}
    \xy (0,0)*{
  \tikzdiagc[yscale=.3,xscale=.25]{
  \draw[ultra thick,blue] (-1,0) -- (-1, 1)node[pos=0, tikzdot]{};
}}\endxy
\mspace{15mu}
}\Bi, 
\end{equation}
with $\rB_i$ placed in homological degree zero in both complexes. 

The identity morphisms of $\rT_i^{\pm 1}$ are depicted by
\[ 
  \id_{\Ti} :=
  \xy (0,-2.25)*{
\tikzdiagc[scale=.9]{
\draw[very thick,densely dotted,blue,-to] (.9,-.5)node[below]{\tiny $i$} -- (.9,.5);
            }}\endxy
\mspace{60mu}\text{and}\mspace{60mu}
\id_{\Tim} :=
\xy (0,-2.25)*{
\tikzdiagc[scale=.9]{
\draw[very thick,densely dotted,blue,to-] (.9,-.5)node[below]{\tiny $i$} -- (.9,.5);
            }}\endxy
\]

All generators in~\eqref{eq:rgens1},~\eqref{eq:rgens2} and~\eqref{eq:rgens3} below have degree zero with respect to both gradings.
\begin{itemize}[wide,labelindent=0pt]
\item The \emph{1-color generators} are the cups and caps:
\begin{equation}\label{eq:rgens1}
\xy (0,1.1)*{
\tikzdiagc[scale=.8]{
\draw[very thick,densely dotted,blue,to-] (-.5,0) to [out=270,in=180] (0,-.7) to [out=0,in=-90] (.5,0)node[above]{\tiny $i$};
}}\endxy
\mspace{35mu}
\xy (0,1.1)*{
\tikzdiagc[scale=.8,xscale=-1]{
\draw[very thick,densely dotted,blue,to-] (-.5,0) to [out=270,in=180] (0,-.7) to [out=0,in=-90] (.5,0)node[above]{\tiny $i$};
}}\endxy
\mspace{50mu}
\xy (0,-2.5)*{
\tikzdiagc[scale=.8,yscale=-1]{
\draw[very thick,densely dotted,blue,to-] (-.5,0) to [out=270,in=180] (0,-.7) to [out=0,in=-90] (.5,0)node[below]{\tiny $i$};
}}\endxy
\mspace{35mu}
\xy (0,-2.5)*{
\tikzdiagc[scale=.8,xscale=-1,yscale=-1]{
\draw[very thick,densely dotted,blue,to-] (-.5,0) to [out=270,in=180] (0,-.7) to [out=0,in=-90] (.5,0)node[below]{\tiny $i$};
}}\endxy
\end{equation}

The following lemma recalls the content of~\cite[Lemma 4.4]{mmv-evalfunctor}.
\begin{lem}\label{lem:diagsbraid}
For any $\ib\in \widehat{I}$, we have the following relations between morphisms in $K^b(\Sext_n)$:
\begingroup\allowdisplaybreaks
\begin{gather}
  \xy (0,0)*{
\tikzdiagc[scale=.8]{
\draw[very thick,densely dotted,blue] (0,0) circle  (.65);\draw [very thick,densely dotted,blue,-to] (.65,0) --(.65,0);
\node[blue] at (.65,-.65) {\tiny $i$};
}}\endxy
  \ = 1 =\ 
 \xy (0,-1)*{
\tikzdiagc[scale=.8,xscale=-1]{
\draw[very thick,densely dotted,blue] (0,0) circle  (.65);\draw [very thick,densely dotted,blue,-to] (.65,0) --(.65,0);
\node[blue] at (-.65,-.65) {\tiny $i$};
}}\endxy
\\[1ex]
\xy
(0,-2)*{
\begin{tikzpicture}[scale=1.5]
\draw[very thick,densely dotted,blue] (0,-.55)node[below]{\tiny $i$} -- (0,0);    
\draw[very thick,densely dotted,blue,-to] (0,0) to [out=90,in=180] (.25,.5) to [out=0,in=90] (.5,0);
\draw[very thick,densely dotted,blue] (.5,0) to [out=270,in=180] (.75,-.5) to [out=0,in=270] (1,0);
\draw[very thick,densely dotted,blue] (1,0) -- (1,0.75);
\end{tikzpicture}
}\endxy
=
\xy (
0,-2.5)*{\begin{tikzpicture}[scale=1.5]
	\draw[very thick,densely dotted,blue,-to] (0,-.55)node[below]{\tiny $i$} to (0,.75);
      \end{tikzpicture}
    }\endxy 
    =
\xy
(0,-2.5)*{
\begin{tikzpicture}[scale=1.5,xscale=-1]
\draw[very thick,densely dotted,blue] (0,-.55)node[below]{\tiny $i$} -- (0,0);    
\draw[very thick,densely dotted,blue,-to] (0,0) to [out=90,in=180] (.25,.5) to [out=0,in=90] (.5,0);
\draw[very thick,densely dotted,blue] (.5,0) to [out=270,in=180] (.75,-.5) to [out=0,in=270] (1,0);
\draw[very thick,densely dotted,blue] (1,0) -- (1,0.75);
\end{tikzpicture}
}\endxy
\mspace{60mu}
\xy
(0,-2.5)*{
\begin{tikzpicture}[scale=1.5]
\draw[very thick,densely dotted,blue] (0,-.55)node[below]{\tiny $i$} -- (0,0);    
\draw[very thick,densely dotted,blue,to-] (0,0) to [out=90,in=180] (.25,.5) to [out=0,in=90] (.5,0);
\draw[very thick,densely dotted,blue] (.5,0) to [out=270,in=180] (.75,-.5) to [out=0,in=270] (1,0);
\draw[very thick,densely dotted,blue] (1,0) -- (1,0.75);
\end{tikzpicture}
}\endxy
=
\xy (
0,-2.5)*{\begin{tikzpicture}[scale=1.5]
	\draw[very thick,densely dotted,blue,to-] (0,-.55)node[below]{\tiny $i$} to (0,.75);
      \end{tikzpicture}
    }\endxy 
    =
\xy
(0,-2.5)*{
\begin{tikzpicture}[scale=1.5,xscale=-1]
\draw[very thick,densely dotted,blue] (0,-.55)node[below]{\tiny $i$} -- (0,0);    
\draw[very thick,densely dotted,blue,to-] (0,0) to [out=90,in=180] (.25,.5) to [out=0,in=90] (.5,0);
\draw[very thick,densely dotted,blue] (.5,0) to [out=270,in=180] (.75,-.5) to [out=0,in=270] (1,0);
\draw[very thick,densely dotted,blue] (1,0) -- (1,0.75);
\end{tikzpicture}
}\endxy
\\[1ex]
\label{eq:invertidt}
\xy (0,-1)*{
\tikzdiagc[yscale=0.9]{
\draw[very thick,densely dotted,blue,-to] (.5,-.75)node[below]{\tiny $i$} -- (.5,.75);
\draw[very thick,densely dotted,blue,to-] (-.5,-.75)node[below]{\tiny $i$} -- (-.5,.75);
}}\endxy\ 
=
\xy (0,-2)*{
  \tikzdiagc[yscale=0.9]{
\draw[very thick,densely dotted,blue,-to] (1,.75) .. controls (1.2,-.05) and (1.8,-.05) .. (2,.75);
\draw[very thick,densely dotted,blue,to-] (1,-.75)node[below]{\tiny $i$} .. controls (1.2,.05) and (1.8,.05) .. (2,.-.75);
            }}\endxy 
  \mspace{90mu}
\xy (0,-2)*{
\tikzdiagc[yscale=.9,xscale=-1]{
\draw[very thick,densely dotted,blue,-to] (.5,-.75)node[below]{\tiny $i$} -- (.5,.75);
\draw[very thick,densely dotted,blue,to-] (-.5,-.75)node[below]{\tiny $i$} -- (-.5,.75);
}}\endxy\ 
=
\xy (0,-2)*{
  \tikzdiagc[yscale=.9,xscale=-1]{
\draw[very thick,densely dotted,blue,-to] (1,.75) .. controls (1.2,-.05) and (1.8,-.05) .. (2,.75);
\draw[very thick,densely dotted,blue,to-] (1,-.75)node[below]{\tiny $i$} .. controls (1.2,.05) and (1.8,.05) .. (2,.-.75);
}}\endxy
\end{gather}
\endgroup
\end{lem}
For the diagrams and relations given in the rest of this section there are also similar diagrams and relations involving inverses of Rouquier complexes. 

For $n>2$, there are additional diagrams and relations.  
\item The \emph{six-valent vertices} for $\ib$ and $\jr$ adjacent: 
\begin{equation}\label{eq:rgens2}
\xy (0,0)*{
  \tikzdiagc[scale=.55]{
\draw[very thick,densely dotted,blue,-to] (-1,-1)node[below]{\tiny \gi} -- (-.5,-.5);\draw[very thick,densely dotted,blue] (-.5,-.5) -- (0,0);
\draw[very thick,densely dotted,blue,-to] ( 1,-1)node[below]{\tiny \gi} -- ( .5,-.5);\draw[very thick,densely dotted,blue] ( .5,-.5) -- (0,0);
\draw[very thick,densely dotted,myred,-to] (0,-1)node[below]{\tiny $j$} -- (0,-.5);\draw[very thick,densely dotted,myred] (0,-.5) -- (0,0);
\draw[very thick,densely dotted,myred,-to] (0,0) -- (-1,1);
\draw[very thick,densely dotted,myred,-to] (0,0) -- ( 1,1);
\draw[very thick,densely dotted,blue,-to] (0,0) -- (0,1);
}}\endxy 
\mspace{60mu}
\xy (0,0)*{
  \tikzdiagc[scale=.55,xscale=1]{
\draw[ultra thick,myred] (0,-1)node[below]{\tiny $j$} -- (0,0);
\draw[very thick,densely dotted,myred,to-] (-.5,.5) -- (-1,1);
\draw[very thick,densely dotted,myred] (0,0) -- (-.5,.5);
\draw[very thick,densely dotted,myred,-to] (0,0) -- ( 1,1);
\draw[ultra thick,blue] (0,0) -- (0,1);
\draw[very thick,densely dotted,blue,-to] (-1,-1)node[below]{\tiny \gi} -- (-.5,-.5);\draw[very thick,densely dotted,blue] (-.5,-.5) -- (0,0);
\draw[very thick,densely dotted,blue,to-] ( 1,-1)node[below]{\tiny \gi} -- (0,0);
}}\endxy
\mspace{60mu}
\xy (0,0)*{
  \tikzdiagc[scale=.6,yscale=1,xscale=-1]{
\draw[very thick,densely dotted,red,to-] (0,-1)node[below]{\tiny $j$} -- (0,0);
\draw[ultra thick,red] (0,0) -- (-1,1);
\draw[very thick,densely dotted,red] (0,0) -- (.5,.5);
\draw[very thick,densely dotted,red,-to] (1,1) -- (.5,.5);
\draw[very thick,densely dotted,blue] (0,0) -- (0,.5);
\draw[very thick,densely dotted,blue,to-] (0,.5) -- (0,1);
\draw[very thick,densely dotted,blue,to-] (-1,-1)node[below]{\tiny \gi} -- (0,0);
\draw[ultra thick,blue] ( 1,-1)node[below]{\tiny \gi} -- (0,0);
}}\endxy 
\end{equation}
Note that a solid strand corresponds to the identity on a Soergel bimodule, whereas a dashed strand corresponds to the identify morphism on a Rouquier complex, and all morphisms are defined in the homotopy category. 
These satisfy
\[
\xy (0,.05)*{
\tikzdiagc[scale=.45,yscale=1]{
  \draw[very thick,densely dotted,myred,-to] (-1,-2)node[below]{\tiny $j$} -- (-1,2)node[above]{\phantom{\tiny $i+1$}};
  \draw[very thick,densely dotted,myred,-to] ( 1,-2)node[below]{\tiny $j$} -- ( 1,2);
  \draw[very thick,densely dotted,blue,-to] (0,-2)node[below]{\tiny \gi}  --  (0,2); 
}}\endxy
=
\xy (0,.05)*{
\tikzdiagc[scale=.45,yscale=1]{
  \draw[very thick,densely dotted,myred] (0,-1) -- (0,1);
  \draw[very thick,densely dotted,myred,-to] (0,1) -- (-1,2)node[above]{\phantom{\tiny $j$}};
  \draw[very thick,densely dotted,myred,-to] (0,1) -- (1,2);
  \draw[very thick,densely dotted,myred] (0,-1) -- (-1,-2)node[below]{\tiny $j$};
  \draw[very thick,densely dotted,myred] (0,-1) -- (1,-2)node[below]{\tiny $j$}  ;
  \draw[very thick,densely dotted,blue,-to] (0,1)  --  (0,2);  \draw[very thick,densely dotted,blue] (0,-1)  --  (0,-2)node[below]{\tiny \gi};
  \draw[very thick,densely dotted,blue] (0,1)  ..controls (-.95,.25) and (-.95,-.25) ..  (0,-1);
  \draw[very thick,densely dotted,blue] (0,1)  ..controls ( .95,.25) and ( .95,-.25) ..  (0,-1);
}}\endxy
 \mspace{45mu}
\xy (0,-2.5)*{
\tikzdiagc[scale=.45,yscale=1,xscale=-1]{
  \draw[very thick,densely dotted,myred,-to] (-1,-2)node[below]{\tiny $j$} -- (-1,2);
  \draw[very thick,densely dotted,myred,to-] ( 1,-2)node[below]{\tiny $j$} -- ( 1,2);
  \draw[ultra thick,blue] (0,-2)node[below]{\tiny \gi}  --  (0,2); 
}}\endxy
\!\!=\,
\xy (0,-2.5)*{
\tikzdiagc[scale=0.45,xscale=-1]{
  \draw[ultra thick,myred] (0,-1) -- (0,1);
  \draw[very thick,densely dotted,myred,-to] (0,1) -- (-1,2);
  \draw[very thick,densely dotted,myred] (0,1) -- (1,2);
  \draw[very thick,densely dotted,myred] (0,-1) -- (-1,-2)node[below]{\tiny $j$};
  \draw[very thick,densely dotted,myred,-to] (0,-1) -- (1,-2)node[below]{\tiny $j$};
  \draw[ultra thick,blue] (0,1)  --  (0,2);
  \draw[ultra thick,blue] (0,-1)  --  (0,-2)node[below]{\tiny \gi};
  \draw[very thick,densely dotted,blue] (0,1)  ..controls (-.95,.25) and (-.95,-.25) ..  (0,-1);
  \draw[very thick,densely dotted,blue] (0,1)  ..controls ( .95,.25) and ( .95,-.25) ..  (0,-1);
}}\endxy
\mspace{45mu}
\xy (0,-2.5)*{
\tikzdiagc[scale=.45,yscale=-1,xscale=-1]{
  \draw[ultra thick,red] (-1,-2) -- (-1,2)node[below]{\tiny $j$};
  \draw[very thick,densely dotted,red,-to] ( 1,-2) -- ( 1,2)node[below]{\tiny $j$};
  \draw[very thick,densely dotted,blue,-to] (0,-2)  --  (0,2)node[below]{\tiny \gi}; 
}}\endxy
=
\xy (0,-2.5)*{
\tikzdiagc[scale=0.45,yscale=-1,xscale=-1]{
  \draw[very thick,densely dotted,red,-to] (0,-1) -- (0,.2);
  \draw[very thick,densely dotted,red] (0,.2) -- (0,1);
  \draw[ultra thick,red] (0,1) -- (-1,2)node[below]{\tiny $j$};
  \draw[very thick,densely dotted,red,-to] (0,1) -- (1,2)node[below]{\tiny $j$};
  \draw[ultra thick,red] (0,-1) -- (-1,-2);
  \draw[very thick,densely dotted,red] (0,-1) -- (.45,-1.45);
  \draw[very thick,densely dotted,red,to-] (.45,-1.45) -- (1,-2);
  \draw[very thick,densely dotted,blue,-to] (0,1)  --  (0,2)node[below]{\tiny \gi};
  \draw[very thick,densely dotted,blue] (0,-1)  --  (0,-1.5);
  \draw[very thick,densely dotted,blue,-to] (0,-2) --  (0,-1.5);
  \draw[very thick,densely dotted,blue] (0,1)  ..controls (-.95,.25) and (-.95,-.25) ..  (0,-1);
  \draw[very thick,densely dotted,blue,-to] (-.71,.15) -- (-.71,.18);
  \draw[ultra thick,blue] (0,1)  ..controls ( .95,.25) and ( .95,-.25) ..  (0,-1);
}}\endxy
\mspace{45mu}
\xy (0,-2.0)*{
  \tikzdiagc[scale=.5,yscale=-1,xscale=-1]{
\draw[ultra thick,myred] (0,-1) -- (0,0);
\draw[very thick,densely dotted,myred,to-] (-.5,.5) -- (-1,1)node[below]{\tiny $j$};
\draw[very thick,densely dotted,myred] (0,0) -- (-.5,.5);
\draw[very thick,densely dotted,myred,-to] (0,0) -- ( 1,1)node[below]{\tiny $j$};
\draw[ultra thick,blue] (0,0) -- (0,.75)node[pos=1, tikzdot]{}node[below]{\tiny \gi};
\draw[very thick,densely dotted,blue,-to] (-1,-1) -- (-.5,-.5);\draw[very thick,densely dotted,blue] (-.5,-.5) -- (0,0);
\draw[very thick,densely dotted,blue,to-] ( 1,-1) -- (0,0);
}}\endxy 
\!\! =\
\xy (0,-2.5)*{
  \tikzdiagc[scale=.5,yscale=-1,xscale=-1]{
\draw[ultra thick,myred] (0,-.5) -- (0,-1)node[pos=0, tikzdot]{};
\draw[very thick,densely dotted,myred,-to] (-1,1)node[below]{\tiny $t$}  to [out=-90,in=-180] (0,.12) to [out=0,in=-90] (1,1);
\draw[very thick,densely dotted,blue,-to] (-1,-1) to [out=90,in=180] (0,-.12) to  [out=0,in=90] (1,-1);
}}\endxy 
\]

\item The \emph{crossings} for $\ib$ and $\ko$ distant:
\begin{equation}\label{eq:rgens3}
\xy (0,0)*{
  \tikzdiagc[scale=.45]{
\draw[very thick,densely dotted,blue,-to] (-1,-1)node[below]{\tiny \gi} -- (1,1);
\draw[very thick,densely dotted,orange,-to] (1,-1)node[below]{\tiny $k$} -- (-1,1);
}}\endxy
\mspace{60mu}
\xy (0,0)*{
  \tikzdiagc[scale=.5,yscale=1,xscale=-1]{
\draw[very thick,densely dotted,orange,-to] (-1,-1)node[below]{\tiny $k$} -- (1,1);
\draw[ultra thick,blue] ( 1,-1)node[below]{\tiny \gi} -- (-1,1);
}}\endxy 
\end{equation}
These satisfy
\[
\xy (0,-2)*{
\tikzdiagc[yscale=1.7,xscale=1.1]{
 \draw[very thick,densely dotted,blue,-to] (0,0)node[below]{\tiny \gi} ..controls (0,.25) and (.65,.25) .. (.65,.5) ..controls (.65,.75) and (0,.75) .. (0,1);
\begin{scope}[shift={(.65,0)}]
\draw[very thick,densely dotted,orange,-to] (0,0)node[below]{\tiny $k$} ..controls (0,.25) and (-.65,.25) .. (-.65,.5) ..controls (-.65,.75) and (0,.75) .. (0,1);
\end{scope}
}}\endxy
= \ 
\xy (0,-2)*{
  \tikzdiagc[yscale=1.7,xscale=1.1]{
\draw[very thick,densely dotted,orange,-to] (.65,0)node[below]{\tiny $k$} -- (.65,1);
\draw[very thick,densely dotted,blue,-to] (0,0)node[below]{\tiny \gi} -- (0,1);
  }}\endxy
\mspace{60mu}
  \xy (0,-2)*{
  \tikzdiagc[yscale=1.7,xscale=1.1]{
    \draw[ultra thick,blue] (0,0)node[below]{\tiny \gi} ..controls (0,.25) and (.65,.25) .. (.65,.5) ..controls (.65,.75) and (0,.75) .. (0,1);
\begin{scope}[shift={(.65,0)}]
    \draw[very thick,densely dotted,orange,-to] (0,0)node[below]{\tiny $k$} ..controls (0,.25) and (-.65,.25) .. (-.65,.5) ..controls (-.65,.75) and (0,.75) .. (0,1);
\end{scope}
}}\endxy
= \ 
\xy (0,-2)*{
  \tikzdiagc[yscale=1.7,xscale=-1.1]{
\draw[ultra thick, blue] (.65,0)node[below]{\tiny \gi} -- (.65,1);
\draw[very thick,densely dotted,orange,-to] (0,0)node[below]{\tiny $k$} -- (0,1);
  }}\endxy
\]

\end{itemize}

\subsection{More diagrammatic shortcuts for Rouquier complexes}\label{sec:rouquier-new}
We now develop some new diagrammatics to handle morphisms between tensor products of Soergel bimodules and Rouquier complexes. 

For any $a,b\in \mathbb{Z}$ define 
\[
\rT_{[a,b]}   
 :=
 \begin{cases}
\rT_{a}\rT_{a+1}\rT_{a+2}\dots\rT_{b}
& \text{if $a<b$},
 \\[1ex]
\rT_{a}     
& \text{if $a=b$},
 \\[1ex]
\rT_{a}\rT_{a-1}\rT_{a-2}\dots\rT_{b}
& \text{if $a>b$},
 \end{cases}
\]
where the indices on the right-hand side can be shifted modulo $n$ so that they belong to $\widehat{I}=\{0,1,\ldots, n-1\}$. For example, we have $\rT_{[-2,2]}=\rT_{n-2}\rT_{n-1}\rT_0\rT_1\rT_2$. Of course, the above notation is not unique, since $\rT_{[a,b]}=\rT_{[a+kn,b+kn]}$ for any $a,b,k\in \mathbb{Z}$, but we trust that this does not cause any confusion. Note that in general $\rT_{[a,b]}\ne \rT_{[a+kn,b+ln]}$ if $k\ne l$, e.g., $\rT_{[0,1]}=\rT_0\rT_1$ while $\rT_{[n,1]}=\rT_0 \rT_{n-1}\cdots \rT_1$, which are different when $n>2$. Observe that it is important that the indices $a,b$ in $\rT_{[a,b]}$ be integers and not residue classes modulo $n$, because that allows for products of $\rT_i$ whose length is greater than $n$, e.g., $\rT_{[0,2n-1]}=\rT_0\rT_1\cdots \rT_{n-1}\rT_0\rT_1\cdots \rT_{n-1}$. 

Using the same conventions, we introduce the identity morphism on $\rT_{[a,b]}$ by  
\[
 \xy (0,-1.5)*{
\tikzdiagc[yscale=0.8,baseline={([yshift=-.8ex]current bounding box.center)}]{
\draw[thick,double,densely dotted,OliveGreen,-to] (-1.15,-.5)node[below] {\tiny $[a,b]$} -- (-1.15, .5);
 }}\endxy     
 :=
 \begin{cases}
\xy (0,-1)*{
\tikzdiagc[yscale=0.8,baseline={([yshift=-.8ex]current bounding box.center)}]{
\draw[thick,densely dotted,mygreen,-to] (0,-.5)node[below]{\tiny $\textcolor{white}{1\!}a\textcolor{white}{\!1}$} -- (0,.5);
\draw[thick,densely dotted,orange,-to] (.7,-.5)node[below]{\tiny $a\!+\!1$} -- (.7,.5);
\draw[thick,densely dotted,myred,-to] (1.4,-.5)node[below]{\tiny $a\!+\!2$} -- (1.4,.5);
\node[black] at (2.2,-0.05) {\small $\dots$};
\draw[thick,densely dotted,OliveGreen,-to] (3.0,-.5)node[below]{\tiny $b$} -- (3.0,.5);
 }}\endxy     
& \text{if $a<b$},
 \\[1ex]
 \xy (0,-1)*{
\tikzdiagc[yscale=0.8,baseline={([yshift=-.8ex]current bounding box.center)}]{
\draw[thick,densely dotted,mygreen,-to] (0,-.5)node[below]{\tiny $\textcolor{white}{1\!}a\textcolor{white}{\!1}$} -- (0,.5);
 }}\endxy     
& \text{if $a=b$},
\\[1ex]
\xy (0,-1)*{
\tikzdiagc[yscale=0.8,baseline={([yshift=-.8ex]current bounding box.center)}]{
\draw[thick,densely dotted,mygreen,-to] (0,-.5)node[below]{\tiny $\textcolor{white}{1\!}a\textcolor{white}{\!1}$} -- (0,.5);
\draw[thick,densely dotted,orange,-to] (.7,-.5)node[below]{\tiny $a\!-\!1$} -- (.7,.5);
\draw[thick,densely dotted,myred,-to] (1.4,-.5)node[below]{\tiny $a\!-\!2$} -- (1.4,.5);
\node[black] at (2.2,-0.05) {\small $\dots$};
\draw[thick,densely dotted,OliveGreen,-to] (3.0,-.5)node[below]{\tiny $b$} -- (3.0,.5);
 }}\endxy     
& \text{if $a>b$},
 \end{cases}
\]

The unit and counit of the adjunction between $\rT_{[a,b]}$ and its inverse are drawn as a cup and a cap, respectively, and these satisfy the usual isotopy and invertibility 
relations (cf.~\cite[Lemma 4.4, Lemma 4.15]{mmv-evalfunctor}).
Cups and caps have degree zero, as do the generators in \eqref{eq:rnew1}, \eqref{eq:rnew2}, \eqref{eq:rnew3}, \eqref{eq:rnew4}, \eqref{eq:rnew5}, \eqref{eq:sixv-comm} and~\eqref{eq:sixv-comm2}.

\begin{lem}\label{l:rouquierbasic}
For any $a,b\in \mathbb{Z}$, we have the following relations between morphisms in $K^b(\eS_n)$: 
\begingroup\allowdisplaybreaks
\begin{gather} 
\label{eq:orcableloop}
  \xy (0,-2)*{
\tikzdiagc[scale=.8]{
\draw[thick,double,densely dotted,OliveGreen] (0,0) circle  (.65);\draw[thick,double,densely dotted,OliveGreen,-to] (.65,0) --(.65,0);
\node[OliveGreen] at (.90,-.75) {\tiny $[a,b]$};
}}\endxy
  \! = 1 =\ 
 \xy (0,-2)*{
\tikzdiagc[scale=.8,xscale=-1]{
\draw[thick,double,densely dotted,OliveGreen] (0,0) circle  (.65);\draw[thick,double,densely dotted,OliveGreen,-to] (.65,0) --(.65,0);
\node[OliveGreen] at (-.90,-.75) {\tiny $[a,b]$};
}}\endxy
\\[1ex] 
\xy
(0,-2)*{
\begin{tikzpicture}[scale=1.5]
\draw[thick,double,densely dotted,OliveGreen] (0,-.55)node[below]{\tiny $[a,b]$} -- (0,0);    
\draw[thick,double,densely dotted,OliveGreen,-to] (0,0) to [out=90,in=180] (.25,.5) to [out=0,in=90] (.5,0);
\draw[thick,double,densely dotted,OliveGreen] (.5,0) to [out=270,in=180] (.75,-.5) to [out=0,in=270] (1,0);
\draw[thick,double,densely dotted,OliveGreen] (1,0) -- (1,0.75);
\end{tikzpicture}
}\endxy
=
\xy (
0,-2.5)*{\begin{tikzpicture}[scale=1.5]
	\draw[thick,double,densely dotted,OliveGreen,-to] (0,-.55)node[below]{\tiny $[a,b]$} to (0,.75);
      \end{tikzpicture}
    }\endxy 
    =
 \xy
(0,-2.5)*{
\begin{tikzpicture}[scale=1.5,xscale=-1]
\draw[thick,double,densely dotted,OliveGreen] (0,-.55)node[below]{\tiny $[a,b]$} -- (0,0);    
\draw[thick,double,densely dotted,OliveGreen,-to] (0,0) to [out=90,in=180] (.25,.5) to [out=0,in=90] (.5,0);
\draw[thick,double,densely dotted,OliveGreen] (.5,0) to [out=270,in=180] (.75,-.5) to [out=0,in=270] (1,0);
\draw[thick,double,densely dotted,OliveGreen] (1,0) -- (1,0.75);
\end{tikzpicture}
}\endxy
\mspace{60mu}
\xy
(0,-2.5)*{
\begin{tikzpicture}[scale=1.5]
\draw[thick,double,densely dotted,OliveGreen] (0,-.55)node[below]{\tiny $[a,b]$} -- (0,0);    
\draw[thick,double,densely dotted,OliveGreen,to-] (0,0) to [out=90,in=180] (.25,.5) to [out=0,in=90] (.5,0);
\draw[thick,double,densely dotted,OliveGreen] (.5,0) to [out=270,in=180] (.75,-.5) to [out=0,in=270] (1,0);
\draw[thick,double,densely dotted,OliveGreen] (1,0) -- (1,0.75);
\end{tikzpicture}
}\endxy
=
\xy (
0,-2.5)*{\begin{tikzpicture}[scale=1.5]
	\draw[thick,double,densely dotted,OliveGreen,to-] (0,-.55)node[below]{\tiny $[a,b]$} to (0,.75);
      \end{tikzpicture}
    }\endxy 
    =
\xy
(0,-2.5)*{
\begin{tikzpicture}[scale=1.5,xscale=-1]
\draw[thick,double,densely dotted,OliveGreen] (0,-.55)node[below]{\tiny $[a,b]$} -- (0,0);    
\draw[thick,double,densely dotted,OliveGreen,to-] (0,0) to [out=90,in=180] (.25,.5) to [out=0,in=90] (.5,0);
\draw[thick,double,densely dotted,OliveGreen] (.5,0) to [out=270,in=180] (.75,-.5) to [out=0,in=270] (1,0);
\draw[thick,double,densely dotted,OliveGreen] (1,0) -- (1,0.75);
\end{tikzpicture}
}\endxy
\\[1ex] \label{eq:invcomm}
\xy (0,-1)*{
\tikzdiagc[yscale=0.9]{
\draw[thick,double,densely dotted,OliveGreen,-to] (.5,-.75)node[below]{\tiny $[a,b]$} -- (.5,.75)node[above]{\phantom{\tiny $[a,b]$}};
\draw[thick,double,densely dotted,OliveGreen,to-] (-.5,-.75)node[below]{\tiny $[a,b]$} -- (-.5,.75);
}}\endxy  
=
\xy (0,-2)*{
  \tikzdiagc[yscale=0.9]{
\draw[thick,double,densely dotted,OliveGreen,-to] (1,.75) .. controls (1.2,-.05) and (1.8,-.05) .. (2,.75)node[above]{\phantom{\tiny $[a,b]$}};
\draw[thick,double,densely dotted,OliveGreen,to-] (1,-.75)node[below]{\tiny $[a,b]$} .. controls (1.2,.05) and (1.8,.05) .. (2,.-.75);
}}\endxy 
\mspace{120mu}
\xy (0,-2)*{
\tikzdiagc[yscale=.9,xscale=-1]{
\draw[thick,double,densely dotted,OliveGreen,-to] (.5,-.75)node[below]{\tiny $[a,b]$} -- (.5,.75);
\draw[thick,double,densely dotted,OliveGreen,to-] (-.5,-.75)node[below]{\tiny $[a,b]$} -- (-.5,.75);
}}\endxy 
=
\xy (0,-2)*{
  \tikzdiagc[yscale=.9,xscale=-1]{
\draw[thick,double,densely dotted,OliveGreen,-to] (1,.75) .. controls (1.2,-.05) and (1.8,-.05) .. (2,.75);
\draw[thick,double,densely dotted,OliveGreen,to-] (1,-.75)node[below]{\tiny $[a,b]$} .. controls (1.2,.05) and (1.8,.05) .. (2,.-.75);
}}\endxy
\end{gather}
\endgroup
\end{lem}

It is useful to introduce {\em mergers and splitters}
\begin{equation}\label{eq:rnew1}
\xy (0,0)*{
\tikzdiagc[yscale=0.9]{
\draw[thick,double,densely dotted,PineGreen,-to] (0,0) -- (0,.75)node[above]{\tiny $[a,c]$};
\draw[thick,double,densely dotted,OliveGreen] (-.5,-.75)node[below]{\tiny $[a,b]$} -- (0,0);
\draw[thick,double,densely dotted,OliveGreen,-to] (-.5,-.75) -- (-.2,-.3);
\draw[thick,double,densely dotted,mygreen] (.5,-.75)node[below]{\tiny $[b\!+\!1,c]$} -- (0,0);
\draw[thick,double,densely dotted,mygreen,-to] (.5,-.75) -- (.2,-.3);
}}\endxy  
\mspace{70mu}
\xy (0,0)*{
\tikzdiagc[yscale=0.9,yscale=-1]{
\draw[thick,double,densely dotted,PineGreen] (0,0) -- (0,.75)node[below]{\tiny $[a,c]$};
\draw[thick,double,densely dotted,PineGreen,to-] (0,.25) -- (0,.5);
\draw[thick,double,densely dotted,OliveGreen,to-] (-.5,-.75)node[above]{\tiny $[a,b]$} -- (0,0);
\draw[thick,double,densely dotted,mygreen,to-] (.5,-.75)node[above]{\tiny $[b\!+\!1,c]$} -- (0,0);
}}\endxy  
\end{equation}
for $c>b\geq a$, realizing the equalities of chain complexes $\rT_{[a,b]}\rT_{[b+1,c]}=\rT_{[a,c]}$ (and analogously for their inverses). By definition, they satisfy 
\begin{equation}\label{eq:dumbcomm}
\xy (0,0)*{
\tikzdiagc[yscale=0.9]{
\draw[thick,double,densely dotted,PineGreen] (0,0) -- (0,.75);
\draw[thick,double,densely dotted,PineGreen,-to] (0,0) -- (0,.45);
\node[PineGreen] at (-.75,.375) {\tiny $[a,c]$};
\draw[thick,double,densely dotted,OliveGreen] (-.5,-.75)node[below]{\tiny $[a,b]$} -- (0,0);
\draw[thick,double,densely dotted,OliveGreen,-to] (-.5,-.75) -- (-.2,-.3);
\draw[thick,double,densely dotted,mygreen] (.5,-.75)node[below]{\tiny $[b\!+\!1,c]$} -- (0,0);
\draw[thick,double,densely dotted,mygreen,-to] (.5,-.75) -- (.2,-.3);
\draw[thick,double,densely dotted,OliveGreen,-to] (0,.75) -- (-.5,1.5)node[above]{\tiny $[a,b]$};
\draw[thick,double,densely dotted,mygreen,-to] (0,.75) -- (.5,1.5)node[above]{\tiny $[b\!+\!1,c]$};
}}\endxy  
\mspace{-10mu}=\mspace{-5mu}
\xy (0,0)*{
\tikzdiagc[yscale=0.9]{
\draw[thick,double,densely dotted,OliveGreen,-to] (-.5,-.75)node[below]{\tiny $[a,b]$} -- (-.5,1.5)node[above]{\phantom{\tiny $[a,b]$}};
\draw[thick,double,densely dotted,mygreen,-to] (.5,-.75)node[below]{\tiny $[b\!+\!1,c]$} -- (.5,1.5)node[above]{\phantom{\tiny $[b\!+\!1,c]$}};
}}\endxy  
\mspace{60mu}
\xy (0,0)*{
\tikzdiagc[yscale=0.9]{
\draw[thick,double,densely dotted,PineGreen] (0,-.75)node[below]{\tiny $[a,c]$} -- (0,0);
\draw[thick,double,densely dotted,PineGreen,-to] (0,-.75) -- (0,-.3);
\draw[thick,double,densely dotted,PineGreen,-to] (0,.75) -- (0,1.5)node[above]{\tiny $[a,c]$};
\draw[thick,double,densely dotted,mygreen] (0,0) to[out=0,in=0] (0,.75);
\draw[thick,double,densely dotted,mygreen,-to] (.22,.44) -- (.22,.46);
\draw[thick,double,densely dotted,OliveGreen] (0,0) to[out=180,in=180] (0,.75);
\draw[thick,double,densely dotted,OliveGreen,-to] (-.22,.44) -- (-.22,.46);
\node[OliveGreen] at (-.7,.6) {\tiny $[a,b]$};
\node[mygreen] at (.8,.6) {\tiny $[b\!+\!1,c]$};
}}\endxy  
=\mspace{-5mu}
\xy (0,0)*{
\tikzdiagc[yscale=0.9]{
\draw[thick,double,densely dotted,PineGreen,-to] (-.5,-.75)node[below]{\tiny $[a,c]$} -- (-.5,1.5)node[above]{\phantom{\tiny $[a,c]$}};
}}\endxy  
\end{equation}
Of course, similar mergers and splitters can be defined for $a\geq b > c$, where $[a,c]$ is subdivided into $[a,b]$ and $[b-1,c]$. 
Further, one can use cups and caps to define generators like
\[
\xy (0,0)*{
\tikzdiagc[yscale=0.9,yscale=-1]{
\draw[thick,double,densely dotted,mygreen] (0,0) -- (0,.75)node[below]{\tiny $[b\!+\!1,c]$};
\draw[thick,double,densely dotted,PineGreen,to-] (0,.25) -- (0,.5);
\draw[thick,double,densely dotted,OliveGreen] (-.5,-.75)node[above]{\tiny $[a,b]$} -- (0,0);
\draw[thick,double,densely dotted,OliveGreen,-to] (-.5,-.75) -- (-.2,-.3);
\draw[thick,double,densely dotted,PineGreen,to-] (.5,-.75)node[above]{\tiny $[a,c]$} -- (0,0);
}}\endxy  
:=
\xy (0,0)*{
\tikzdiagc[yscale=0.9]{
\draw[thick,double,densely dotted,PineGreen,-to] (0,0) -- (0,.75)node[above]{\tiny $[a,c]$};
\draw[thick,double,densely dotted,OliveGreen] (-1,.75)node[above]{\tiny $[a,b]$} to[out=-90,in=180] (-.55,-.5) to[out=0,in=-135] (0,0);
\draw[thick,double,densely dotted,OliveGreen,-to] (-.24,-.32) -- (-.2,-.26);
\draw[thick,double,densely dotted,mygreen] (.5,-.75)node[below]{\tiny $[b\!+\!1,c]$} -- (0,0);
\draw[thick,double,densely dotted,mygreen,-to] (.5,-.75) -- (.2,-.3);
}}\endxy  
\]
satisfying the obvious relations in $K^b(\eS_n)$.


\begin{lem}\label{l:dumbelinsidebubbles}
The following equations hold in $K^b(\Sext_n)$: 
\begin{equation*}
\xy (0,0)*{
\tikzdiagc[scale=.8]{
\draw[thick,double,densely dotted,OliveGreen] (0,0) circle  (.65);\draw [thick,double,densely dotted,OliveGreen,-to] (.65,.09) --(.65,.1);
\node[OliveGreen] at (1.2,-.7) {\tiny $[b\!-\!1,a]$};
\node[OliveGreen] at (1.2, .7) {\phantom{\tiny $[b\!-\!1,a]$}};
\draw[ultra thick,BrickRed] (.02,-.25) -- (.02, 0.25)node[pos=0, tikzdot]{}node[pos=1, tikzdot]{};
\node[BrickRed] at (-.25,-.2) {\tiny $b$};
}}\endxy\mspace{-10mu}
=\ 
\sum\limits_{i=a}^{b}
\xy (0,0)*{
\tikzdiagc[scale=.8]{
\draw[ultra thick,mygreen] (0,-.25)node[below]{\tiny $i$} -- (0, 0.25)node[above]{\phantom{\tiny $i$}}node[pos=0, tikzdot]{}node[pos=1, tikzdot]{};
}}\endxy
\mspace{100mu}
\xy (0,0)*{
\tikzdiagc[scale=.8,yscale=1,xscale=1]{
\draw[thick,double,densely dotted,OliveGreen] (0,0) circle  (.65);\draw[thick,double,densely dotted,OliveGreen,-to] (.65,-.09) --(.65,-.1);
\node[OliveGreen] at (1.2,-.7) {\tiny $[a,b\!-\!1]$};
\node[OliveGreen] at (1.2, .7) {\phantom{\tiny $[a,b\!-\!1]$}};
\draw[ultra thick,BrickRed] (.02,-.25) -- (.02, 0.25)node[pos=0, tikzdot]{}node[pos=1, tikzdot]{};
\node[BrickRed] at (-.25,-.2) {\tiny $b$};
}}\endxy\mspace{-10mu}
=\ 
\sum\limits_{i=a}^{b}
\xy (0,0)*{
\tikzdiagc[scale=.8]{
\draw[ultra thick,mygreen] (0,-.25)node[below]{\tiny $i$} -- (0, 0.25)node[above]{\phantom{\tiny $i$}}node[pos=0, tikzdot]{}node[pos=1, tikzdot]{};
}}\endxy
\end{equation*}
\end{lem}

\begin{proof}
An easy computation using~\cite[Lemma 4.6]{mmv-evalfunctor}.
\end{proof}

\vspace{0.1in}
 
The chain complexes  $B_\rho^{-1}\rT_{[a,b]}$ and $\rT_{[a-1,b-1]}B_\rho^{-1}$ are isomorphic. This is an easy consequence of the isomorphism of the chain complexes $B_\rho^{-1}\rT_{k}\cong \rT_{k-1}B_\rho^{-1}$, which follows from~\eqref{eq:orReidII}. 
This homotopy equivalence and its variants are represented by the diagrams  
\begin{equation}\label{eq:rnew5}
\xy (0,0)*{
  \tikzdiagc[yscale=1,xscale=1]{
\draw[thick,double,densely dotted, myblue,-to] (1,0)node[below]{\tiny $[a,b]$} -- (.4,0);
\draw[thick,double,densely dotted, myblue] (.4,0) -- (0,0);
\draw[thick,double,densely dotted, myred] (-1,0)node[below]{\tiny $[a\!-\!1,b\!-\!1]\mspace{30mu}$} -- (-.6,0);
\draw[thick,double,densely dotted, myred,to-] (-.6,0) -- (0,0);
\draw[ultra thick,black,to-] (0,-1) -- (0,1);
  }}\endxy
\mspace{60mu}
\xy (0,0)*{
  \tikzdiagc[yscale=1,xscale=1]{
\draw[thick,double,densely dotted, myblue] (1,0)node[below]{\tiny $[a,b]$} -- (.65,0);
\draw[thick,double,densely dotted, myblue,to-] (.65,0) -- (0,0);
\draw[thick,double,densely dotted, myred,-to] (-1,0)node[below]{\tiny $[a\!-\!1,b\!-\!1]\mspace{30mu}$} -- (-.35,0);
\draw[thick,double,densely dotted, myred] (-.35,0) -- (0,0);
\draw[ultra thick,black,to-] (0,-1) -- (0,1);
  }}\endxy
\end{equation}
satisfying the relations (which are easy consequences of~\eqref{eq:orReidII})
\begin{equation}\label{eq:orcomm}
\xy (0,0)*{
\tikzdiagc[yscale=2.1,xscale=1.1]{
\draw[thick,double,densely dotted,myblue] (1,0)node[below]{\tiny $[a,b]$} .. controls (1,.15) and  (.7,.24) .. (.5,.29);
\draw[thick,double,densely dotted,myblue,-to] (0.85,.17) --  (.75,.21);
\draw[thick,double,densely dotted,myred] (.5,.29) .. controls (-.1,.4) and (-.1,.6) .. (.5,.71);
\draw[thick,double,densely dotted,myred,-to] (.05,.50) -- (.05,.52);
\draw[thick,double,densely dotted,myblue,to-] (1,1)node[above]{\tiny $[a,b]$} .. controls (1,.85) and  (.7,.74) .. (.5,.71) ;
\node[myred] at (-.8,.5) {\tiny $[a\!-\!1,b\!-\!1]$};
\draw[ultra thick,black,to-] (0,0) ..controls (0,.35) and (1,.25) .. (1,.5) ..controls (1,.75) and (0,.65) .. (0,1);
  }}\endxy
= \ \ 
\xy (0,0)*{
  \tikzdiagc[yscale=2.1,xscale=1.1]{
\draw[thick,double,densely dotted,myblue,-to] (1,0)node[below]{\tiny $[a,b]$} -- (1,1)node[above]{\phantom{\tiny $[a,b]$}};
    \draw[ultra thick,black,to-] (0,0) -- (0,1);
}}\endxy
\mspace{60mu}
\xy (0,0)*{
\tikzdiagc[yscale=-2.1,xscale=1.1]{
\draw[thick,double,densely dotted,myblue] (1,0)node[above]{\tiny $[a,b]$} .. controls (1,.15) and  (.7,.24) .. (.5,.29);
\draw[thick,double,densely dotted,myblue,-to] (0.85,.17) --  (.75,.21);
\draw[thick,double,densely dotted,myred] (.5,.29) .. controls (-.1,.4) and (-.1,.6) .. (.5,.71);
\draw[thick,double,densely dotted,myred,-to] (.05,.50) -- (.05,.52);
\draw[thick,double,densely dotted,myblue,to-] (1,1)node[below]{\tiny $[a,b]$} .. controls (1,.85) and  (.7,.74) .. (.5,.71) ;
\node[myred] at (-.8,.5) {\tiny $[a\!-\!1,b\!-\!1]$};
\draw[ultra thick,black,-to] (0,0) ..controls (0,.35) and (1,.25) .. (1,.5) ..controls (1,.75) and (0,.65) .. (0,1);
  }}\endxy
= \ \ 
\xy (0,0)*{
  \tikzdiagc[yscale=-2.1,xscale=1.1]{
\draw[thick,double,densely dotted,myblue,-to] (1,0)node[above]{\tiny $[a,b]$} -- (1,1)node[below]{\phantom{\tiny $[a,b]$}};
    \draw[ultra thick,black,-to] (0,0) -- (0,1);
}}\endxy
\end{equation}

\smallskip

We assume that $n>2$ until the end of this section. We say that $[a,b]$ and $[c,d]$ are \emph{distant} if every index appearing in $\rT_{[a,b]}$ is distant from every index in $\rT_{[c,d]}$. In this case the chain complexes $\rT_{[a,b]}\rT_{[c,d]}^{\pm 1}$ and $\rT_{[c,d]}^{\pm 1}\rT_{[a,b]}$ are isomorphic, with the isomorphism being given by 
diagrams like
\begin{equation}\label{eq:rnew2}
\xy (0,-3)*{
\tikzdiagc[yscale=1,xscale=1]{
\draw[thick,double,densely dotted,BrickRed,-to] (1,0)node[below]{\tiny $[c,d]$} -- (-1,0);
\draw[thick,double,densely dotted,OliveGreen,-to] (0,-1)node[below]{\tiny $[a,b]$} -- (0,1);
}}\endxy
\end{equation}
which satisfy the relations in $K^b(\Sext)$ given below,
for mutually distant $[a,b]$, $[c,d]$ and $[e,f]$ and with any of the possible orientations 

\begin{equation}\label{eq:Rtwo-commcomm}
\xy (0,0)*{
\tikzdiagc[yscale=2.1,xscale=1.1]{
\draw[thick,double,densely dotted,OliveGreen] (0,0)node[below]{\tiny $[a,b]$} ..controls (0,.35) and (1,.25) .. (1,.5) ..controls (1,.75) and (0,.65) .. (0,1);
\draw[thick,double,densely dotted,BrickRed] (1,0)node[below]{\tiny $[c,d]$} ..controls (1,.35) and (0,.25) .. (0,.5) ..controls (0,.75) and (1,.65) .. (1,1)node[above]{\phantom{\tiny $[c,d]$}};
  }}\endxy
= 
\xy (0,0)*{
  \tikzdiagc[yscale=2.1,xscale=1.1]{
\draw[thick,double,densely dotted,BrickRed] (1,0)node[below]{\tiny $[c,d]$} -- (1,1)node[above]{\phantom{\tiny $[c,d]$}};
\draw[thick,double,densely dotted,OliveGreen] (0,0)node[below]{\tiny $[a,b]$} -- (0,1);
}}\endxy
\mspace{90mu}
\xy (0,0)*{
\tikzdiagc[yscale=2.1,xscale=1.1]{
\draw[thick,double,densely dotted,BrickRed] (1,0)node[below]{\tiny $[c,d]$} ..controls (1,.35) and (0,.25) .. (0,.5) ..controls (0,.75) and (1,.65) .. (1,1)node[above]{\phantom{\tiny $[c,d]$}};
\draw[thick,double,densely dotted,OliveGreen] (0,0)node[below]{\tiny $[a,b]$} to[out=45,in=-135] (2,1);
\draw[thick,double,densely dotted,NavyBlue] (2,0)node[below]{\tiny $[e,f]$} to[out=135,in=-45] (0,1);
}}\endxy
=
\xy (0,0)*{
\tikzdiagc[yscale=2.1,xscale=-1.1]{
\draw[thick,double,densely dotted,BrickRed] (1,0)node[below]{\tiny $[c,d]$} ..controls (1,.35) and (0,.25) .. (0,.5) ..controls (0,.75) and (1,.65) .. (1,1)node[above]{\phantom{\tiny $[c,d]$}};
\draw[thick,double,densely dotted,NavyBlue] (0,0)node[below]{\tiny $[e,f]$} to[out=45,in=-135] (2,1);
\draw[thick,double,densely dotted,OliveGreen] (2,0)node[below]{\tiny $[a,b]$} to[out=135,in=-45] (0,1);
}}\endxy
\end{equation}
and (together with its variants)
\begin{equation}\label{eq:cablepitchfork}
\xy (0,0)*{
\tikzdiagc[yscale=2.1,xscale=1.1]{
\draw[thick,densely dotted,double,NavyBlue] (.85,0)node[below]{\tiny $[a,b]$} to (.25,.5);
\draw[thick,densely dotted,double,purple] (.25,.5) to (.85,1)node[above]{\tiny $[g\!+\!1,b]$};
\draw[thick,densely dotted,double,Brown,-to] (.25,.5) to (-.75,.5)node[below]{\tiny $[a,g]$};
\draw[thick,double,densely dotted,OliveGreen] (0,0)node[below]{\tiny $[c,d]$} ..controls (0,.35) and (1,.25) .. (1,.5) ..controls (1,.75) and (0,.65) .. (0,1);
  }}\endxy
=
\xy (0,0)*{
\tikzdiagc[yscale=2.1,xscale=1.1]{
\draw[thick,densely dotted,double,NavyBlue] (.85,0)node[below]{\tiny $[a,b]$} to (.25,.5);
\draw[thick,densely dotted,double,purple] (.25,.5) to (.85,1)node[above]{\tiny $[g\!+\!1,b]$};
\draw[thick,densely dotted,double,Brown,-to] (.25,.5) to (-.75,.5)node[below]{\tiny $[a,g]$};
\draw[thick,double,densely dotted,OliveGreen] (0,0)node[below]{\tiny $[c,d]$} to[out=90,in=-90] (-.10,.5) to[out=90,in=-90] (0,1);
  }}\endxy
\end{equation}

\vspace{0.1in}

When $k=[k,k]$ is distant from $[a,b]$, the chain complexes $\rB_k\rT_{[a,b]}$ and $\rT_{[a,b]}\rB_k$ are homotopy equivalent and the 
homotopy equivalence can be represented by the diagram 
by the diagram 
\begin{equation}\label{eq:rnew3}
\xy (0,-3)*{
  \tikzdiagc[yscale=1,xscale=1]{
\draw[ultra thick,NavyBlue] (1,0)node[below]{\tiny $k$} -- (-1,0);
\draw[thick,double,densely dotted,OliveGreen,-to] (0,-1)node[below]{\tiny $[a,b]$} -- (0,1);
  }}\endxy
\end{equation}
satisfying the relations
\begin{equation}\label{eq:distcomm}
\xy (0,-3)*{
\tikzdiagc[yscale=2.1,xscale=1.1]{
  \draw[ultra thick,NavyBlue] (1,0)node[below]{\tiny $k$} ..controls (1,.35) and (0,.25) .. (0,.5) ..controls (0,.75) and (1,.65) .. (1,1);
\draw[thick,double,densely dotted,OliveGreen,-to] (0,0)node[below]{\tiny $[a,b]$} ..controls (0,.35) and (1,.25) .. (1,.5) ..controls (1,.75) and (0,.65) .. (0,1);
  }}\endxy
= 
\xy (0,0)*{
  \tikzdiagc[yscale=2.1,xscale=1.1]{
\draw[ultra thick,NavyBlue] (1,0)node[below]{\tiny $k$} -- (1,1)node[above]{\phantom{\tiny $j$}};
\draw[thick,double,densely dotted,OliveGreen,-to] (0,0)node[below]{\tiny $[a,b]$} -- (0,1);
}}\endxy
\mspace{55mu}
\xy (0,-3)*{
\tikzdiagc[yscale=2.1,xscale=1.1]{
\draw[ultra thick,NavyBlue] (.85,0)node[below]{\tiny $k$} to (.25,.5) to (.85,1);
\draw[ultra thick,NavyBlue] (.25,.5) to (-.5,.5);
\draw[thick,double,densely dotted,OliveGreen,-to] (0,0)node[below]{\tiny $[a,b]$} ..controls (0,.35) and (1,.25) .. (1,.5) ..controls (1,.75) and (0,.65) .. (0,1);
  }}\endxy
=
\xy (0,-3)*{
\tikzdiagc[yscale=2.1,xscale=1.1]{
\draw[ultra thick,NavyBlue] (.85,0)node[below]{\tiny $k$} to (.25,.5) to (.85,1);
\draw[ultra thick,NavyBlue] (.25,.5) to (-.5,.5);
\draw[thick,double,densely dotted,OliveGreen,-to] (0,0)node[below]{\tiny $[a,b]$} to  (0,1);
  }}\endxy
\mspace{55mu}  
\xy (0,-3)*{
\tikzdiagc[yscale=1,xscale=1]{
\draw[ultra thick,NavyBlue] (1,0)node[below]{\tiny $k$} -- (-.75,0)node[pos=1, tikzdot]{};
\draw[thick,double,densely dotted,OliveGreen,-to] (0,-1)node[below]{\tiny $[a,b]$} -- (0,1);
  }}\endxy
  =
\xy (0,-3)*{
\tikzdiagc[yscale=1,xscale=1]{
\draw[ultra thick,NavyBlue] (1,0)node[below]{\tiny $k$} -- (.4,0)node[pos=1, tikzdot]{};
\draw[thick,double,densely dotted,OliveGreen,-to] (0,-1)node[below]{\tiny $[a,b]$} -- (0,1);
}}\endxy
\end{equation}
and its variants. 
If $[a,b]$ is distant from $i$, $j$ and $k$, and $j$ is adjacent to $i$ but distant from $k$, we also have
\begin{equation}\label{eq:cablecrossingsixv}
\xy (0,0)*{
  \tikzdiagc[scale=.7]{
 \draw[ultra thick,myred] (-1,-1)node[below]{\tiny $j$} -- (1,1)node[above]{\phantom{\tiny $i$}};
 \draw[ultra thick,NavyBlue] (1,-1)node[below]{\tiny $k$} -- (-1,1)node[above]{\phantom{\tiny $k$}};
 \draw[thick,double,densely dotted,OliveGreen,-to] (-1,0)node[left]{\tiny $[a\!,\!b]$} ..controls (-.25,.75) and (.25,.75) .. (1,0);
}}\endxy
=
\xy (0,0)*{
  \tikzdiagc[scale=.7]{
 \draw[ultra thick,myred] (-1,-1)node[below]{\tiny $j$} -- (1,1)node[above]{\phantom{\tiny $i$}};
 \draw[ultra thick,NavyBlue] (1,-1)node[below]{\tiny $k$} -- (-1,1)node[above]{\phantom{\tiny $k$}};
\draw[thick,double,densely dotted,OliveGreen,-to] (-1,0)node[left]{\tiny $[a\!,\!b]$} ..controls (-.25,-.75) and (.25,-.75) .. (1,0);
}}\endxy
\mspace{60mu}
\xy (0,0)*{
  \tikzdiagc[scale=.7]{
  \draw[ultra thick,myred] (0,-1)node[below]{\tiny $j$} -- (0,0);\draw[ultra thick,myred] (0,0) -- (-1, 1)node[above]{\phantom{\tiny $j$}};\draw[ultra thick,myred] (0,0) -- (1, 1);
  \draw[ultra thick,blue] (0,0)-- (0, 1)node[above]{\phantom{\tiny $i$}}; \draw[ultra thick,blue] (-1,-1)node[below]{\tiny $i$} -- (0,0); \draw[ultra thick,blue] (1,-1) -- (0,0);
\draw[thick,double,densely dotted,OliveGreen,-to] (-1,0)node[left]{\tiny $[a\!,\!b]$} ..controls (-.25,.75) and (.25,.75) .. (1,0);  
}}\endxy
\ =\ \ 
\xy (0,0)*{
  \tikzdiagc[scale=.7]{
  \draw[ultra thick,myred] (0,-1)node[below]{\tiny $j$} -- (0,0);\draw[ultra thick,myred] (0,0) -- (-1, 1)node[above]{\phantom{\tiny $j$}};\draw[ultra thick,myred] (0,0) -- (1, 1);
  \draw[ultra thick,blue] (0,0)-- (0, 1)node[above]{\phantom{\tiny $i$}}; \draw[ultra thick,blue] (-1,-1)node[below]{\tiny $i$} -- (0,0); \draw[ultra thick,blue] (1,-1) -- (0,0);
\draw[thick,double,densely dotted,OliveGreen,-to] (-1,0)node[left]{\tiny $[a\!,\!b]$} ..controls (-.25,-.75) and (.25,-.75) .. (1,0);  
}}\endxy
\end{equation}

\vspace{0.1in}

For $0<a,b<n$ the chain complex $\rT_{[a,b]}B_k$ is homotopy equivalent to 
$B_{k-1}\rT_{[a,b]}$ when $b<k\leq a$, and homotopy equivalent to 
$B_{k+1}\rT_{[a,b]}$ when $a\leq k<b$ (note that $a\neq b$). This is proved in the same way as~\cite[Lemma 4.11]{mmv-evalfunctor}.
The homotopy equivalences are realized by the following family of diagrams 
\begin{equation}\label{eq:rnew4}
\xy (0,-3)*{
  \tikzdiagc[yscale=1,xscale=1]{
\draw[ultra thick,myblue] (1,0)node[below]{\tiny $k$} -- (0,0);
\draw[ultra thick,myred] (-1,0)node[below]{\tiny $k+\varepsilon_{ab}$} -- (0,0);
\draw[thick,double,densely dotted,OliveGreen,-to] (0,-1)node[below]{\tiny $[a,b]$} -- (0,1);
  }}\endxy,
  \; \text{where $\varepsilon_{ab}=\frac{b-a}{\vert b-a\vert}$ and $k$ between  $a$ and $b-\varepsilon_{ab}$.}
\end{equation}
satisfying the relations 
\begingroup\allowdisplaybreaks
\begin{gather}
\label{eq:Rtwo-BTcomm}
\xy (0,0)*{
\tikzdiagc[yscale=2.1,xscale=1.1]{
\draw[ultra thick,myblue] (1,0)node[below]{\tiny $k$} .. controls (1,.15) and  (.7,.24) .. (.5,.29);
\draw[ultra thick,myred] (.5,.29) .. controls (-.1,.4) and (-.1,.6) .. (.5,.71);
\draw[ultra thick,myblue] (1,1)node[above]{\tiny $k$} .. controls (1,.85) and  (.7,.74) .. (.5,.71) ;
\node[myred] at (-.55,.5) {\tiny $k+\varepsilon_{ab}$};
\draw[thick,double,densely dotted,OliveGreen,-to] (0,0)node[below]{\tiny $[a,b]$} ..controls (0,.35) and (1,.25) .. (1,.5) ..controls (1,.75) and (0,.65) .. (0,1);
  }}\endxy
= 
\xy (0,0)*{
  \tikzdiagc[yscale=2.1,xscale=1.1]{
\draw[ultra thick,myblue] (1,0)node[below]{\tiny $k$} -- (1,1)node[above]{\phantom{\tiny $k$}};
\draw[thick,double,densely dotted,OliveGreen,-to] (0,0)node[below]{\tiny $[a,b]$} -- (0,1);
}}\endxy
\mspace{70mu}
\xy (0,0)*{
\tikzdiagc[yscale=2.1,xscale=1.1,xscale=-1]{
\draw[ultra thick,myred] (1,0)node[below]{\tiny $k+\varepsilon_{ab}$} .. controls (1,.15) and  (.7,.24) .. (.5,.29);
\draw[ultra thick,myblue] (.5,.29) .. controls (-.1,.4) and (-.1,.6) .. (.5,.71);
\draw[ultra thick,myred] (1,1)node[above]{\tiny $k+\varepsilon_{ab}$} .. controls (1,.85) and  (.7,.74) .. (.5,.71) ;
\node[myblue] at (-.15,.5) {\tiny $k$};
\draw[thick,double,densely dotted,OliveGreen,-to] (0,0)node[below]{\tiny $[a,b]$} ..controls (0,.35) and (1,.25) .. (1,.5) ..controls (1,.75) and (0,.65) .. (0,1);
  }}\endxy
=  
\xy (0,0)*{
  \tikzdiagc[yscale=2.1,xscale=1.1,xscale=-1]{
\draw[ultra thick,myred] (1,0)node[below]{\tiny $k+\varepsilon_{ab}$} -- (1,1)node[above]{\phantom{\tiny $k$}};
\draw[thick,double,densely dotted,OliveGreen,-to] (0,0)node[below]{\tiny $[a,b]$} -- (0,1);
}}\endxy
\\[1ex] 
\xy (0,0)*{
\tikzdiagc[yscale=1,xscale=1]{
  \draw[ultra thick,myred] (0,0) -- (-.75,0)node[pos=1, tikzdot]{}node[below]{\tiny $k+\varepsilon_{ab}$};
  \draw[ultra thick,myblue] (1,0)node[below]{\tiny $k$} -- (0,0);
\draw[thick,double,densely dotted,OliveGreen,-to] (0,-1)node[below]{\tiny $[a,b]$} -- (0,1)node[above]{\phantom{\tiny $[a,b]$}};
  }}\endxy
  =
\xy (0,0)*{
\tikzdiagc[yscale=1,xscale=1]{
\draw[ultra thick,myblue] (1,0)node[below]{\tiny $k$} -- (.4,0)node[pos=1, tikzdot]{};
\draw[thick,double,densely dotted,OliveGreen,-to] (0,-1)node[below]{\tiny $[a,b]$} -- (0,1)node[above]{\phantom{\tiny $[a,b]$}};
}}\endxy
\mspace{65mu}
\xy (0,0)*{
\tikzdiagc[yscale=1,xscale=-1]{
  \draw[ultra thick,myblue] (0,0) -- (-.75,0)node[pos=1, tikzdot]{}node[below]{\tiny $k$};
  \draw[ultra thick,myred] (1,0)node[below]{\tiny $k+\varepsilon_{ab}$} -- (0,0);
\draw[thick,double,densely dotted,OliveGreen,-to] (0,-1)node[below]{\tiny $[a,b]$} -- (0,1)node[above]{\phantom{\tiny $[a,b]$}};
  }}\endxy
  =
\xy (0,0)*{
\tikzdiagc[yscale=1,xscale=-1]{
\draw[ultra thick,myred] (1,0)node[below]{\tiny $k+\varepsilon_{ab}$} -- (.4,0)node[pos=1, tikzdot]{};
\draw[thick,double,densely dotted,OliveGreen,-to] (0,-1)node[below]{\tiny $[a,b]$} -- (0,1)node[above]{\phantom{\tiny $[a,b]$}};
}}\endxy
\\[1ex]
\xy (0,0)*{
\tikzdiagc[yscale=2.1,xscale=1.1]{
\draw[ultra thick,myblue] (.85,0)node[below]{\tiny $k$} to (.49,.295);
\draw[ultra thick,myblue] (.85,1) to (.49,.70);    
\draw[ultra thick,myred] (.49,.30) to (.25,.5) to (.49,.705);
\draw[ultra thick,myred] (.25,.5) to (-.5,.5)node[below]{\tiny $k+\varepsilon_{ab}$};
\draw[thick,double,densely dotted,OliveGreen,-to] (0,0)node[below]{\tiny $[a,b]$} ..controls (0,.35) and (1,.25) .. (1,.5) ..controls (1,.75) and (0,.65) .. (0,1)node[above]{\phantom{\tiny $[a,b]$}};
  }}\endxy
=
\xy (0,0)*{
\tikzdiagc[yscale=2.1,xscale=1.1]{
  \draw[ultra thick,myblue] (.85,0)node[below]{\tiny $k$} to (.25,.5) to (.85,1);
\draw[ultra thick,myblue] (.25,.5) to (0,.5);
\draw[ultra thick,myred] (0,.5) to (-.85,.5)node[below]{\tiny $k+\varepsilon_{ab}$};
\draw[thick,double,densely dotted,OliveGreen,-to] (0,0)node[below]{\tiny $[a,b]$} to  (0,1)node[above]{\phantom{\tiny $[a,b]$}};
}}\endxy
\mspace{65mu}  
\xy (0,0)*{
\tikzdiagc[yscale=2.1,xscale=-1.1]{
\draw[ultra thick,myred] (.85,0)node[below]{\tiny $k+\varepsilon_{ab}$} to (.49,.295);
\draw[ultra thick,myred] (.85,1) to (.49,.70);    
\draw[ultra thick,myblue] (.49,.30) to (.25,.5) to (.49,.705);
\draw[ultra thick,myblue] (.25,.5) to (-.5,.5)node[below]{\tiny $k$};
\draw[thick,double,densely dotted,OliveGreen,-to] (0,0)node[below]{\tiny $[a,b]$} ..controls (0,.35) and (1,.25) .. (1,.5) ..controls (1,.75) and (0,.65) .. (0,1)node[above]{\phantom{\tiny $[a,b]$}};
  }}\endxy
=
\xy (0,0)*{
\tikzdiagc[yscale=2.1,xscale=-1.1]{
  \draw[ultra thick,myred] (.85,0)node[below]{\tiny $k+\varepsilon_{ab}$} to (.25,.5) to (.85,1);
\draw[ultra thick,myred] (.25,.5) to (0,.5);
\draw[ultra thick,myblue] (0,.5) to (-.85,.5)node[below]{\tiny $k$};
\draw[thick,double,densely dotted,OliveGreen,-to] (0,0)node[below]{\tiny $[a,b]$} to (0,1)node[above]{\phantom{\tiny $[a,b]$}};
}}\endxy
\end{gather} 

and their variants. Relations~\eqref{eq:Rtwo-BTcomm} are immediate, while the relations in the remaining two lines are proved 
as in~\cite[Lemmas 4.18 and 4.21]{mmv-evalfunctor}, respectively.

To prove the next lemma we will use the \emph{hom \& dot trick}. Due to its extensive use in the next section, we will briefly explain it in the following remark. 

\begin{rem}[The hom \& dot trick]\label{r:dottrick}
Let $A$ and $B$ be two non-zero diagrams in a given morphism space. 
If the latter space is one-dimensional, then $A=\lambda B$, where $\lambda$ is a non-zero scalar. 
To prove that $\lambda=1$ we will often attach dots to certain endpoints of $A$ and $B$ and simplify the resulting diagrams so that they can be easily compared. 
\end{rem}

\begin{lem}\label{l:commstr}
For $a,b,c\in \mathbb{Z}$ such that $c$ is distant from $[a,b]$, the following relations hold in $K^b(\Sext_n)$: 
\begin{equation}\label{eq:commstr2}
\xy (0,0)*{
\tikzdiagc[scale=.75]{
\draw[thick,double,densely dotted,myred,-to] (-.89,-.6) to[out=80,in=-110] (-.89,-.58);
\draw[thick,double,densely dotted,myblue,-to] (0,0) to[out=25,in=-90] (1,1)node[above]{\tiny $[a,\!b]$};
\draw[thick,double,densely dotted,myred] (-1,-1) to[out=90,in=-155] (0,0);
\draw[ultra thick,black,-to] (1,-1) to[out=90,in=-90] (-1,1);
\draw[thick,double,densely dotted,myred] (-1,-1.75)node[below]{\tiny $[a\!+\!1,b\!+\!1]$} to (-1,-1);
\draw[ultra thick,black] (1,-1.75) to (1,-1);
\draw[ultra thick,NavyBlue] (-1.65,-1)node[left]{\tiny $c\!+\!1$} to (1,-1);
\draw[ultra thick,Purple] (1,-1) to (1.65,-1)node[right]{\tiny $c$};
  }}\endxy
=
\xy (0,0)*{
\tikzdiagc[scale=.75]{
\draw[thick,double,densely dotted,myred,-to] (-.89,-.6) to[out=80,in=-110] (-.89,-.58);
\draw[thick,double,densely dotted,myblue] (0,0) to[out=25,in=-90] (1,1);
\draw[thick,double,densely dotted,myred] (-1,-1)node[below]{\tiny $[a\!+\!1,b\!+\!1]$} to[out=90,in=-155] (0,0);
\draw[ultra thick,black] (1,-1) to[out=90,in=-90] (-1,1);
\draw[thick,double,densely dotted,myblue,-to] (1,1) to (1,1.75)node[above]{\tiny $[a,\!b]$};
\draw[ultra thick,black,-to] (-1,1) to (-1,1.75);
\draw[ultra thick,NavyBlue] (-1.65,1)node[left]{\tiny $c\!+\!1$} to (-1,1);
\draw[ultra thick,Purple] (-1,1) to (1.65,1)node[right]{\tiny $c$};
  }}\endxy
\end{equation}
\end{lem}

\begin{proof}
We use the hom \& dot trick.
By using cups and caps, both diagrams can be deformed to morphisms of degree zero from $\rB_{c+1}$ to $\rB_\rho \rT_{[a,b]}\rB_c\rB_\rho^{-1}\rT_{[a\!+\!1,b\!+\!1]}^{-1}$, which is homotopy equivalent to $\rB_{c+1}$ because  
\[
\rB_\rho \rT_{[a,b]}\rB_c\rB_\rho^{-1}\rT_{[a\!+\!1,b\!+\!1]}^{-1}\simeq \rT_{[a+1,b+1]}\rB_\rho \rB_c\rB_\rho^{-1}\rT_{[a\!+\!1,b\!+\!1]}^{-1} 
\simeq \rT_{[a+1,b+1]}\rB_{c+1}\rT_{[a\!+\!1,b\!+\!1]}^{-1}\simeq \rB_{c+1}.
\] 
Soergel's hom formula in \eqref{eq:Soergelhomformula} implies that the space of degree-preserving endomorphisms of $\rB_{c+1}$  
(either in $\Sext_n$ or in $K^b(\Sext_n)$) is one-dimensional, so the two deformed diagrams are non-zero multiples of each other, which implies that the same holds for the two diagrams in~\eqref{eq:commstr2}. Attaching a dot to one of the free ends of the non-oriented strand we see that these morphisms are equal.
\end{proof}
Of course, there are other variants of~\eqref{eq:commstr2} with, for example, other orientations, or involving one oriented black strand and two thick oriented strands.   
Note that we can even relax the condition that $c$ is distant from $[a,b]$. 
We leave the details to the reader. 

\smallskip

Until the end of this subsection, assume that $a,b\in \mathbb{Z}$ such that $a<b$. By applying~\cite[Lemma 4.8]{mmv-evalfunctor} recursively, one obtains a homotopy equivalence  
\begin{equation}\label{eq:homotopeq1}
\rT_{[b\!-\!1,a]}^{-1}\rB_{b}\rT_{[b\!-\!1,a]}\simeq \rT_{[b,a\!+\!1]}\rB_{a}\rT_{[b,a\!+\!1]}^{-1}.
\end{equation}
The corresponding isomorphism in $K^b(\Sext_n)$ and its inverse are represented by the diagrams 
\begin{equation}\label{eq:sixv-comm}
  \xy (0,0)*{
  \tikzdiagc[scale=0.6,xscale=-1]{
\draw[ultra thick,BrickRed] (0,-1)node[below]{\tiny $b$} to  (0,0);
\draw[ultra thick,orange] (0,0) to  (0,1)node[above]{\tiny $a$};
\draw[thick,double,densely dotted,myblue] (-1,-1) -- (0,0);
\draw[thick,double,densely dotted,myblue,to-] ( 1,-1) -- (0,0);
\draw[thick,double,densely dotted,mygreen] (0,0) to (-1,1);
\draw[thick,double,densely dotted,mygreen,-to] (0,0) to ( 1,1);
\node[myblue] at (-1.30,-1.47) {\tiny $[b\!-\!1,a]$};
\node[myblue] at ( 1.30,-1.47) {\tiny $[b\!-\!1,a]$};
\node[mygreen] at (-1.30, 1.39) {\tiny $[b,a\!+\!1]$};
\node[mygreen] at ( 1.30, 1.39) {\tiny $[b,a\!+\!1]$};
  }}\endxy
\quad \text{and} \quad 
 \xy (0,0)*{
  \tikzdiagc[scale=0.6,xscale=-1]{
\draw[ultra thick,orange] (0,-1)node[below]{\tiny $a$} to  (0,0);
\draw[ultra thick,BrickRed] (0,0) to  (0,1)node[above]{\tiny $b$};
\draw[thick,double,densely dotted,mygreen,to-] (-1,-1) -- (0,0);
\draw[thick,double,densely dotted,mygreen] ( 1,-1) -- (0,0);
\draw[thick,double,densely dotted,myblue,-to] (0,0) to (-1,1);
\draw[thick,double,densely dotted,myblue] (0,0) to ( 1,1);
\node[mygreen] at (-1.30,-1.47) {\tiny $[b,a\!+\!1]$};
\node[mygreen] at ( 1.30,-1.47) {\tiny $[b,a\!+\!1]$};
\node[myblue] at (-1.30, 1.39) {\tiny $[b\!-\!1,a]$};
\node[myblue] at ( 1.30, 1.39) {\tiny $[b\!-\!1,a]$};
  }}\endxy
\end{equation}
satisfying the relations
\begin{align}\label{eq:sixvertcommzero-ab}
\xy (0,0)*{
\tikzdiagc[yscale=2.1,xscale=1.35]{
\draw[ultra thick,BrickRed] (.5,0)node[below]{\tiny $b$} -- (.5,.3);
\draw[ultra thick,orange] (.5,.3) -- (.5,.7);
\draw[ultra thick,BrickRed] (.5,.7) -- (.5,1)node[above]{\tiny $b$};
\node[orange] at (.62,.5) {\tiny $a$};
\draw[thick,double,densely dotted,myblue] (1,0)node[below]{\tiny $[b\!-\!1,\! a]$} .. controls (1,.15) and  (.7,.24) .. (.5,.29);
\draw[thick,double,densely dotted,myblue,-to] (0.85,.17) --  (.75,.21);
\draw[thick,double,densely dotted,mygreen] (.5,.29) .. controls (-.1,.4) and (-.1,.6) .. (.5,.71);
\draw[thick,double,densely dotted,mygreen,-to] (.05,.52) -- (.05,.54);
\draw[thick,double,densely dotted,myblue,to-] (1,1)node[above]{\tiny $[b\!-\!1,\!a]$} .. controls (1,.85) and  (.7,.74) .. (.5,.71) ;
\node[mygreen] at (-.5,.5) {\tiny $[b,\!a\!+\!1]$};
\node[mygreen] at (1.5,.5) {\tiny $[b,\!a\!+\!1]$};
\draw[thick,double,densely dotted,myblue,to-] (0,0)node[below]{\tiny $[b\!-\!1,\! a]$} .. controls (0,.15) and  (.3,.24) .. (.5,.29);
\draw[thick,double,densely dotted,mygreen] (.5,.29) .. controls (1.1,.4) and (1.1,.6) .. (.5,.71);
\draw[thick,double,densely dotted,myblue] (0,1)node[above]{\tiny $[b\!-\!1,\!a]$} .. controls (0,.85) and  (.3,.74) .. (.5,.71) ;
\draw[thick,double,densely dotted,mygreen,to-] (.95,.46) -- (.95,.48);
\draw[thick,double,densely dotted,myblue,-to] (0.15,.83) --  (.25,.79);
  }}\endxy
&= \!\!
\xy (0,0)*{
  \tikzdiagc[yscale=2.1,xscale=1.35]{
\draw[thick,double,densely dotted,myblue,to-] (0,0)node[below]{\tiny $[b\!-\!1,\! a]$} -- (0,1)node[above]{\phantom{\tiny $[b\!-\!1,\! a]$}};
\draw[ultra thick,BrickRed] (.5,0)node[below] {\tiny $b$} -- (.5,1);
\draw[thick,double,densely dotted,myblue,-to] (1,0)node[below]{\tiny $[b\!-\!1,\! a]$} -- (1,1);
}}\endxy
&
\xy (0,0)*{
\tikzdiagc[yscale=2.1,xscale=-1.35]{
\draw[ultra thick,orange] (.5,0)node[below]{\tiny $a$} -- (.5,.3);
\draw[ultra thick,BrickRed] (.5,.3) -- (.5,.7);
\draw[ultra thick,orange] (.5,.7) -- (.5,1)node[above]{\tiny $a$};
\node[BrickRed] at (.38,.5) {\tiny $b$};
\draw[thick,double,densely dotted,mygreen] (1,0)node[below]{\tiny $[b,\!a\!+\!1]$} .. controls (1,.15) and  (.7,.24) .. (.5,.29);
\draw[thick,double,densely dotted,mygreen,-to] (0.85,.17) --  (.75,.21);
\draw[thick,double,densely dotted,myblue] (.5,.29) .. controls (-.1,.4) and (-.1,.6) .. (.5,.71);
\draw[thick,double,densely dotted,myblue,-to] (.05,.52) -- (.05,.54);
\draw[thick,double,densely dotted,mygreen,to-] (1,1)node[above]{\tiny $[b,\!a\!+\!1]$} .. controls (1,.85) and  (.7,.74) .. (.5,.71) ;
\node[myblue] at (-.5,.5) {\tiny $[b\!-\!1,\! a]$};
\node[myblue] at (1.5,.5) {\tiny $[b\!-\!1,\! a]$};
\draw[thick,double,densely dotted,mygreen,to-] (0,0)node[below]{\tiny $[b,a\!+\!1]$} .. controls (0,.15) and  (.3,.24) .. (.5,.29);
\draw[thick,double,densely dotted,myblue] (.5,.29) .. controls (1.1,.4) and (1.1,.6) .. (.5,.71);
\draw[thick,double,densely dotted,mygreen] (0,1)node[above]{\tiny $[b,\!a\!+\!1]$} .. controls (0,.85) and  (.3,.74) .. (.5,.71) ;
\draw[thick,double,densely dotted,myblue,to-] (.95,.46) -- (.95,.48);
\draw[thick,double,densely dotted,mygreen,-to] (0.15,.83) --  (.25,.79);
  }}\endxy
&=  \!\!
\xy (0,0)*{
  \tikzdiagc[yscale=2.1,xscale=1.35]{
\draw[thick,double,densely dotted,mygreen,-to] (0,0)node[below]{\tiny $[b,\!a\!+\!1]$} -- (0,1)node[above]{\phantom{\tiny $[b,\!a\!+\!1]$}};
\draw[ultra thick,orange] (.5,0)node[below] {\tiny $a$} -- (.5,1);
\draw[thick,double,densely dotted,mygreen,to-] (1,0)node[below]{\tiny $[b,\!a\!+\!1]$} -- (1,1);
}}\endxy
\end{align}
in $K^b(\Sext)$. There is a similar homotopy equivalence 
\begin{equation}\label{eq:homotopeq2}
\rT_{[a,b\!-\!1]}\rB_{b}\rT_{[a,b\!-\!1]}^{-1}\simeq \rT_{[a\!+\!1,b]}^{-1}\rB_{a}\rT_{[a\!+\!1,b]},
\end{equation}
and the corresponding isomorphism in $K^b(\Sext_n)$ and its inverse are represented by the diagrams 
\begin{equation}\label{eq:sixv-comm2}
 \xy (0,0)*{
  \tikzdiagc[scale=0.6,xscale=-1]{
\draw[ultra thick,BrickRed] (0,-1)node[below]{\tiny $b$} to  (0,0);
\draw[ultra thick,orange] (0,0) to  (0,1)node[above]{\tiny $a$};
\draw[thick,double,densely dotted,mygreen,to-] (-1,-1) -- (0,0);
\draw[thick,double,densely dotted,mygreen] ( 1,-1) -- (0,0);
\draw[thick,double,densely dotted,myblue,-to] (0,0) to (-1,1);
\draw[thick,double,densely dotted,myblue] (0,0) to ( 1,1);
\node[mygreen] at (-1.30,-1.47) {\tiny $[a,b\!-\!1]$};
\node[mygreen] at ( 1.30,-1.47) {\tiny $[a,b\!-\!1]$};
\node[myblue] at (-1.30, 1.39) {\tiny $[a\!+\!1,b]$};
\node[myblue] at ( 1.30, 1.39) {\tiny $[a\!+\!1,b]$};
  }}\endxy
\quad \text{and} \quad 
\xy (0,0)*{
  \tikzdiagc[scale=0.6,xscale=-1]{
\draw[ultra thick,orange] (0,-1)node[below]{\tiny $a$} to  (0,0);
\draw[ultra thick,BrickRed] (0,0) to  (0,1)node[above]{\tiny $b$};
\draw[thick,double,densely dotted,myblue] (-1,-1) -- (0,0);
\draw[thick,double,densely dotted,myblue,to-] ( 1,-1) -- (0,0);
\draw[thick,double,densely dotted,mygreen] (0,0) to (-1,1);
\draw[thick,double,densely dotted,mygreen,-to] (0,0) to ( 1,1);
\node[myblue] at (-1.30,-1.47) {\tiny $[a\!+\!1,b]$};
\node[myblue] at ( 1.30,-1.47) {\tiny $[a\!+\!1,b]$};
\node[mygreen] at (-1.30, 1.39) {\tiny $[a,b\!-\!1]$};
\node[mygreen] at ( 1.30, 1.39) {\tiny $[a,b\!-\!1]$};
  }}\endxy
\end{equation}
satisfying relations analogous to those in~\eqref{eq:sixvertcommzero-ab}. 

\begin{rem}
One needs to be careful with the use of the colors of the strands in the diagrams in~\eqref{eq:sixv-comm} and~\eqref{eq:sixv-comm2}. For example, 
the definition of $\Psi_R$ in \autoref{sec:thefunctor} contains the diagram  
\[
\xy (0,0)*{
\tikzdiagc[scale=1.6]{
\draw[thick,double,densely dotted,mygreen] (.20,-.5) to (-.15,-.15);
\draw[thick,double,densely dotted,mygreen,to-] (.8,-.5) to (.15,.15);
\draw[thick,double,densely dotted,mygreen,to-] (.05,-.35) to (.10,-.40);
\draw[thick,double,densely dotted,orange] (-.15,-.15) to[out=135,in=-125] (-.3,.3);
\draw[thick,double,densely dotted,orange] (.15,.15) to[out=135,in=35] (-.3,.3);
\draw[thick,double,densely dotted,orange,to-] (-.30,.03) to (-.285,.01);
\draw[thick,double,densely dotted,orange,-to] (0.06,.225) to (.08,.210);
\draw[thick,double,densely dotted,myblue] (-.31,.31) to[out=-135,in=135] (-.3,-.3);
\draw[thick,double,densely dotted,myblue] (-.31,.31) to[out=40,in=135] (.3,.3);
\draw[thick,double,densely dotted,myblue,-to] (0.21,.375) to (.24,.355);
\draw[thick,double,densely dotted,myblue,to-] (-.40,-.159) to (-.35,-.230);
\draw[thick,double,densely dotted,BrickRed] (-.3,-.3) to[out=-45,in=45] (-.3,-.5);
\draw[thick,double,densely dotted,BrickRed,-to] (.3,.3) to[out=-45,in=-135] (.7,.5);
\node[myblue] at (0.10,.665) {\tiny $[0,k\!-\!1]$};
\node[BrickRed] at (-.25,-.68) {\tiny $[1,k]$};
\node[mygreen] at (1,-.68) {\tiny $[0,k\!-\!1]$};
\node[BrickRed] at (0.75, .68) {\tiny $[1,k]$};
\draw[ultra thick,violet] (.5,-.5)node[below] {\tiny $k$} -- (0,0);
\draw[ultra thick,brown] (-.5,.5)node[above] {\tiny $n\!-\!1$} -- (-.3,.3);
\draw[ultra thick,blue] (-.3,.3) -- (-.0,.0);
\draw[ultra thick,black,to-] (-.5,-.5) -- (.5,.5);
  }}\endxy
\] 
for $0<k<n$. To compare the labels of this diagram with the conventions in~\eqref{eq:sixv-comm2}, one needs to use the equality
\[
 \xy (0,0)*{
  \tikzdiagc[scale=0.6,xscale=-1]{
\draw[ultra thick,blue] (0,-1)node[below]{\tiny $k\!-\!1$} to  (0,0);
\draw[ultra thick,brown] (0,0) to  (0,1)node[above]{\tiny $-\!1$};
\draw[thick,double,densely dotted,orange,to-] (-1,-1) -- (0,0);
\draw[thick,double,densely dotted,orange] ( 1,-1) -- (0,0);
\draw[thick,double,densely dotted,myblue,-to] (0,0) to (-1,1);
\draw[thick,double,densely dotted,myblue] (0,0) to ( 1,1);
\node[orange] at (-1.7,-1.47) {\tiny $[-\!1,k\!-\!2]$};
\node[orange] at ( 1.7,-1.47) {\tiny $[-\!1,k\!-\!2]$};
\node[myblue] at (-1.7, 1.39) {\tiny $[0,k\!-\!1]$};
\node[myblue] at ( 1.7, 1.39) {\tiny $[0,k\!-\!1]$};
  }}\endxy
=
 \xy (0,0)*{
  \tikzdiagc[scale=0.6,xscale=-1]{
\draw[ultra thick,blue] (0,-1)node[below]{\tiny $k\!-\!1$} to  (0,0);
\draw[ultra thick,brown] (0,0) to  (0,1)node[above]{\tiny $n\!-\!1$};
\draw[thick,double,densely dotted,orange,to-] (-1,-1) -- (0,0);
\draw[thick,double,densely dotted,orange] ( 1,-1) -- (0,0);
\draw[thick,double,densely dotted,myblue,-to] (0,0) to (-1,1);
\draw[thick,double,densely dotted,myblue] (0,0) to ( 1,1);
\node[orange] at (-1.7,-1.47) {\tiny $[-\!1,k\!-\!2]$};
\node[orange] at ( 1.7,-1.47) {\tiny $[-\!1,k\!-\!2]$};
\node[myblue] at (-1.7, 1.39) {\tiny $[0,k\!-\!1]$};
\node[myblue] at ( 1.7, 1.39) {\tiny $[0,k\!-\!1]$};
  }}\endxy
\]
However, there is no homotopy equivalence given by a diagram of the form 
\[
 \xy (0,0)*{
  \tikzdiagc[scale=0.6,xscale=-1]{
\draw[ultra thick,blue] (0,-1)node[below]{\tiny $k\!-\!1$} to  (0,0);
\draw[ultra thick,brown] (0,0) to  (0,1)node[above]{\tiny $n\!-\!1$};
\draw[thick,double,densely dotted,orange,to-] (-1,-1) -- (0,0);
\draw[thick,double,densely dotted,orange] ( 1,-1) -- (0,0);
\draw[thick,double,densely dotted,myblue,-to] (0,0) to (-1,1);
\draw[thick,double,densely dotted,myblue] (0,0) to ( 1,1);
\node[orange] at (-1.7,-1.47) {\tiny $[n\!-\!1,k\!-\!2]$};
\node[orange] at ( 1.7,-1.47) {\tiny $[n\!-\!1,k\!-\!2]$};
\node[myblue] at (-1.7, 1.39) {\tiny $[n,k\!-\!1]$};
\node[myblue] at ( 1.7, 1.39) {\tiny $[n,k\!-\!1]$};
  }}\endxy
\]
because that would imply that $a=n-1$ and $b=k-1$ in~\eqref{eq:sixv-comm2}, violating our assumption that $a<b$. 
\end{rem}

Proceeding as in the proof of~\cite[Lemma 4.25]{mmv-evalfunctor}, one can show that the following holds in $K^b(\Sext)$:
\begin{equation}\label{eq:sixcommdot}
\xy (0,0)*{
\tikzdiagc[scale=0.6,xscale=-1]{
\draw[ultra thick,BrickRed] (0,-1)node[below]{\tiny $b$} --  (0,0);
\node[orange] at (0,1.4) {\tiny $a$};
\draw[ultra thick,orange] (0,0) -- (0,.85)node[pos=1, tikzdot]{};
\draw[thick,double,densely dotted,myblue] (-1,-1) -- (0,0);
\draw[thick,double,densely dotted,myblue,to-] ( 1,-1) -- (0,0);
\draw[thick,double,densely dotted,mygreen] (0,0) to (-1,1);
\draw[thick,double,densely dotted,mygreen,-to] (0,0) to ( 1,1);
\node[myblue] at (-1.30,-1.47) {\tiny $[b\!-\!1,a]$};
\node[myblue] at ( 1.30,-1.47) {\tiny $[b\!-\!1,a]$};
\node[mygreen] at (-1.30, 1.39) {\tiny $[b,a\!+\!1]$};
\node[mygreen] at ( 1.30, 1.39) {\tiny $[b,a\!+\!1]$};
}}\endxy
\mspace{-10mu}=\mspace{-10mu}
\xy (0,0)*{
\tikzdiagc[scale=0.6,xscale=-1]{
\draw[ultra thick,BrickRed] (0,-.55) -- (0,-1)node[pos=0, tikzdot]{};
\node[BrickRed] at (0,-1.4) {\tiny $b$};
\draw[thick,double,densely dotted,myblue,-to] (-1,-1) to[out=60,in=180] (0,-.15) to[out=0,in=120] (1,-1);
\draw[thick,double,densely dotted,mygreen,-to] (-1,1) to[out=-60,in=180] (0,.15) to[out=0,in=-120] ( 1,1);
\node[myblue] at (-1.30,-1.47) {\tiny $[b\!-\!1,a]$};
\node[myblue] at ( 1.30,-1.47) {\tiny $[b\!-\!1,a]$};
\node[mygreen] at (-1.30, 1.39) {\tiny $[b,a\!+\!1]$};
\node[mygreen] at ( 1.30, 1.39) {\tiny $[b,a\!+\!1]$};
}}\endxy
\end{equation}
For $k$ distant from $[a,b]$, it is easy to see that 
\begin{equation}\label{eq:jslidecommab}
  \xy (0,0)*{
  \tikzdiagc[scale=0.75,xscale=-1]{
\draw[ultra thick,BrickRed] (0,-1)node[below]{\tiny $b$} to  (0,0);
\draw[ultra thick,orange] (0,0) to  (0,1)node[above]{\tiny $a$};
\draw[thick,double,densely dotted,myblue] (-1,-1) -- (0,0);
\draw[thick,double,densely dotted,myblue,to-] ( 1,-1) -- (0,0);
\draw[thick,double,densely dotted,mygreen] (0,0) to (-1,1);
\draw[thick,double,densely dotted,mygreen,-to] (0,0) to ( 1,1);
\node[myblue] at (-1.30,-1.47) {\tiny $[b\!-\!1,a]$};
\node[myblue] at ( 1.30,-1.47) {\tiny $[b\!-\!1,a]$};
\node[mygreen] at (-1.30, 1.39) {\tiny $[b,a\!+\!1]$};
\node[mygreen] at ( 1.30, 1.39) {\tiny $[b,a\!+\!1]$};
\draw[ultra thick,gray] (-1.1,0) to[out=45,in=135]  (1.1,0);
\node[gray] at (1.25,0) {\tiny $k$};
}}\endxy
=
  \xy (0,0)*{
  \tikzdiagc[scale=0.75,xscale=-1]{
\draw[ultra thick,BrickRed] (0,-1)node[below]{\tiny $b$} to  (0,0);
\draw[ultra thick,orange] (0,0) to  (0,1)node[above]{\tiny $a$};
\draw[thick,double,densely dotted,myblue] (-1,-1) -- (0,0);
\draw[thick,double,densely dotted,myblue,to-] ( 1,-1) -- (0,0);
\draw[thick,double,densely dotted,mygreen] (0,0) to (-1,1);
\draw[thick,double,densely dotted,mygreen,-to] (0,0) to ( 1,1);
\node[myblue] at (-1.30,-1.47) {\tiny $[b\!-\!1,a]$};
\node[myblue] at ( 1.30,-1.47) {\tiny $[b\!-\!1,a]$};
\node[mygreen] at (-1.30, 1.39) {\tiny $[b,a\!+\!1]$};
\node[mygreen] at ( 1.30, 1.39) {\tiny $[b,a\!+\!1]$};
\draw[ultra thick,gray] (-1.1,0) to[out=-45,in=-135]  (1.1,0);
\node[gray] at (1.25,0) {\tiny $k$};
}}\endxy
\end{equation}

\begin{lem}
 We have the following in $K^b(\Sext)$:
\begin{equation}\label{eq:sixvtriv}
\xy (0,0)*{
\tikzdiagc[scale=0.6]{
\draw[ultra thick,BrickRed] (0,-1)node[below]{\tiny $b$} --  (0,0);
\draw[ultra thick,BrickRed] (-1,1) to[out=-60] (0,0) to[in=-120] (1,1);
\draw[ultra thick,orange] (-1,1) to[out=100,in=-90]  (-1,2.5)node[above]{\tiny $a$};
\draw[ultra thick,orange] (1,1) to[out=80,in=-90]  (1,2.5)node[above]{\tiny $a$};
\draw[thick,double,densely dotted,myblue,to-] (-2,0) to (-1,1) to[out=-45,in=-135] (1,1) to (2,0)node[below]{\tiny $[b\!-\!1,a]$};
\draw[thick,double,densely dotted,mygreen,to-] (-2.2,1.5) to (-1,1) to[out=45,in=135] (1,1) to (2.2,1.5)node[above]{\tiny $[b,a\!+\!1]$}; 
}}\endxy
\mspace{-10mu}=\mspace{10mu}
\xy (0,0)*{
\tikzdiagc[scale=0.6]{
\draw[ultra thick,BrickRed] (0,-1)node[below]{\tiny $b$} --  (0,.5);
\draw[ultra thick,orange] (0,.5) to  (0,1.25);
\draw[ultra thick,orange] (0,1.25) to[out=135,in=-90]  (-1,2.5)node[above]{\tiny $a$};
\draw[ultra thick,orange] (0,1.25) to[out=45,in=-90]  (1,2.5)node[above]{\tiny $a$};
\draw[thick,double,densely dotted,myblue,to-] (-2,0) to (0,.5) to (2,0) node[below]{\tiny $[b\!-\!1,a]$};
\draw[thick,double,densely dotted,mygreen,to-] (-2.2,1.5) to (0,.5) to (2.2,1.5)node[above]{\tiny $[b,a\!+\!1]$}; 
}}\endxy
\end{equation}
\end{lem}
\begin{proof}
We use the hom \& dot trick. By using cups and caps, both diagrams are isotopic to diagrams realizing morphisms from $\rT_{[b-1,a]}^{-1}\rB_b\rT_{[b-1,a]}$ to $\rT_{[b,a+1]}\rB_a\rB_a\rT_{[b,a+1]}^{-1}$. Gluing the second diagram in~\eqref{eq:sixv-comm} at the bottom results in two morphisms of degree $-1$ from $\rT_{[b,a+1]}\rB_a\rT_{[b,a+1]}^{-1}$ to $\rT_{[b,a+1]}\rB_a\rB_a\rT_{[b,a+1]}^{-1}$. 
By biadjointness and the invertibility of $\rT_{[b,a+1]}$ in $K^b(\Sext_n)$, there is a canonical isomorphism 
\begin{align*}
\hom_{K^b(\Sext_n)}\left(\rT_{[b,a+1]}\rB_a\rT_{[b,a+1]}^{-1}, \rT_{[b,a+1]}\rB_a\rB_a\rT_{[b,a+1]}^{-1}\langle -1\rangle\right) 
&\cong \hom_{K^b(\Sext_n)}\left(\rB_a, \rB_a\rB_a\langle -1\rangle\right)
\\
&\cong \hom_{\Sext_n}\left(\rB_a,\rB_a\rB_a\langle -1\rangle\right) 
\cong\bR .
\end{align*}
The latter isomorphism follows from Soergel's hom formula in \eqref{eq:Soergelhomformula} and the fact that $\rB_a \rB_a\cong \rB_a\langle 1\rangle\oplus \rB_a\langle -1\rangle$.
By attaching dots to the free top ends labeled $a$ in \eqref{eq:sixvtriv} and using~\eqref{eq:sixcommdot} and \eqref{eq:onecolorsecond}, one sees that they are equal. 
\end{proof}


\begin{lem}\label{l:sixv-reidthree}
For $c$ distant from $[b,a]$, we have  
\begin{equation}\label{eg:sixv-reidthree}
\xy (0,0)*{
\tikzdiagc[scale=0.6]{
\draw[ultra thick,blue] (-2,2) --  (2,-2)node[below]{\tiny $c$};
\draw[ultra thick,PineGreen] (-2,-2)node[below]{\tiny $a$} -- (1,1);
\draw[ultra thick,Orchid] (1,1) -- (2,2)node[above]{\tiny $b$};
\draw[thick,double,densely dotted,myblue] (-.5,2)node[above]{\phantom{\tiny $[b,\!a\!+\!1]$}} to[out=-90,in=150] (1,1);
\draw[thick,double,densely dotted,myblue,-to]  (1,1) to[out=-90,in=90] (-.5,-2)node[below,xshift=-7pt]{\tiny $[b,\!a\!+\!1]$};
\draw[thick,double,densely dotted,RawSienna] (.5,2)node[above]{\phantom{\tiny $[b,\!a\!+\!1]$}} to[out=-90,in=90] (1,1);
\draw[thick,double,densely dotted,RawSienna,-to]  (1,1) to[out=-50,in=90] (.5,-2)node[below, xshift=7pt]{\tiny $[b\!-\!1\!,\!a]$}; 
}}\endxy
=
\xy (0,0)*{
\tikzdiagc[scale=0.6,xscale=-1,yscale=-1]{
\draw[ultra thick,blue] (-2,2)node[below]{\tiny $c$} --  (2,-2);
\draw[ultra thick,Orchid] (-2,-2)node[above]{\tiny $b$} -- (1,1);
\draw[ultra thick,PineGreen] (1,1) -- (2,2)node[below]{\tiny $a$};
\draw[thick,double,densely dotted,RawSienna,to-] (-.5,2)node[below, xshift=7pt]{\tiny $[b\!-\!1\!,\!a]$} to[out=-90,in=150] (1,1);
\draw[thick,double,densely dotted,RawSienna]  (1,1) to[out=-90,in=90] (-.5,-2)node[above]{\phantom{\tiny $[a,b]$}};
\draw[thick,double,densely dotted,myblue,to-] (.5,2)node[below, xshift=-7pt]{\tiny $[b,\!a\!+\!1]$} to[out=-90,in=90] (1,1);
\draw[thick,double,densely dotted,myblue]  (1,1) to[out=-50,in=90] (.5,-2); 
}}\endxy
\end{equation}

\begin{equation}\label{eq:reidtwofourv}
\xy (0,0)*{
\tikzdiagc[yscale=2.1,xscale=1.1]{
\draw[ultra thick,mygreen] (-.5,.5)node[left]{\tiny $c\!+\!1$} to (1,.5);
\draw[ultra thick,orange]  (1,.5) to (1.5,.5)node[right]{\tiny $c$};
\draw[thick,double,densely dotted,myblue] (1,0)node[below]{\tiny $[b,a]$} .. controls (1,.15) and  (.7,.24) .. (.5,.29);
\draw[thick,double,densely dotted,myblue,-to] (0.85,.17) --  (.75,.21);
\draw[thick,double,densely dotted,myred] (.5,.29) .. controls (-.1,.4) and (-.1,.6) .. (.5,.71);
\draw[thick,double,densely dotted,myblue,to-] (1,1)node[above]{\tiny $[b,a]$} .. controls (1,.85) and  (.7,.74) .. (.5,.71) ;
\draw[ultra thick,black,-to] (0,0) ..controls (0,.35) and (1,.25) .. (1,.5) ..controls (1,.75) and (0,.65) .. (0,1);
  }}\endxy
= \ \ 
\xy (0,0)*{
  \tikzdiagc[yscale=2.1,xscale=1.1]{
\draw[ultra thick,mygreen] (-.5,.5)node[left]{\tiny $c\!+\!1$} to (0,.5);
\draw[ultra thick,orange]  (0,.5) to (1.5,.5)node[right]{\tiny $c$};
\draw[ultra thick,black,-to] (0,0) -- (0,1);
\draw[thick,double,densely dotted,myblue,-to] (1,0)node[below]{\tiny $[b,a]$} -- (1,1)node[above]{\phantom{\tiny $[b,a]$}};
}}\endxy
\end{equation}
in $K^b(\Sext)$. 
\end{lem}  
\begin{proof}
We use the hom \& dot trick again and only prove \eqref{eg:sixv-reidthree}, \eqref{eq:reidtwofourv} being similar. 
As in the proof of~\autoref{l:commstr}, first apply cups and caps to get two degree preserving morphisms from 
$\rT_{[b-1,a]} \rT_{[b,a+1]} \rB_c \rB_a \rT_{[b,a+1]}^{-1}\rT_{[b-1,a]}^{-1}$ to $\rB_b\rB_c$, then use the homotopy equivalences 
\begin{eqnarray*}
\rT_{[b-1,a]} \rT_{[b,a+1]} \rB_c \rB_a \rT_{[b,a+1]}^{-1}\rT_{[b-1,a]}^{-1} & \simeq &  
\rT_{[b-1,a]} \rT_{[b,a+1]} \rB_a \rT_{[b,a+1]}^{-1}\rT_{[b-1,a]}^{-1} \rB_c \\
&\simeq &\rT_{[b-1,a]} \rT_{[b-1,a]}^{-1} \rB_b \rT_{[b-1,a]}\rT_{[b-1,a]}^{-1} \rB_c \\
&\simeq & \rB_{b} \rB_{c}
\end{eqnarray*}
and the fact that $\rB_b \rB_c\cong \rB_{bc}$, which implies that $\hom_{\Sext_n}(\rB_{bc}, \rB_{bc})$ is one-dimensional by Soergel's hom formula in~\eqref{eq:Soergelhomformula}. 
We attach dots to the upper boundary points labeled $b$ and $c$ of the two diagrams in \eqref{eg:sixv-reidthree} and show that the resulting diagrams are equal using~\eqref{eq:sixcommdot} and \eqref{eq:onecolorsecond}.
\end{proof}

%
%

\input{sections/catembedding.tex}
\section{Parabolic induction: an example}
\subsection{The decategorified story}\label{sec:exdecatstory}
In this section, we will work out an explicit example of an induced triangulated birepresentation $\mathbf{W}$ of $\Sext_2$ and its wide finitary cover $\mathbf{U}$. Below, we will start with the decategorified story, before we move on to the categorification, after a general intermezzo on completions of additive categories. The categorified story is divided into two parts: one dedicated to $\mathbf{U}$ and another to $\mathbf{W}$. Both birepresentations are defined using a certain algebra object $\rY$ in a completion of $K^b(\Sext_2)$, but the construction of $\mathbf{W}$ is much less straightforward than that of $\mathbf{U}$, as the reader will see. 
\subsubsection{The Hecke algebra $\eah{2}$}
Recall that $\widehat{\Sy}_2$ is generated by two simple reflections $s_0$ and $s_1$ which are only subject to the quadratic relations 
$s_0^2=e=s_1^2$. Any element in $\widehat{\Sy}_2$ has therefore a unique rex, which can be identified with an alternating binary sequence. 
To work with those, we have to introduce some notation. For $m\in \mathbb{Z}_{>0}$ and $i\in\{0,1\}$, let $\raltseq{m}{i}$ and $\laltseq{m}{i}$ denote the alternating binary sequences of length $m$ ending with $i$ and starting with $i$, respectively.  If $\raltseq{m}{j}=\laltseq{m}{i}$ for certain $m\in \mathbb{Z}$ and $i,j\in \{0,1\}$ (we call this a {\em parity condition}), then we write $\lraltseq{m}{i}{j}$ for that 
alternating sequence to make the computations below easier to read. For example, 
$\raltseq{4}{0}=\laltseq{4}{1}=\lraltseq{4}{1}{0}=1010$. By convention, we put $\raltseq{0}{i}=\laltseq{0}{i}=\emptyset$, for $i\in \{0,1\}$, 
which of course corresponds to the neutral element $e\in \widehat{\Sy}_2$. 

The extended affine symmetric group $\widehat{\Sy}_2^{\mathrm{ext}}$ has an extra generator $\rho$ of infinite order, and two extra relations: $\rho s_0\rho^{-1}=s_1$ and $\rho s_1 \rho^{-1} =s_0$. 

The extended affine Hecke algebra $\eah{2}$ is generated by $T_0, T_1$ and $\rho^{\pm 1}$, subject to only the first relation in~\eqref{eq:affHeckeSnrels} and all relations in~\eqref{eq:affHeckeRrels}. 
In this case, it is easy to express the KL basis elements $b_w$ in terms of the regular basis elements $T_u$ explicitly:
\begin{equation}\label{eq:KL-regular}
b_w=\sum_{u\preceq w} q^{\ell(w)-\ell(u)}T_u,
\end{equation}
where $\preceq$ denotes the Bruhat order on $\widehat{\Sy}_2$. Note that $u \preceq w$ holds iff the rex for $u$ is an alternating binary subsequence of the rex for $w$ and it is easy to invert this change of basis:
\begin{equation}\label{eq:regular-KL}
T_w=\sum_{u\preceq w} (-q)^{\ell(w)-\ell(u)}b_u.
\end{equation}
The multiplication constants w.r.t. the KL basis are also easy to compute in this case and are all equal to 0, 1, 2 or $[2]=q+q^{-1}$. For $m,n\in \mathbb{Z}_{\geq 1}$ (assuming that the parity conditions are met in each line below), we have 
\begin{equation}\label{eq:multconst}
\begin{gathered}
b_{\lraltseq{m}{1}{1}} b_{\lraltseq{n}{1}{1}}= [2]\left(b_{\lraltseq{m+n-1}{1}{1}} + b_{\lraltseq{m+n-3}{1}{1}}+\ldots + b_{\lraltseq{\vert m-n\vert +3}{1}{1}} + b_{\lraltseq{\vert m-n\vert +1}{1}{1}}  \right);\\
b_{\lraltseq{m}{0}{1}} b_{\lraltseq{n}{1}{1}}= [2]\left(b_{\lraltseq{m+n-1}{0}{1}} + b_{\lraltseq{m+n-3}{0}{1}}+\ldots + b_{\lraltseq{\vert m-n\vert +3}{0}{1}} + b_{\lraltseq{\vert m-n\vert +1}{0}{1}}  \right);\\
b_{\lraltseq{m}{1}{1}} b_{\lraltseq{n}{1}{0}}= [2]\left(b_{\lraltseq{m+n-1}{1}{0}} + b_{\lraltseq{m+n-3}{1}{0}}+\ldots + b_{\lraltseq{\vert m-n\vert +3}{1}{0}} + b_{\lraltseq{\vert m-n\vert +1}{1}{0}}  \right);\\
b_{\lraltseq{m}{1}{0}} b_{\lraltseq{n}{0}{1}}= [2]\left(b_{\lraltseq{m+n-1}{1}{1}} + b_{\lraltseq{m+n-3}{1}{1}}+\ldots + b_{\lraltseq{\vert m-n\vert +3}{1}{1}} + b_{\lraltseq{\vert m-n\vert+1}{1}{1}}  \right);\\
b_{\lraltseq{m}{1}{1}} b_{\lraltseq{n}{0}{1}}= b_{\lraltseq{m+n}{1}{1}} + 2b_{\lraltseq{m+n-2}{1}{1}}+\ldots + 2b_{\lraltseq{\vert m-n\vert +2}{1}{1}}+ b_{\lraltseq{\vert m-n\vert}{1}{1}};\\ 
b_{\lraltseq{m}{1}{0}} b_{\lraltseq{n}{1}{1}}= b_{\lraltseq{m+n}{1}{1}} + 2b_{\lraltseq{m+n-2}{1}{1}}+\ldots + 2b_{\lraltseq{\vert m-n\vert +2}{1}{1}}+ b_{\lraltseq{\vert m-n\vert}{1}{1}};\\ 
b_{\lraltseq{m}{0}{1}} b_{\lraltseq{n}{0}{1}}= 
\begin{cases}
b_{\lraltseq{2m}{0}{1}} + 2b_{\lraltseq{2m-2}{0}{1}}+\ldots + 2b_{\lraltseq{2}{0}{1}},&\text{if}\; m=n \\
b_{\lraltseq{m+n}{0}{1}} + 2b_{\lraltseq{m+n-2}{0}{1}}+\ldots + 2b_{\lraltseq{\vert m-n\vert +2}{0}{1}}+ b_{\lraltseq{\vert m-n\vert}{0}{1}},&\text{if}\; m\ne n 
\end{cases}
;\\
b_{\lraltseq{m}{1}{1}} b_{\lraltseq{n}{0}{0}}= 
\begin{cases}
b_{\lraltseq{2m}{1}{0}} + 2b_{\lraltseq{2m-2}{1}{0}}+\ldots + 2b_{\lraltseq{2}{1}{0}},&\text{if}\; m=n \\
b_{\lraltseq{m+n}{1}{0}} + 2b_{\lraltseq{m+n-2}{1}{0}}+\ldots + 2b_{\lraltseq{\vert m-n\vert +2}{1}{0}}+ b_{\lraltseq{\vert m-n\vert}{1}{0}},&\text{if}\; m\ne n 
\end{cases};
\end{gathered}
\end{equation}
and the analogous equations with $0$ and $1$ swapped. These formulas are certainly known to experts, but we couldn't find an explicit reference for all of them in the literature (the first formula can be found in \cite[Proposition 7.7(a)]{Lusztig2003}). Let us, therefore, briefly explain how we obtained them. There is a close relation between 
$\eah{2}$ and a two-colored version of the Grothendieck ring of finite-dimensional representations (of type I) of quantum $\mathfrak{sl}_2$, 
see~\cite[Section 2.2.1]{elias2016}. In the first four cases above, when the last element of the left alternating binary sequence and 
the first element of the right alternating binary sequence are equal, this correspondence is given by the bijection 
\[ 
b_{\lraltseq{m}{i}{j}}\longleftrightarrow [2][V_{m-1}],
\]
where $[V_{m-1}]$ is the Grothendieck class of the (essentially unique) $m$-dimensional irreducible representation of quantum $\mathfrak{sl}_2$. Note that 
the Grothendieck ring over $\mathbb{Z}$ must be tensored with $\mathbb{Z}[q,q^{-1}]$ for multiplication by $[2]$ to make sense. In the first case, for example, we get 
\begin{gather*}
b_{\lraltseq{m}{1}{1}} b_{\lraltseq{n}{1}{1}} \longleftrightarrow \\
[2]^2 [V_{m-1}][V_{n-1}]= \\
[2]^2 \left([V_{m+n-2}] + [V_{m+n-4}] + 
\ldots + [V_{\vert m-n\vert +2}]+ [V_{\vert m-n\vert}]\right) \longleftrightarrow \\
[2]\left(b_{\lraltseq{m+n-1}{1}{1}} + b_{\lraltseq{m+n-3}{1}{1}}+\ldots + b_{\lraltseq{\vert m-n\vert +3}{1}{1}} + b_{\lraltseq{\vert m-n\vert +1}{1}{1}}  \right),
\end{gather*}
where we have used the well-known Clebsch-Gordan rule for the decomposition of the tensor product of two finite dimensional irreducible representations of quantum $\mathfrak{sl}_2$. It is well-known that $b_i b_{\lraltseq{m}{i}{j}}=b_{\lraltseq{m}{i}{j}} b_j = [2] b_{\lraltseq{m}{i}{j}}$, for $m\in \mathbb{Z}_{\geq 1}$ and $i,j\in \{0,1\}$ satisfying the parity condition. The remaining cases above can be reduced to the first four cases (sometimes with $0$ and $1$ swapped) by that basic multiplication rule and a simple trick, assuming (just for the sake of this trick) that we work over 
$\mathbb{Q}(q)$ for example. Suffice it to give an illustrative example:
\[
b_{\lraltseq{m}{1}{1}} b_{\lraltseq{n}{0}{1}}=\frac{1}{[2]} (b_{\lraltseq{m}{1}{1}} b_1) b_{\lraltseq{n}{0}{1}}= 
\frac{1}{[2]} b_{\lraltseq{m}{1}{1}} (b_1 b_{\lraltseq{n}{0}{1}})=
\frac{1}{[2]} b_{\lraltseq{m}{1}{1}} (b_{\lraltseq{n+1}{1}{1}}+ b_{\lraltseq{n-1}{1}{1}}),
\]
where we use the equality $b_1 b_{\lraltseq{n}{0}{1}}=b_{\lraltseq{n+1}{1}{1}}+ b_{\lraltseq{n-1}{1}{1}}$, which is the special case of the fifth equation in \eqref{eq:multconst} for $m=1$.
\smallskip

Recall the standard trace $\epsilon\colon \eah{2}\to \mathbb{Z}[q,q^{-1}]$ from \eqref{eq:standardtrace} and 
the associated $q$-sesquilinear form $(-,-)$ on $\eah{2}$ from~\eqref{eq:sesquilinform}. The formula in \eqref{eq:KL-regular} implies that 
\[
\epsilon(b_w)=q^{\ell(w)} 
\]
for any $w\in \widehat{\Sy}_n$. Note further that $\omega(b_{\lraltseq{m}{i}{j}})=b_{\lraltseq{m}{j}{i}}$, for $m\in \mathbb{Z}_{\geq 0}$ and $i,j\in \{0,1\}$. The following lemma, which is an immediate consequence of the above, is a more detailed version of \autoref{thm:asymportho} in 
this particular case and will be needed in~\autoref{sec:excatstory}.

\begin{lem}\label{lem:innerproduct}
For any $r,s\in \mathbb{Z}$ and $i,j,k,l\in \{0,1\}$, we have $(\rho^r b_{\lraltseq{m}{i}{j}}, \rho^s b_{\lraltseq{n}{k}{l}})=\delta_{r,s} (b_{\lraltseq{m}{i}{j}}, b_{\lraltseq{n}{k}{l}})$ and 
\[
(b_{\lraltseq{m}{i}{j}}, b_{\lraltseq{n}{k}{l}})=
\begin{cases}
q^{2}p(q), & \text{if}\;\, m=n\; \wedge\; i\ne k\; \wedge\; j\ne l;\\
q^{\vert m-n\vert}p'(q) & \text{else},  
\end{cases}
\]
where $p(q), p'(q)\in \mathbb{Z}_{\geq 0}[q]$ have constant terms equal to two and to one, respectively.  
\end{lem}

\subsubsection{Two representations of $\eah{2}$}
Let $V=\mathrm{Span}\{v\}$ be the trivial one-dimensional $\eah{1}$-module, defined by 
\[
\rho(v)=v,
\]
and let  
\[
W:=V\odot V,
\]
the Zelevinsky tensor product of $V$ with itself. 
Recall from~\autoref{sec:intro} that  
\[
V\odot V=\mathrm{Ind}_{\eah{1}\otimes \eah{1}}^{\eah{2}}(V\otimes V)=
\eah{2} \otimes_{\eah{1}\otimes \eah{1}}(V\otimes V).
\]
Recall also that 
\[
\psi_{1,1}(\rho_L\otimes 1)=\rho T_1,\quad \psi_{1,1}(1\otimes \rho_R)=T_1^{-1}\rho, \quad 
\psi_{1,1}(\rho_L\otimes \rho_R)=\rho^2,
\]
where $\rho\in \eah{2}$.
Thus, for any $r,s\in \mathbb{Z}$, we have 
\[
\psi_{1,1}(\rho_L^r\otimes \rho_R^s)=(\rho T_1)^r(T_1^{-1}\rho)^s.
\]
This implies that $W$ is two-dimensional and that $\{w,w'\}$, where 
\begin{equation}\label{eq:basisW}
w:=1\otimes v \otimes v\quad\text{and}\quad w':=\rho \otimes v \otimes v,
\end{equation}
is a basis of $W$, on which the action of $\eah{2}$ is given by 
the matrices 
\begin{equation}\label{eq:example0}
[\rho]=
\begin{pmatrix}
0 & 1 \\
1 & 0
\end{pmatrix}, 
\quad 
[T_1]= 
\begin{pmatrix}
0 & 1 \\
1 & q^{-1}-q
\end{pmatrix},
\quad 
[T_0]=[\rho T_0 \rho^{-1}]= 
\begin{pmatrix}
q^{-1}-q & 1 \\
1 & 0
\end{pmatrix}.
\end{equation}
On the KL-basis, the action thus becomes 
\begin{equation}\label{eq:example1}
[b_1]=[T_1+q]=
\begin{pmatrix}
q & 1 \\
1 & q^{-1}
\end{pmatrix},
\quad 
[b_0]=[T_0+q]= 
\begin{pmatrix}
q^{-1} & 1 \\
1 & q
\end{pmatrix}.
\end{equation}
Using the rules for multiplication w.r.t. the KL-basis, we see that the action matrices satisfy:
\[
[b_{\raltseq{2r}{0}}]=r[b_1][b_0],\; [b_{\raltseq{2r+1}{1}}]=(2r+1)[b_1],\; [b_{\raltseq{2r}{1}}]=r[b_0][b_1],\;  
[b_{\raltseq{2r+1}{0}}]=(2r+1)[b_0], \; r\in \mathbb{Z}_{\geq 1}.
\]

Recall that $q$ is a formal invertible parameter and let ${-}^{\mathbb{C}(q)}:={-}\otimes_{\mathbb{Z}[q,q^{-1}]}\mathbb{C}(q)$. The following result is well-known, see e.g. \cite[Theorem 1]{LNTh}. 

\begin{prop}\label{prop:simplicityW}
The $\eahCq{2}$-module $W^{\mathbb{C}(q)}$ is simple. 
\end{prop}

There is also an infinite dimensional $\eah{2}$-module $U$ which is closely related to $W$ and whose categorification 
will be important below. By definition, 
\[
U:=\eah{2}/I,
\]
the quotient of the regular left $\eah{2}$-module by the left ideal $I$ generated by 
\begin{equation}\label{eq:I}
\rho^2- 1\quad \text{and}\quad b_1\rho - q^{-1} b_1. 
\end{equation}

\begin{rem}
We stress that $I$ is only a left ideal and not a two-sided ideal. For $\rho^2 -1$ this distinction is irrelevant, because $\rho^2$ and $1$ both belong to the center of $\eah{2}$, but for $b_1\rho - q^{-1}b_1$ it matters. For example, the element $(1-q^{-2})b_1$ belongs to the two-sided ideal generated by the elements in \eqref{eq:I} because 
\[
(1-q^{-2})b_1 = (b_1\rho -q^{-1} b_1)(\rho+ q^{-1}) - b_1(\rho^2 - 1),  
\]
but it does not belong to $I$.
\end{rem}

\begin{prop}\label{prop:basisofU}
Let $\pi_U\colon \eah{2}\to U$ be the natural projection and define $u_k:=\pi_U(b_{\raltseq{k}{1}})$ and $u'_k:=\pi_U(\rho b_{\raltseq{k}{1}})$, for 
all $k\in \mathbb{Z}_{\geq 0}$. Then $\{u_k, u'_k\mid k\in \mathbb{Z}_{\geq 0}\}$ is a basis of $U$ and the action of $\eah{2}$ on this basis is determined by 
\begin{equation}\label{eq:basisofU}
\begin{gathered}
\rho u_{k-1} = u'_{k-1}, \quad \rho u'_{k-1} = u_{k-1} \\
b_1 u_0= u_1,\quad b_1 u_{2k} = u_{2k+1} + u_{2k-1}, \quad b_1 u_{2k-1} = [2] u_{2k-1}, \\
b_0 u_0 = q^{-1} u'_1,\quad b_0 u_1 = u_2, \quad b_0 u_{2k} = [2] u_{2k},\quad 
b_0 u_{2k+1} = u_{2k+2} + u_{2k}, \\
b_1 u'_0 = q^{-1} u_1, \quad b_1 u'_1 = u'_2,\quad b_1 u'_{2k} = [2] u'_{2k}, \quad b_1 u'_{2k+1} = u'_{2k+2} + u'_{2k},\\
b_0 u'_0 = u'_1, \quad b_0 u'_{2k} = u'_{2k+1} + u'_{2k-1},\quad 
b_0 u'_{2k-1} = [2] u'_{2k-1},
\end{gathered}
\end{equation}
for $k\in \mathbb{Z}_{\geq 1}$.
\end{prop}
\begin{proof}
To show that $\{u_k, u'_k\mid k\in \mathbb{Z}_{\geq 0}\}$ spans $U$ is easy. The KL-basis of $\eah{2}/\langle \rho^2 -1 \rangle$, 
where $\langle \rho^2 - 1\rangle$ is the (left) ideal generated by $\rho^2-1$, is given by 
$\{\rho^m b_{\raltseq{k}{1}}, \rho^m b_{\raltseq{k}{0}}\mid m\in \{0,1\}, k\in \mathbb{Z}_{\geq 0}\}$. The images of those elements under $\pi_U$, therefore,  
span $U$. By induction, it's easy to show that $b_{\raltseq{k}{1}}\rho - q^{-1} b_{\raltseq{k}{1}}\in I$ for all $k\in \mathbb{Z}_{>0}$. For $k=2$, we have 
\[
b_{01}\rho -q^{-1} b_{01}=b_0(b_1\rho - q^{-1} b_1)\in I
\]
For $k>2$, induction yields  
\[
b_{\raltseq{k}{1}}\rho - q^{-1} b_{\raltseq{k}{1}} = 
\begin{cases}
b_0 (b_{\raltseq{k-1}{1}}\rho - q^{-1} b_{\raltseq{k-1}{1}}) - (b_{\raltseq{k-2}{1}}\rho - q^{-1} b_{\raltseq{k-2}{1}})\in I, & \text{if $k$ is even};\\
b_1 (b_{\raltseq{k-1}{1}}\rho - q^{-1} b_{\raltseq{k-1}{1}}) - (b_{\raltseq{k-2}{1}}\rho - q^{-1} b_{\raltseq{k-2}{1}})\in I, & \text{if $k$ is odd}.
\end{cases}
\]
This shows that $\rho b_{\raltseq{k}{0}}= b_{\raltseq{k}{1}}\rho =q^{-1} b_{\raltseq{k}{1}}$ in $U$, for all $k\in \mathbb{Z}_{\geq 0}$. This in turn implies that $b_{\raltseq{k}{0}}=\rho^2 b_{\raltseq{k}{0}}=q^{-1} \rho b_{\raltseq{k}{1}}$, for all $k\in \mathbb{Z}_{\geq 0}$, which completes 
the proof that $\{u_k, u'_k\mid k\in \mathbb{Z}_{\geq 0}\}$ spans $U$.

It remains to prove that the elements $u_k, u'_k$, for $k\in \mathbb{Z}_{\geq 0}$, are all linearly independent in $U$. Suppose that 
\[
\lambda_{m_1} u_{m_1} + \ldots + \lambda_{m_r} u_{m_r} + \mu_{n_1} u'_{n_1} + \ldots + \mu_{n_s} u'_{n_s}=0\;\, \text{in $U$},
\]
for some $r,s\in \mathbb{Z}_{\geq 0}$ and some $\mu_{m_1},\ldots, 
\mu_{m_r},\nu_{n_1}\ldots, \nu_{n_s}\in \Bbbk$. Then 
\begin{equation}\label{eq:linindep}
\lambda_{m_1} b_{\raltseq{m_1}{1}} + \ldots + \lambda_{m_r}  b_{\raltseq{m_r}{1}} + \mu_{n_1} \rho  b_{\raltseq{n_1}{1}} + \ldots + \mu_{n_s} 
\rho  b_{\raltseq{n_s}{1}}\in I.
\end{equation}
Since $b_1\rho = \rho b_0$, any non-zero element in $I$ will be a linear 
combination of KL basis elements containing at least a non-zero multiple of 
$\rho b_{\raltseq{k}{0}}$, for some $k\in \mathbb{Z}_{> 0}$, or 
a non-zero multiple of $\rho^2 b_{\raltseq{k}{i}}$,  
for some $k\in \mathbb{Z}_{\geq 0}$ and $i\in \{0,1\}$. Either way, linear independence of the KL basis elements in $\eah{2}$ implies that \eqref{eq:linindep} can only hold if $\lambda_{m_1} b_{\raltseq{m_1}{1}} + \ldots + \lambda_{m_r}  b_{\raltseq{m_r}{1}} + \mu_{n_1} \rho  b_{\raltseq{n_1}{1}} + \ldots + \mu_{n_s} 
\rho  b_{\raltseq{n_s}{1}}=0$ and, therefore, 
that $\lambda_{m_i}=\mu_{n_j}=0$ for all $i=1,\ldots, r$ and $j=1,\ldots, s$.
This completes the proof that the elements $u_k, u'_k$, for $k\in \mathbb{Z}_{\geq 0}$, are all linearly independent in $U$.

Finally, it is easy to determine the action of $\rho, b_0$ and $b_1$ on this basis using the KL multiplication rules, because $\pi_U$ is a morphism of 
$\eah{2}$-modules. 
\end{proof}

Let $J\subset U$ be the $\eah{2}$-submodule generated by $T_1u_0-\rho u_0$ and let $\overline{u}\in U/J$ be the image of $u\in U$ 
under the natural projection $U\to U/J$. We have 
\begin{gather*}
\overline{T_1u_0}=\overline{\rho u_0}=\overline{u'_0}, \\
\overline{T_1 u'_0}= \overline{T_1^2 u_0}=\overline{(q^{-1}-q)T_1u_0 + u_0}=\overline{(q^{-1}-q)u'_0+ u_0}.
\end{gather*}
\begin{prop}\label{prop:projection} There is a surjective morphism of $\eah{2}$-modules 
\[
\pi^U_W\colon U\to W
\]
and $\ker(\pi^U_W)=J$.
\end{prop}
\begin{proof}
Define $\pi^U_W\colon U\to W$ by 
\[
\pi^U_W(u_0):=w,
\]
extending it to the whole of $U$ by the requirement that it be a morphism of $\eah{2}$-modules.  
Note that the matrices in \eqref{eq:example1} imply that 
\begin{gather*}
\pi^U_W(u'_0)=\pi^U_W(\rho u_0)=\rho \pi^U_W(u_0)=\rho w=w'; \\
\pi^U_W(b_1\rho u_0)=b_1\pi^U_W(u'_0)=b_1 w'=w+q^{-1}w'=q^{-1}b_1w=q^{-1}\pi^U_W(b_1 u_0),
\end{gather*}
so $\pi^U_W$ is well-defined. 

The fact that $\pi^U_W(u_0)=w$ and $\pi^U_W(u'_0)=w'$ proves that $\pi^U_W$ is surjective. 

It remains to show that $\ker(\pi^U_W)=J$. The matrices in 
\eqref{eq:example1} show that $\rho T_1 w= w$, which implies that  
\[
\pi^U_W(\rho T_1 u_0)=\rho T_1 \pi^U_W(u_0)=\rho T_1 w= w= \pi^U_W(u_0),  
\]
so $J\subseteq \ker(\pi^U_W)$. To prove that this inclusion is an equality we will show that  
$\dim_{\Bbbk}(U/J)=2$, which amounts to showing that $\overline{u_k}, \overline{u'_k}\in 
\mathrm{span}\{\overline{u_0}, \overline{u'_0}\}$ for all $k\in \mathbb{Z}_{\geq 0}$. We will prove this by induction. 
Recall that $b_i=T_i+q$, for $i\in \{0,1\}$. Then, for the base step, we have 
\begin{gather*}
\overline{u_1} = \overline{b_1 u_0}=\overline{T_1 u_0 + q u_0}=\overline{\rho u_0 + q u_0}= 
\overline{u'_0+qu_0}, \\
\overline{u'_1}= \overline{\rho u_1}=\overline{\rho u'_0+q \rho u_0}=\overline{u_0+qu'_0},\\
\overline{u_2}=\overline{b_0 u_1} = \overline{b_0(u'_0+qu_0)} = 2 \overline{u'_1}= 2(\overline{u_0+qu'_0}),\\  
\overline{u'_2}= \overline{\rho u_2}= 2\overline{u_1}= 2(\overline{u'_0+qu_0}).   
\end{gather*}
For the inductive step, let $k>2$ and assume that 
that $\overline{u_m}, \overline{u'_m}\in \mathrm{span}\{\overline{u_0}, \overline{u'_0}\}$ for all $0\leq m< k$. If $k>2$ is odd, then 
\[
\overline{u_k}=\overline{b_1 u_{k-1}-u_{k-2}}.
\]
By induction, the elements $\overline{u_{k-1}}, \overline{u_{k-2}}$ belong to $\mathrm{span}\{\overline{u_0}, \overline{u'_0}\}$. Since 
$\overline{b_1 u_0}=\overline{u_1}=\overline{u'_0+qu_0}$ and $\overline{b_1 u'_0}=q^{-1} \overline{u_1}=\overline{q^{-1}u'_0+ u_0}$, we conclude 
that $\overline{u_k}\in \mathrm{span}\{\overline{u_0}, \overline{u'_0}\}$. If $k>2$ is even, then 
\[
\overline{u_k}=\overline{b_0 u_{k-1} - u_{k-2}}.
\]
By induction again, the elements $\overline{u_{k-1}}, \overline{u_{k-2}}$ belong to $\mathrm{span}\{\overline{u_0}, \overline{u'_0}\}$. Since 
$\overline{b_0 u_0}=\overline{q^{-1} u'_1}=\overline{q^{-1} u_0+u'_0}$ and $\overline{b_0 u'_0}=\overline{u'_1}=\overline{u_0+qu'_0}$, we conclude 
that $\overline{u_k}\in \mathrm{span}\{\overline{u_0}, \overline{u'_0}\}$. This completes the proof that $u_k$ belongs to $\mathrm{span}\{\overline{u_0}, \overline{u'_0}\}$ for all $k\in \mathbb{Z}_{\geq 0}$. Using similar arguments, one can show that the same also holds for $u'_k$, for all $k\in \mathbb{Z}_{\geq 0}$, which concludes the proof. 
\end{proof}

Just to fix notation, let $\pi_W\colon \eah{2}\to W$ be the projection of $\eah{2}$-modules defined by $\pi_W(x):=xw$, for any $x\in \eah{2}$. Then there is a commutative triangle 
\begin{equation}\label{eq:projectiontriangle}
\begin{tikzcd}
\eah{2} \arrow{r}{\pi_U}  \arrow{rd}{\pi_W}  
  & U \arrow{d}{\pi^U_W} \\
    & W
\end{tikzcd}
\end{equation}

\begin{rem} The infinite dimensional representation $U$ of $\eah{2}$ is the decategorification of a wide finitary birepresentation $\mathbf{U}$ of $\Sext_2$, which appears naturally as the wide finitary cover of the induced birepresentation $\mathbf{W}$ in \autoref{sec:excatstory}. 
However, we do not know if $U$ also appears naturally in any approach to the representation theory of $\eah{2}$ that does not use categorification. Note that this contrasts with the finitary cover $\widehat{\mathbf{M}}_{r,s}$ ($r,s\in \mathbb{Z}$) of the evaluation birepresentation $\mathbf{M}^{\mathcal{E}v_{r,s}}$ of $\Sext_n$ in \cite[Corollary 6.5]{mmv-evalfunctor}, 
which decategorifies to a Graham-Lehrer cell module of $\eah{n}$, see \cite[Remark  6.6]{mmv-evalfunctor}.
\end{rem}


\subsection{Completions of linear additive categories}\label{sec:completions}

In this subsection, we lay the groundwork for studying the application of our embedding to the induction of birepresentations. Roughly speaking, birepresentations over an additive bicategory correspond to algebras in some completion of the latter and one can define induction from one bicategory to another by pushing forward the algebras through an appropriate embedding. In our case, the situation is rather delicate as we move between additive and triangulated bicategories, so in this section we carefully study the completions in which our algebras live, both in the additive and triangulated world.
Since this section is more general, we will work over an arbitrary field and with bicategories rather than monoidal categories.

\subsubsection{Coproduct completions of additive $\Bbbk$-linear categories} Let $\cC$ be an additive $\Bbbk$-linear category. 
Recall that the \emph{additive closure} $\add(\mathcal{A})$ of a full subcategory $\mathcal{A}$ of $\cC$ is the closure under direct sums and direct summands in $\cC$.

We note that $\Vect_{\Bbbk}$ is complete and cocomplete with respect to small weighted (and hence, in particular, conical) limits and colimits. This follows from \cite[Corollary 7.6.4]{Ri} using \cite[Example 3.7.5]{Ri} and the fact that $\Vect_{\Bbbk}$ is enriched over itself.

This implies that the category of $\Bbbk$-linear functors $\cpl\cC =\F un_\Bbbk(\cC^{op},\Vect_{\Bbbk})$,  is also (co)complete under weighted (co)limits by \cite[Section 3.3]{Ke} with (co)limits being computed object-wise. As usual, we identify an object $X$ in $\cC$ with their representables $X^\vee=\Hom_{\cC}( -, X)$ under the Yoneda embedding. We denote by $\cC^\vee$ the image of $\cC$ under this embedding (which is, of course, equivalent to $\cC$).

Note that we can realise a (co)product indexed by a set $J$ as the conical (co)limit over the diagram $J \to \cC\colon j \mapsto X_j$. This is a special case of the weighted (co)limit of the $\Vect_{\Bbbk}$-enriched functor $\tilde J \to \cC$ where $\tilde J$ is the free $\Vect_{\Bbbk}$-enriched category on $J$ and the weight $\tilde J\to \Vect_{\Bbbk}$ is just the constant functor $j\mapsto \Bbbk$. For a family of objects $(X_j)_{j\in J}$ in $\cC$, we slightly abuse notation by writing $\coprod_{j\in J}X_j$ for said colimit in $\cpl{\cC}$ rather than the formally correct $\coprod_{j\in J}X_j^\vee$ and likewise for products.

 Since $\cC$ and hence $\cpl \cC$ is additive, in particular has a zero object, there is always a canonical monomorphism 
 $$\nu_{(X_j)_{j\in J}} \colon \coprod_{j\in J} X_j \to \prod_{j\in J } X_j.$$ Explicitly, denoting by $\pi_i \colon \prod_{j\in J } X_j\to X_i$ and $\iota_i\colon X_i \to \coprod_{j\in J } X_j$ the canonical projections and injections, respectively, $\nu$ is defined by the condition that $$\pi_i\nu\iota_k = \begin{cases} \id_{X_i} & \text{if } i=k,\\ 0&\text{otherwise.} \end{cases}$$

We define $\copr{\cC}$ as the full subcategory of $\cpl{\cC}$ where we close $\cC^\vee$ under countable coproducts. 
We then denote by $\cop{\cC}=(\copr{\cC})_e$ its Karoubi envelope.
 Note that both are again additive $\Bbbk$-linear categories.  Moreover, $\cop{\cC}$ is closed under countable coproducts. We observe that, by \cite[Theorem 4.1]{Bre}, $\cop{\cC}$ and 
 $\copr{\cC}$ coincide provided $\cC$ is Krull-Schmidt. 

A $\Bbbk$-linear functor $\rF\colon \cC\to \cD$ between additive $\Bbbk$-linear categories induces $\Bbbk$-linear functors, also denoted by $\rF$, from $\copr{\cC} \to \copr{\cD}$, sending an object $\coprod_{j\in J}X_j$ to $\coprod_{j\in J}\rF(X_j)$, and hence from $\cop{\cC} \to \cop{\cD}$.

The next lemma will be crucial for computing morphism spaces in explicit examples.

\begin{lem}\label{finprod}
Assume $Y=\coprod_{j\in J} Y_j \in \cpl{\cC}$, where each $Y_j\in \cC$. Let $X\in \cC$ and assume that $\Hom_\cC(X,Y_j)= 0$ for all but finitely many $j\in J$. Then $\Hom_{\cpl{\cC}}(X, Y) = \coprod_{j\in J}\Hom_\cC(X,Y_j)$.
\end{lem}

\begin{proof}
Consider the natural monomorphism $$\coprod_{j\in J}\Hom_{\cpl{\cC}}(X,Y_j) \xrightarrow{\kappa} \Hom_{\cpl{\cC}}(X, \coprod_{j\in J} Y_j )$$ and compose this with the monomorphism induced by $\nu = \nu_{(Y_j)_{j\in J}},$ i.e. 
$$\coprod_{j\in J}\Hom_{\cpl{\cC}}(X,Y_j) \xrightarrow{\kappa}  \Hom_{\cpl{\cC}}(X, \coprod_{j\in J} Y_j ) \xrightarrow{\nu\circ-} \Hom_{\cpl{\cC}}(X, \prod_{j\in J} Y_j ) \cong \prod_{j\in J}\Hom_{\cpl{\cC}}(X,  Y_j ) $$
The coproduct in the domain and the product in the codomain are both finite and hence isomorphic, so this monomorphism is an isomorphism. In particular, $\kappa$ is an isomorphism. The lemma now follows from the fact that $\Hom_{\cpl{\cC}}(X,  Y_j ) = \Hom_\cC(X,Y_j)$ by the Yoneda Lemma.
\end{proof}

Given that the (co)product of the morphism spaces in \autoref{finprod} is really finite, we write it as 
\[
\bigoplus_{j\in J}\Hom_{\cpl{\cC}}(X,Y_j).
\]

\subsubsection{Coproduct completions and homotopy categories}

For a $\Bbbk$-linear additive category $\cC$, we can consider the category $C^b(\cC)$ of bounded complexes over $\cC$ and its completion $\copr{C^b(\cC)}$ under countable coproducts, and the Karoubi envelope $\cop{C^b(\cC)}$ of the latter. Similarly, we can consider $C^b(\copr{\cC})$ and $C^b(\cop{\cC})$.

For reference, we record the following lemma, for which we could not find a proof in the literature.

\begin{lem}\label{chidcomp}
For an idempotent complete $\Bbbk$-linear additive category $\cC$, the category $C^b(\cC)$ is idempotent complete.
\end{lem}

\begin{proof}
Assume $e^\bullet$ is an idempotent endomorphism of an object $(X^\bullet, d^\bullet)$  in $C^b(\cC)$. In particular, each $e^i$ is an idempotent endormorphism of $X^i$. Since $\cC$ is idempotent complete, $X^i$ splits into a direct sum of object $X^i_{e}$ and $X^i_{1-e}$. Since $e^\bullet$ commutes with the differential $d^\bullet$, each of these summands is preserved under the differential and the complex $(X^\bullet, d^\bullet)$ splits into a direct sum of complexes $(X_e^\bullet, e^\bullet d^\bullet e^\bullet)$ and $(X^\bullet, (1-e)^\bullet d^\bullet (1-e)^\bullet)$.
\end{proof}

Note that the Yoneda embedding $\cC\xrightarrow{\Upsilon} \cop{\cC}$ extends to an embedding $C^b(\cC)\hookrightarrow C^b(\cop{\cC})$, which gives rise to an embedding $\cop{C^b(\cC)}\hookrightarrow \cop{C^b(\cop{\cC})}$. The following lemma is straightforward, noticing that a countable coproduct of complexes, which each have countable coproducts of objects in each component, is again a complex with countable coproducts in each component.

\begin{lem}
$C^b(\cop{\cC})$ is closed under countable coproducts, thus $C^b(\cop{\cC})\simeq\cop{C^b(\cop{\cC})}$ and we have an embedding $\iota\colon\cop{C^b(\cC)}\hookrightarrow C^b(\cop{\cC})$.
\end{lem}

\begin{proof}
Let $(X_i^\bullet, d_i^\bullet)_{i\in I}$ be a family of objects in $C^b(\cop{\cC})$ and consider the coproduct $X=\coprod_{i\in I}(X_i^\bullet, d_i^\bullet)$ in $\copr{C^b(\cop{\cC})}$. Now define $Y=(Y^\bullet, d^\bullet)$ in $C^b(\cop{\cC})$ as the complex whose $k$th component is given by $\coprod_{i\in I}X_i^k$ and whose differential $d^k$ is given by $\coprod_{i\in I}X_i^k\xrightarrow{(d_i^k)_{i\in I}}\coprod_{i\in I}X_i^{k+1}$. Then clearly, $Y\cong X$, so  $C^b(\cop{\cC})$ is closed under coproducts and $C^b(\cop{\cC})\simeq\copr{C^b(\cop{\cC})}$. Since by \autoref{chidcomp}, $C^b(\cop{\cC})$ is idempotent complete, $C^b(\cop{\cC})\simeq\cop{C^b(\cop{\cC})}$. 
\end{proof}

Denote by $N^b(\cC)\subset C^b(\cC)$ the ideal of nulhomotopic morphisms of complexes, and by $K^b(\cC) = C^b(\cC)/N^b(\cC)$ the bounded homotopy category.

Observe that the fully faithful functor $C^b(\cC)\hookrightarrow C^b(\cop{\cC})$ induced by the Yoneda embedding preserves nullhomotopic maps, and we thus obtain a fully faithful functor $K^b(\cC)\hookrightarrow K^b(\cop{\cC})$. Morover, $\iota$ also preserves nullhomotopic morphisms so we also obtain a fully faithful functor $ \cop{K^b(\cC)} \to K^b(\cop{\cC})$, which we again denote by $\iota$.

\begin{cor}\label{functorextend}
\begin{enumerate}[(a)]
\item We have fully faithful functors $K^b(\cC)\xrightarrow{K^b(\Upsilon)}\cop{K^b(\cC)} \xrightarrow{\iota} K^b(\cop{\cC})$.
\item If $\cD$ is a $\Bbbk$-linear additive category, a $\Bbbk$-linear functor $\Phi\colon \cC\to K^b(\cD)$ induces a $\Bbbk$-linear functor $\cop{\cC}\to K^b(\cop{\cD})$.
\end{enumerate}
\end{cor}

\begin{proof}
The first part is immediate. For the second part, note that $\Phi$ gives rise to a functor $\cop{\cC}\to  \cop{K^b(\cD)}$. Composing this with the inclusion from the first part yields the desired functor $\cop{\cC}\to K^b(\cop{\cD})$.
\end{proof}

\subsubsection{Wide finitary bicategories} Here we recall some of the definitions from \cite{macpherson22a}. Remember that a $2$-category is a strict bicategory, meaning that all coherers 
are identities. We will generally be speaking about bicategories, but some of these happen to be $2$-categories. 

Denote by $\mathfrak{A}_{\Bbbk}$ the $2$-category whose objects are additive $\Bbbk$-linear categories, whose $1$-morphisms are $\Bbbk$-linear functors and whose $2$-morphisms are natural transformations.

A category $\cC$ is {\it wide finitary} if it is an additive $\Bbbk$-linear Krull--Schmidt
category with at most countably many isomorphism classes of indecomposable objects and morphism spaces of at most countable dimension. We denote by $\ind{\cC}$ a set of representatives of isomorphism classes of indecomposable objects in $\cC$.

We denote the $2$-category whose objects are wide finitary categories, whose $1$-morphisms are $\Bbbk$-linear functors, and whose $2$-morphisms are natural transformations by $\mathfrak{A}^{wf}_{\Bbbk}$.

We say that a bicategory $\cC$ is {\it wide finitary} if 
\begin{itemize}
\item it has finitely many objects;
\item $\cC(\ti,\tj)$ is in $\mathfrak{A}^{wf}_{\Bbbk}$ for all $\ti,\tj \in \cC$;
\item horizontal composition is biadditive and $\Bbbk$-linear;
\item the identity $1$-morphism $\one_\ti$ is indecomposable for any $\ti\in\cC$.
\end{itemize}

We use $\circh$, resp.\ $\circv$, for horizontal, respectively vertical, composition in a bicategory.

Recall that if $\cC$ is a bicategory, $\cC^{\mathsf{co},\mathsf{op}}$ is the bicategory obtained by reversing the direction of both $1$- and $2$-morphisms. A wide finitary bicategory $\cC$ is said to be {\it wide fiat} if there exists a weak involution $(-)^*\colon \cC\to \cC^{\mathsf{co},\mathsf{op}}$ such that, for any $1$-morphism $\rF\in \cC(\ti,\tj)$, there are
natural $2$-morphisms $\alpha\colon \rF\rF^* \to\one_\tj, \beta \colon \one_\ti\to \rF^*\rF$ satisfying the usual adjunction triangles, i.e. $(\alpha\circ_{0}\id_{\rF})\circ_1 (\id_\rF \circ_0\beta) = \id_\rF$ and $(\beta\circ_0\id_{\rF^*})\circ_1(\id_{\rF^*}\circ_0\alpha)=\id_{\rF^*}$.

Note that $\Sext_n$ is a wide fiat bicategory, where the fiat structure is defined by rotating the diagrams in Soergel calculus by 180 degrees and the adjunctions $\alpha$ and $\beta$ are defined 
by cups and caps, respectively.

\subsubsection{Completed and triangulated bicategories.}

Let $\cC$ be a wide finitary bicategory. We define the bicategory $\cop{\cC}$ to be the bicategory on the same objects as $\cC$, but with morphism categories $\cop{\cC}(\ti,\tj) =\cop{ \cC(\ti,\tj)}(=\copr{ \cC(\ti,\tj)})$, with horizontal composition given component-wise, i.e. 
$$\coprod_{i\in I}\rX_i\coprod_{j\in J}\rY_j = \coprod_{\substack{i\in I\\j\in J}}\rX_i\rY_j $$
and similarly for 2-morphisms. Observe that $\cop{\cC}$ is locally additive and $\Bbbk$-linear.

\begin{rem}
Note that this naturally embeds into the completion $\cpl{\cC}$ with morphism categories $\cpl{\cC}(\ti,\tj) =\cpl{ \cC(\ti,\tj)}$ and horizontal composition given by Day convolution.
\end{rem}

Likewise, for any locally additive $\Bbbk$-linear (or in particular wide finitary) bicategory $\cC$, we can define $K^b(\cC)$ as having the same objects as $\cC$, morphism categories $K^b(\cC)(\ti,\tj)=K^b(\cC(\ti,\tj))$ and horizontal composition induced by taking the total complex of the tensor product of complexes. This is a triangulated bicategory, which fits within the bicategorical version of the framework considered in \cite[Section 3.1]{LM1}. 

\begin{lem}
If $\cC$ is wide finitary and all $1$-morphisms have finite-dimensional endomorphism rings, then $K^b(\cC)$ is wide finitary.
\end{lem}

\begin{proof}
Since $\cC$ is locally idempotent complete, so is $K^b(\cC)$ by \cite[Theorem~3.4]{Schn}. Since endomorphism rings of $1$-morphisms in $\cC$ are finite-dimensional, the same is true for endomorphism rings of $1$-morphisms in $K^b(\cC)$. Thus $K^b(\cC)$ is Krull-Schmidt by \cite[Cor 4.4]{Kr}. Countability of isomorphism classes of indecomposables is inherited from $\cC$.
\end{proof}

\begin{cor}\label{cor:widefineS}
For any $n\in \mathbb{Z}_{\geq 2}$, the monoidal category (one-object bicategory) $K^b(\Sext_n)$ is wide finitary.
\end{cor}

\subsubsection{Birepresentations}

An {\it additive birepresentation} of a locally additive $\Bbbk$-linear bicategory $\cC$ is a $\Bbbk$-linear pseudofunctor from $\cC$ to  $\mathfrak{A}_{\Bbbk}$.
 A {\it wide finitary birepresentation} of a wide finitary bicategory $\cC$ is a $\Bbbk$-linear pseudofunctor from $\cC$ to $\mathfrak{A}^{wf}_{\Bbbk}$.

We can extend an additive or wide finitary birepresentation $\bfM$ of $\cC$ to an additive birepresentation $\cop{\bfM}$ of  $\cop{\cC}$ by defining the action of $\rF=\coprod_{i\in I}\rF_i \in \cop{ \cC}(\ti,\tj)$ on $X =\coprod_{j\in J}X_j  \in \cop{\bfM}(\ti)$ componentwise,  i.e. 
$$\cop{\bfM}\left(\rF\right)\left(X\right)=\cop{\bfM}\left(\coprod_{i\in I} \rF_i\right)\left(\coprod_{j\in J}X_j\right) = \coprod_{\substack{i\in I\\j\in J}}\bfM(\rF_i)\left(X_j\right) .$$

Similarly, any additive birepresentation of $\bfM$ gives rise to a {\it triangulated} birepresentation $K^b(\bfM)$ of $K^b(\cC)$ where $K^b(\bfM)(\ti) = K^b(\bfM(\ti))$.

\subsubsection{Algebras and birepresentations in $\cop{\cC}$}

Let $\cC$ be a wide finitary bicategory.

Now assume $G$ is an abelian group and $\rY = \coprod_{j\in G} \rY_j\in \cop{\cC}(\ti,\ti)$ is an algebra $1$-morphism  with product $$\bfm\colon \coprod_{i,j\in G} \rY_i\rY_j\to\coprod_{k\in G} \rY_k $$ determined by component morphisms $\bfm_{i,j} = \bfm (\iota_i\circh\iota_j) \colon \rY_i\rY_j \to \rY_{i+j}$, where we identify $\rY_{i+j}$ with the subobject of $\coprod_{k\in G} \rY_k$ given by $\iota_{i+j}(\rY_{i+j})$, and unit $\bfu$ determined by a morphism $\mathbbm{1}_{\ti}\to \rY_0$, where $0$ is the neutral element in $G$.

For each $\tj\in\cC$, we can consider the category of $\rY$-modules in $\cop{\cC}(\ti,\tj)$, denoted by $\rmod_{\cop{\cC(\ti,\tj)}}\rY$, which gives rise to an additive birepresentation $\brmod_{\cop{\cC}}\rY$ of $\cop{\cC}$. Note that each $\rmod_{\cop{\cC(\ti,\tj)}}\rY$ is 
idempotent complete by construction.

By the usual free forgetful adjunction, we have, for $\rF\in \cC(\ti,\ti)$ an isomorphism 
\begin{equation}\label{freeforg}
\begin{array}{rcl}\Hom_{\rmod_{\cop{\cC(\ti,\ti)}}\rY} (\rF \rY, \rY)&\cong &\Hom_{\cop{\cC(\ti,\ti)}} (\rF, \rY)\\
f &\mapsto & f\circv( \id_{\rF}\circh \bfu)\\
\bfm\circv(g\circh \id_\rY) &\mapsfrom & g.
\end{array}
\end{equation}

\begin{lem}\label{lem:freeforgcoord} 
Let $\rY = \coprod_{j\in G} \rY_j\in \cop{\cC}(\ti,\ti)$ be an algebra $1$-morphism as above, $\rF\in \cC(\ti,\ti)$ and $g\in \Hom_{\cop{\cC(\ti,\ti)}} (\rF, \rY)$. Assume that $g$ factors over $\iota_j$, i.e. there exists an $g_j\in \Hom_{{\cC(\ti,\ti)}} (\rF, \rY_j)$ such that $g = \iota_j\circv g_j$. Then under the isomorphism in \eqref{freeforg}, $g$ corresponds to a morphism $f$ defined by 
$$f \circv (\id_\rF\circh \iota_k) = \bfm_{j,k} \circv (g_j \circh\id_{\rY_k})\colon \rF \rY_k \to \rY_{j+k}.$$
\end{lem}

\begin{proof}
This is a direct computation, observing that
\begin{equation*}\begin{split}
f\circv (\id_\rF\circh \iota_k) &= \bfm\circv(g\circh \id_\rY) \circv (\id_\rF\circh \iota_k)\\
&= \bfm\circv((\iota_j\circv g_j)\circh \iota_k)\\
&=\bfm\circv(\iota_j\circh \iota_k)\circv (g_j \circh\id_{\rY_k})\\
&=\bfm_{j,k} \circv (g_j \circh\id_{\rY_k})\\
\end{split}\end{equation*}
so the claim follows.
\end{proof}

\begin{lem}\label{lem:widefineadd}
Let $\cC$ be a wide finitary bicategory and  $\rY\in \cop{\cC(\ti,\ti)}$ an algebra $1$-morphism. Assume that for any $\rG\in \cC(\ti,\ti)$, the morphism space $\Hom_{\cop{\cC(\ti,\ti)}}(\rG, \rY)$ is finite-dimensional. Then the birepresentation $\bfM_\rY$ given by  $\bfM_\rY(\tj) = \add\{\rF \rY\,\vert\, \rF\in \cC(\ti,\tj)\}$ (where the additive closure is taken in $\rmod_{\cop{\cC(\ti,\tj)}}\rY$) is wide finitary.
\end{lem}

\begin{proof}
By \eqref{freeforg}, for each $\rF\in \cC(\ti,\tj)$, the endomorphism ring $\End_{\rmod_{\cop{\cC(\ti,\tj)}}\rY}(\rF\rY)$ is isomorphic to $\Hom_{\cop{\cC(\ti,\ti)}}(\rF^*\rF, \rY)$, which by assumption is finite-dimensional. Since each $\rmod_{\cop{\cC(\ti,\tj)}}\rY$ is idempotent complete, \cite[Cor 4.4]{Kr} implies that each $\bfM(\tj)$ is Krull-Schmidt. Moreover, the number of isomorphism classes of indecomposable objects is countable and morphism spaces are even finite-dimensional, so $\bfM_\rY$ is wide finitary.
\end{proof}


\subsection{The categorified story}\label{sec:excatstory}
In the first part of this section, we will define a graded wide-finitary birepresentation $\mathbf{U}$ of $\Sext_2$ and show that it categorifies the representation $U$ of $\eah{2}$. In the second part, we will define a triangulated birepresentation $\mathbf{W}$ of $\Sext_2$ and conjecture that it categorifies the representation $W$ of $\eah{2}$.

\subsubsection{The categorification of $U$}
Let $\mathbf{V}$ be the rank one finitary $2$-representation of $\Sext_1$ categorifying $V$. The corresponding algebra object 
$\rX\in \Sextdiamond_1$ is given by  
\[
\rX:=\coprod_{r \in \mathbb{Z}} \rB_{\rho}^r,
\]
with multiplication $\mu\colon \rX\rX\to \rX$ defined by the identity on $\rB_{\rho}^k \rB_{\rho}^{r-k}=\rB_{\rho}^r$, for all 
$r,k\in \mathbb{Z}$, and unit $\epsilon\colon \rI\to \rX$ defined as the identity on $\rI=\rB_{\rho}^0$ and zero to $\rB_{\rho}^r$, for all 
$r\ne 0$. Note that $(\rB_{\rho})^k\cong \rB_{\rho^k}$ and $\mathrm{hom}_{\Sext_1}(\rB_{\rho^k}\langle s\rangle, \rB_{\rho^m}\langle t\rangle)\cong \delta_{s,t}\delta_{k,m}\mathbb{R}$, for any $k,m,s,t\in \mathbb{Z}$, where $\delta_{-,-}$ is the Kronecker delta. 

Consider $\rY:=\Psi_{1,1}(\rX\boxtimes \rX)\in K^b(\Sextdiamond_2)$. The fact that $\Psi_{1,1}$ is a monoidal functor 
and the isomorphisms $\rB_{\rho}\boxtimes \rB_{\rho}\cong (\rB_{\rho}\boxtimes \rI)(\rI\boxtimes \rB_{\rho})\cong (\rI\otimes \rB_{\rho})(\rB_{\rho}\boxtimes \rI)$ imply that 
\[
\rY\cong \coprod_{r\in \mathbb{Z}} \coprod_{s\in \mathbb{Z}}(\rB_\rho \rT_1)^r (\rT_1^{-1} \rB_\rho)^s\cong  \coprod_{r\in \mathbb{Z}}\coprod_{s\in \mathbb{Z}}(\rT_1^{-1} \rB_\rho)^s (\rB_\rho \rT_1)^r.
\]
Note that, for any $k,l\in \mathbb{Z}$, there are natural isomorphisms 
\[
(\rB_\rho \rT_1)^k\, (\rT_1^{-1} \rB_\rho)^l\, \rY\cong \rY \cong (\rT_1^{-1} \rB_\rho)^l (\rB_\rho \rT_1)^k\, \rY,
\]
because $(\rB_\rho \rT_1)^k$ and $ (\rT_1^{-1} \rB_\rho)^l$ commute and 
\[
(\rB_\rho \rT_1)^k\, (\rT_1^{-1} \rB_\rho)^l\, \rY\cong \coprod_{r\in \mathbb{Z}}\coprod_{s\in \mathbb{Z}} (\rB_\rho \rT_1)^{r+k} (\rT_1^{-1} \rB_\rho)^{s+l} \cong \rY.
\]
In particular, we see that 
\begin{equation}\label{eq:fundamentalisos}
\rB_\rho \rT_1 \rY \cong \rY, \quad \rT_1^{-1} \rB_\rho \rY\cong \rY, \quad \rB_\rho^2 \rY\cong \rB_\rho \rT_1 \rT_1^{-1} \rB_\rho \rY \cong \rY.
\end{equation}

For any $r,s\in \mathbb{Z}$, define  
\[
\rY^{r,s}:=(\rB_\rho \rT_1)^r (\rT_1^{-1} \rB_\rho)^s.
\]

\begin{lem}\label{lem:homfiniteness}
Let $i\in\{0,1\}$, and $k,t\in \mathbb{Z}$, and $m\in \mathbb{Z}_{\geq 0}$. Then  
\begin{equation}\label{eq:homfiniteness}
\mathrm{hom}_{K^b(\Sext_2)}\left(\rB_{\rho^k}\rB_{\raltseq{m}{i}}\langle t\rangle, \rY^{r,s}\right)=0
\end{equation}
for all but finitely many $r,s\in \mathbb{Z}$.    
\end{lem}
\begin{proof}
By \cite[Theorem 2.5]{mackaay-thiel}, the hom-space in \eqref{eq:homfiniteness} is zero if $r+s\ne k$. Now, assume that $r+s=k$. We also assume that $r\geq k$, the case $r\leq k$ being similar. Then 
\[
\rY^{r,s}=(\rB_\rho \rT_1)^r (\rT_1^{-1} \rB_\rho)^{k-r}\cong (\rB_\rho \rT_1)^k (\rB_\rho \rT_1)^{r-k}(\rB_\rho^{-1} \rT_1)^{r-k}
\cong \rB_\rho^k \rT_{\raltseq{2r-k}{1}}\cong \rB_{\rho^k} \rT_{\raltseq{2r-k}{1}}.
\]
By biadjointness of $\rB_{\rho^k}$ and $\rB_{\rho^{-k}}$, we have 
\[
\mathrm{hom}_{K^b(\Sext_2)}\left(\rB_{\rho^k}\rB_{\raltseq{m}{i}}\langle t\rangle, \rB_{\rho^k} \rT_{\raltseq{2r-k}{1}}\right)\cong 
\mathrm{hom}_{K^b(\Sext_2)}\left(\rB_{\raltseq{m}{i}}\langle t\rangle, \rT_{\raltseq{2r-k}{1}}\right),
\]
which in turn is isomorphic to 
\[
\mathrm{hom}_{K^b(\Sext_2)}\left(\rB_{\raltseq{m}{i}}\langle t\rangle, \rT_{\raltseq{2r-k}{1}}^{\min}\right),
\]
where $\rT_{\raltseq{2r-k}{1}}^{\min}$ is the minimal complex obtained from $\rT_{\raltseq{2r-k}{1}}$ by Gaussian elimination, 
which is unique up to homotopy equivalence. By~\cite[Theorem 6.9]{elias-williamson-1}, the degree-zero cochain object of 
$\rT_{\raltseq{2r-k}{1}}^{\min}$ is isomorphic to $\rB_{\raltseq{2r-k}{1}}$, which implies that  
\[
\dim_{\mathbb{R}} \left(\mathrm{hom}_{K^b(\Sext_2)}\left(\rB_{\raltseq{m}{i}}\langle t\rangle, \rT_{\raltseq{2r-k}{1}}^{\min}\right)\right) \leq 
\dim_{\mathbb{R}} \left(\mathrm{hom}_{K^b(\Sext_2)}\left(\rB_{\raltseq{m}{i}}\langle t\rangle, \rB_{\raltseq{2r-k}{1}}\right)\right).
\]
For a fixed choice of $m,t\in \mathbb{Z}$ and $m\in \mathbb{Z}_{\geq 0}$, Soergel's hom formula and the equations in~\eqref{eq:multconst} imply that the latter morphism space is zero for all but finitely many $r\in \mathbb{Z}$, which proves the lemma.
\end{proof}
\noindent The following corollary is an immediate consequence of~\autoref{finprod} and~\autoref{lem:homfiniteness}.
\begin{cor}\label{cor:homfiniteness}
For any $i\in\{0,1\}$, and $k,t\in \mathbb{Z}$, and $m\in \mathbb{Z}_{\geq 0}$, there is an isomorphism 
\begin{equation}\label{eq:homfiniteness2}
\mathrm{hom}_{K^b(\Sextdiamond_2)}\left(\rB_{\rho^k}\rB_{\raltseq{m}{i}}\langle t\rangle, \rY\right)\cong 
\bigoplus_{r,s\in \mathbb{Z}} \mathrm{hom}_{K^b(\Sext_2)}\left(\rB_{\rho^k}\rB_{\raltseq{m}{i}}\langle t\rangle, \rY^{r,s}\right).
\end{equation}
\end{cor}

We are now going to define a graded wide finitary birepresentation $\mathbf{U}$ of $\Sext_2$ and the rest of this section after the definition will be dedicated to the proof that it categorifies $U$.
\begin{defn}
Let 
\[
\mathbf{U}:=\mathbf{M}_\rY
\]
be the wide finitary birepresentation of $\Sext_2$ generated by $\rY$, as defined in \autoref{lem:widefineadd}.
\end{defn}
\autoref{finprod} and the isomorphism in~\eqref{freeforg} 
imply that
\begin{equation}\label{eq:free-forget}
\mathrm{hom}_{\rY}(\mathrm{A}\rY, \rB\rY) \cong \mathrm{hom}_{K^b(\Sextdiamond_2)}(\mathrm{A}, \rB\rY)
\cong \bigoplus_{r,s\in \mathbb{Z}} \mathrm{hom}_{K^b(\Sext_2)}(\mathrm{A}, \rB\rY^{r,s}), \quad \mathrm{A}, \rB\in \Sext_2,  
\end{equation}
where for simplicity we write $\mathrm{hom}_{\rY}(\mathrm{A}\rY, \rB\rY)=\mathrm{hom}_{\rmod_{K^b(\Sextdiamond_2)}\rY}(\mathrm{A}\rY, \rB\rY)$. 
The isomorphisms in~\eqref{eq:free-forget} show that any morphism $f\in 
\mathrm{hom}_{\rY}(\mathrm{A}\rY, \mathrm{B}\rY)$ is determined by its components 
$f^{r,s}\in \mathrm{hom}_{K^b(\Sext_2)}(\mathrm{A}, \mathrm{B}\rY^{r,s})$, for $r,s\in \mathbb{Z}$. 
Suppose $f\colon \rA\rY\to BY$ and $g\colon \rB\rY\to \rC\rY$ are such that their only nonzero components under the free forgetful adjunction are
$f^{r,s}\colon \rA\to \rB\rY^{r,s}$ and $g^{p,q}\colon \rB\to \rC\rY^{p,q}$ .
Then the only nonzero component of $g\circ f\colon \rA\rY\to \rC\rY$ under the free forgetful adjunction, is given by the composition
\begin{equation}\label{eq:free-forget-composition}
\rA \xrightarrow{f^{r,s}} \rB\rY^{r,s} \xrightarrow{g^{p,q}\circh \id_{\rY^{r,s}}} \rC\rY^{p,q}\rY^{r,s} \xrightarrow{\id_\rC\circh \mathbf{m}} \rC\rY^{p+r,q+s}.
\end{equation}
For morphisms with multiple nonzero components, this extends additively.

\begin{prop}\label{prop:homformula}
Let $k,m\in \mathbb{Z}_{\geq 0}$ and $r,s\in \mathbb{Z}$. When $r\in \mathbb{Z}_{>0}$, the morphism spaces  
\begin{gather*}
\mathrm{hom}_{K^b(\Sext_2)}(\rB_{\raltseq{k}{1}}\langle t\rangle, \rB_{\raltseq{m}{1}}\rY^{r,s})\;\,\text{and}\;\, \mathrm{hom}_{K^b(\Sext_2)}(\rB_{\raltseq{k}{0}}
\langle t\rangle, \rB_{\raltseq{m}{1}}\rY^{r,s})
\end{gather*}
are both zero for all $t\geq 0$. 
\end{prop}
\begin{proof}
We only prove the first case, the proof of the second one being similar. Biadjointness and the isomorphism $\rB_{\raltseq{m}{1}}^*\cong  \rB_{\laltseq{m}{1}}$ imply that 
\[
\mathrm{hom}_{K^b(\Sext_2)}(\rB_{\raltseq{k}{1}}\langle t\rangle, \rB_{\raltseq{m}{1}}\rY^{r,s})\cong \mathrm{hom}_{K^b(\Sext_2)}(\rB_{\laltseq{m}{1}}
\rB_{\raltseq{k}{1}}\langle t\rangle, \rY^{r,s}).
\]
By \cite[Theorem 2.5]{mackaay-thiel}, that hom-space is zero if $r+s\ne 0$, so suppose that $r+s=0$. Then 
\[
\rY^{r,s}\cong (\rT_0 \rT_1)^r.
\]
Let $r>0$. By \cite[Section 9.1]{AMRW} (see also \cite[Theorem 6.9]{elias-williamson-1} and \cite[Theorem 19.47]{e-m-t-w}), the minimal complex 
$[(\rT_0 \rT_1)^r]^{\mathrm{min}}$ is of the form 
\begin{gather}\label{eq:mincomplex1} 
\underline{\rB_{\raltseq{2r}{1}}} 
\xlongrightarrow{\mathrm{d}^0} \rB_{\raltseq{2r-1}{0}}\langle 1\rangle \oplus 
\rB_{\raltseq{2r-1}{1}}\langle 1 \rangle \longrightarrow\ldots, 
\end{gather}
where the underlined cochain object has homological degree zero and the differential $\mathrm{d}^{0}$ is given by 
\begin{equation}\label{eq:diff}
\mathrm{d}^0=
\begin{pmatrix}
\xy (0,.05)*{
\tikzdiagc[scale=0.4,yscale=1]{
  \draw[ultra thick,violet] (-1.5,-2)node[below]{\tiny $0$} -- (-1.5,-.75);
  \draw[ultra thick,violet] (-1.5,.75) -- (-1.5,2)node[above]{\tiny $0$};
  \draw[ultra thick,blue] (1.5,-2)node[below]{\tiny $1$} -- ( 1.5,-.75);
   \draw[ultra thick,blue] (1.5,.75) -- ( 1.5,1.75)node[pos=1, tikzdot]{} ;
 \draw[ultra thick,blue] (-1,-2)node[below]{\tiny $1$}  --  (-1,-.75); 
 \draw[ultra thick,blue] (-1,.75) --  (-1,2)node[above]{\tiny $1$};
  \draw[ultra thick,violet] (1,-2)node[below]{\tiny $0$}  --  (1,-.75); 
    \draw[ultra thick,violet] (1,.75)--  (1,2)node[above]{\tiny $0$}; 
    \node at (0,1.5) {$\cdots$};
     \node at (0,-1.5) {$\cdots$};
  \draw [draw=black] (-2.25,-.75) rectangle (2.25,.75);
  \node at (0,0) {\tiny $\mathrm{JW}_{(0,\ldots,1)}$};
}}\endxy
\; , \;
\xy (0,-.9)*{
\tikzdiagc[scale=0.4,yscale=1]{
  \draw[ultra thick,violet] (-1.5,-2)node[below]{\tiny $0$} -- (-1.5,-.75);
  \draw[ultra thick,violet] (-1.5,.75) -- (-1.5,1.75)node[pos=1, tikzdot]{};
  \draw[ultra thick,blue] (1.5,-2)node[below]{\tiny $1$} -- ( 1.5,-.75);
   \draw[ultra thick,blue] (1.5,.75) -- ( 1.5,2)node[above]{\tiny $1$};
 \draw[ultra thick,blue] (-1,-2)node[below]{\tiny $1$}  --  (-1,-.75); 
 \draw[ultra thick,blue] (-1,.75) --  (-1,2)node[above]{\tiny $1$};
  \draw[ultra thick,violet] (1,-2)node[below]{\tiny $0$}  --  (1,-.75); 
    \draw[ultra thick,violet] (1,.75)--  (1,2)node[above]{\tiny $0$}; 
    \node at (0,1.5) {$\cdots$};
     \node at (0,-1.5) {$\cdots$};
  \draw [draw=black] (-2.25,-.75) rectangle (2.25,.75);
  \node at (0,0) {\tiny $\mathrm{JW}_{(0,\ldots,1)}$};
}}\endxy
\end{pmatrix}^T,
\end{equation}
where $\mathrm{JW}_{(0,\ldots,1)}$ is the idempotent in $\mathrm{end}_{\Sext_2}(\rB_0\cdots\rB_1)$ obtained from the Jones-Wenzl projector 
$\mathrm{JW}_{2r}$ by the quantum Satake correspondence, as in~\cite[Section 5.3.2]{elias2016} (see also \cite[Section 9.3]{e-m-t-w}), 
and $T$ denotes the transpose. Note that Elias and Williamson do not give the multiplicities of the indecomposables in positive 
homological degree, but in affine type $A_1$ 
both multiplicities in homological degree one can be deduced from the corresponding coefficients in~\eqref{eq:regular-KL}.
On the other hand, the formulas in \eqref{eq:multconst} imply that the decomposition of $\rB_{\laltseq{m}{1}}\rB_{\raltseq{k}{1}}\langle t\rangle$ only contains indecomposables of the form $\rB_{\lraltseq{n}{1}{1}}\langle t+v \rangle$, with $0\leq n\leq k+m$ and $v\in
\{-1,0,1\}$. The parity conditions for $\lraltseq{n}{1}{1}$ amount to $n$ being odd. 
Any non-zero morphism of complexes in $\mathrm{hom}_{K^b(\Sext_2)}(\rB_{\laltseq{m}{1}}\rB_{\raltseq{k}{1}}\langle t\rangle, [(\rT_0\rT_1)^r]^{\mathrm{min}})$ is given by a non-zero morphism in the morphism space 
$\mathrm{hom}_{\Sext_2}(\rB_{\laltseq{m}{1}}\rB_{\raltseq{k}{1}}\langle t\rangle, \rB_{\raltseq{2r}{1}})$ that is annihilated by postcomposition with the differential $\mathrm{d}^0$. The latter morphism space is the direct sum of morphism spaces whose sources are the indecomposables 
of the form $\rB_{\lraltseq{n}{1}{1}}\langle t+v \rangle$, with $0\leq n\leq k+m$ and $v\in
\{-1,0,1\}$, and Soergel's hom formula in 
\eqref{eq:Soergelhomformula} and \autoref{lem:innerproduct} imply that 
the graded rank of $\mathrm{Hom}_{\Sextstar_2}(\rB_{\lraltseq{n}{1}{1}}, \rB_{\raltseq{2r}{1}})$ is equal to 
\begin{equation}\label{eq:laurentpol}
\left(b_{\lraltseq{n}{1}{1}}, b_{\lraltseq{2r}{0}{1}}\right)=q^{\vert n-2r\vert}p'(q),  
\end{equation}
where $p'(q)\in \mathbb{Z}_{\geq 0}[q]$ has constant term equal to one. This implies that the morphism space 
$\mathrm{hom}_{\Sext_2}(\rB_{\lraltseq{n}{1}{1}}\langle t+v\rangle, \rB_{\raltseq{2r}{1}})$ is zero unless $t+v =-\vert n-2r\vert$. Suppose that 
$t\geq 0$ and recall that $v\in \{-1,0,1\}$. Since $n$ must be odd, we also see that $\vert n-2r\vert>0$, so $t+v =-\vert n-2r\vert$ can only 
hold if $t=0$, $v=-1$ and $\vert n-2r\vert =1$. In other words, the only non-zero morphism spaces are  
\[
\mathrm{hom}_{\Sext_2}(\rB_{\lraltseq{2r-1}{1}{1}}\langle -1\rangle, 
\rB_{\lraltseq{2r}{0}{1}})  \quad \text{and}\quad \mathrm{hom}_{\Sext_2}(\rB_{\lraltseq{2r+1}{1}{1}}\langle -1\rangle, 
\rB_{\lraltseq{2r}{0}{1}}), 
\]
which are both one-dimensional and generated by
\begin{equation}\label{eq:homgens}
\xy (0,.05)*{
\tikzdiagc[scale=0.4,yscale=1]{
  \draw[ultra thick,violet] (-1.5,-1.75) -- (-1.5,-.75)node[pos=0, tikzdot]{};
  \draw[ultra thick,violet] (-1.5,.75) -- (-1.5,2)node[above]{\tiny $0$};
  \draw[ultra thick,blue] (1.5,-2)node[below]{\tiny $1$} -- ( 1.5,-.75);
   \draw[ultra thick,blue] (1.5,.75) -- ( 1.5,2)node[above]{\tiny $1$};
 \draw[ultra thick,blue] (-1,-2)node[below]{\tiny $1$}  --  (-1,-.75); 
 \draw[ultra thick,blue] (-1,.75) --  (-1,2)node[above]{\tiny $1$};
  \draw[ultra thick,violet] (1,-2)node[below]{\tiny $0$}  --  (1,-.75); 
    \draw[ultra thick,violet] (1,.75)--  (1,2)node[above]{\tiny $0$}; 
    \node at (0,1.5) {$\cdots$};
     \node at (0,-1.5) {$\cdots$};
  \draw [draw=black] (-2.25,-.75) rectangle (2.25,.75);
  \node at (0,0) {\tiny $\mathrm{JW}_{(0,\ldots,1)}$};
}}\endxy
\quad \text{and} \quad 
\xy (0,-.9)*{
\tikzdiagc[scale=0.4,yscale=1]{
  \draw[ultra thick,blue] (-1.5,-2)node[below]{\tiny $1$} -- (-1.5,-.75);
  \draw[ultra thick,blue] (-1.5,.75) -- (-1.5,1.75)node[pos=1, tikzdot]{};
  \draw[ultra thick,blue] (1.5,-2)node[below]{\tiny $1$} -- ( 1.5,-.75);
   \draw[ultra thick,blue] (1.5,.75) -- ( 1.5,2)node[above]{\tiny $1$};
 \draw[ultra thick,violet] (-1,-2)node[below]{\tiny $0$}  --  (-1,-.75); 
 \draw[ultra thick,violet] (-1,.75) --  (-1,2)node[above]{\tiny $0$};
  \draw[ultra thick,violet] (1,-2)node[below]{\tiny $0$}  --  (1,-.75); 
    \draw[ultra thick,violet] (1,.75)--  (1,2)node[above]{\tiny $0$}; 
    \node at (0,1.5) {$\cdots$};
     \node at (0,-1.5) {$\cdots$};
  \draw [draw=black] (-2.25,-.75) rectangle (2.25,.75);
  \node at (0,0) {\tiny $\mathrm{JW}_{(1,\ldots,1)}$};
}}\endxy\;,
\end{equation}
respectively. However, neither of these two morphisms extends to a map of complexes whose 
target is $[(\rT_0\rT_1)^r]^{\mathrm{min}}$. To see this, note that postcomposition with 
$\mathrm{d}^0$ adds an extra dot to the rightmost, resp. leftmost, top strand of the morphisms in~\eqref{eq:homgens}. These morphisms are all 
non-zero, because even with dots on all strands such morphisms are non-zero, see e.g.~\cite[Section 5.3.4]{elias2016} 
and~\cite[(10.8i) and (10.9)]{e-m-t-w}. 
This finishes the proof that 
\[
\mathrm{hom}_{K^b(\Sext_2)}(\rB_{\raltseq{k}{1}}\langle t\rangle, \rB_{\raltseq{m}{1}}\rY^{r,s})=0,\quad \text{if}\;\, r>0\;\text{and}\; t\geq 0.
\qedhere
\]
\end{proof}

\begin{cor}\label{cor:indecomp}
For all $k\in \mathbb{Z}_{\geq 0}$, the objects $\rB_{\raltseq{k}{1}}\rY$ and $\rB_\rho \rB_{\raltseq{k}{1}}\rY$ are indecomposable in 
$\mathbf{U}$. 
\end{cor}
\begin{proof}
The isomorphisms in~\eqref{eq:free-forget} show that any non-zero idempotent 
$e\in\mathrm{hom}_{\rY}(\rB_{\raltseq{k}{1}}\rY, \rB_{\raltseq{k}{1}}\rY)$ is completely 
determined by its components $e^{r,s}\in \mathrm{hom}_{K^b(\Sext_2)}(\rB_{\raltseq{k}{1}}, \rB_{\raltseq{k}{1}}\rY^{r,s})$, for $r,s\in \mathbb{Z}$, 
and that only finitely many components are non-zero. By \cite[Theorem 2.5]{mackaay-thiel} and 
\autoref{prop:homformula}, we have $e^{r,s}=0$ unless $r+s=0$ and $r\leq 0$. Furthermore, $e^{0,0}$ must be a multiple of 
the identity on $\rB_{\raltseq{k}{1}}$, by Soergel's hom formula in~\eqref{eq:Soergelhomformula}. By a slight abuse of notation, justified by 
\autoref{lem:freeforgcoord}, let us write $e=e^{r_1,-r_1}+\ldots + e^{r_m,-r_m}$, for some $r_i\in \mathbb{Z}_{\leq 0}$ and $m\in \mathbb{Z}_{>0}$. Then $e^2=e$ if and only if 
\begin{equation}\label{eq:idempotent}
\sum_{\substack{1\leq i,j\leq m \\[.3ex] r_i+r_j=r_k}} e^{r_i,-r_i}\circ e^{r_j,-r_j}=e^{r_k,-r_k},\;\text{for all}\; 1\leq k\leq m.
\end{equation}

First, assume that $e^{0,0}=0$ and let $r_{\max}$ be the maximal $r\in \mathbb{Z}_{<0}$ such that $e^{r,-r}\ne 0$. Then the equation in \eqref{eq:idempotent} for $r_k=r_{\max}$ has no solution, which shows that $e^2= e$ cannot hold. 

Next, assume that $e^{0,0}\ne 0$. Without loss of 
generality, we may assume that $r_1=0$ and $r_i<0$ for all $2\leq i\leq m$. Then the equations in \eqref{eq:idempotent} can only hold if 
$e^{0,0}=\mathrm{id}_{\rB_{\raltseq{k}{1}}}$ and 
\begin{gather}\label{eq:idempotent2}
e^{r_k,-r_k} + \sum_{\substack{i,j\in \{2,\ldots, \hat{k},\ldots, m\} \\[.3ex] r_i+r_j=r_k}} e^{r_i,-r_i}\circ e^{r_j,-r_j}=0, \;\text{for all}\; 2\leq k\leq m.
\end{gather}
In turn, the equations in \eqref{eq:idempotent2} can only hold if $e^{r_k,-r_k}=0$ for all $2\leq k\leq m$, as can be seen by first considering  
the maximal $r_{k}<0$ and then working one's way down until the minimal $r_k<0$. This shows that $e^2=e$ holds if and only if 
$e=\mathrm{id}_{\rB_{\raltseq{k}{1}}\rY}$, hence $\rB_{\raltseq{k}{1}}\rY$ is indecomposable in $\mathbf{U}$.

Since $\rB_{\rho}^*\cong \rB_{\rho^{-1}}$, we have 
\[
\mathrm{hom}_{\rY}(\rB_\rho \rB_{\raltseq{k}{1}}\rY, \rB_\rho \rB_{\raltseq{k}{1}}\rY)\cong \mathrm{hom}_{\rY}(\rB_{\raltseq{k}{1}}\rY, \rB_{\raltseq{k}{1}}\rY),
\]
which implies that $\rB_\rho \rB_{\raltseq{k}{1}}\rY$ is also indecomposable in $\mathbf{U}$, by the arguments above. 
\end{proof}

\begin{rem}\label{rem:counterex}
When $r<0$, the morphism space $\mathrm{hom}_{K^b(\Sext_2)}(\rB_{\raltseq{k}{1}}, \rB_{\raltseq{m}{1}}\rY^{r,s})$ can be non-zero. 
Take $k=m=2$ and $r=-1$, for example. Then $\rY^{-1,+1}\cong (\rT_0 \rT_1)^{-1}\cong \rT_1^{-1} \rT_0^{-1}$ and 
\[
\mathrm{hom}_{K^b(\Sext_2)}(\rB_{01}, \rB_{01}\rT_1^{-1} \rT_0^{-1})\cong \mathrm{hom}_{K^b(\Sext_2)}(\rB_{10}\rB_{01}, \rT_1^{-1} \rT_0^{-1}) 
\]
On the one hand, the complex $\rT_1^{-1} \rT_0^{-1}$ (which is already minimal) is given by 
\begin{gather}\label{eq:mincomplex3} 
 R\langle -2\rangle \xlongrightarrow{d^{-2}} \rB_0 \langle -1\rangle \oplus 
\rB_1\langle -1 \rangle \xlongrightarrow{d^{-1}}  \underline{\rB_{10}}, 
\end{gather}
with $d^{-2}$ and $d^{-1}$ given by 
\begin{equation}\label{eq:diff2}
\mathrm{d}^{-2}=
\begin{pmatrix}
-\,\xy (0,2)*{
\tikzdiagc[scale=0.4,yscale=1]{
  \draw[ultra thick,violet] (-1.5,-.75) -- (-1.5,.75)node[above]{\tiny $0$}node[pos=0, tikzdot]{};
}}\endxy
 , 
\xy (0,2)*{
\tikzdiagc[scale=0.4,yscale=1]{
  \draw[ultra thick,blue] (-1.5,-.75)-- (-1.5,.75)node[above]{\tiny $1$}node[pos=0, tikzdot]{};
}}\endxy
\end{pmatrix}^T
\quad
\text{and}
\quad
\mathrm{d}^{-1}=
\begin{pmatrix}
\xy (0,2)*{
\tikzdiagc[scale=0.4,yscale=1]{
  \draw[ultra thick,blue] (-1.5,-.75) -- (-1.5,.75)node[above]{\tiny $1$}node[pos=0, tikzdot]{};
   \draw[ultra thick,violet] (-1,-1) -- (-1,.75)node[above]{\tiny $0$};
}}\endxy
 , 
\xy (0,1)*{
\tikzdiagc[scale=0.4,yscale=1]{
\draw[ultra thick,blue] (-1.5,-1)-- (-1.5,.75)node[above]{\tiny $1$};
  \draw[ultra thick,violet] (-1,-.75)-- (-1,.75)node[above]{\tiny $0$}node[pos=0, tikzdot]{};
}}\endxy
\end{pmatrix}
\end{equation}
On the other hand, we have 
\begin{equation}\label{eq:indecomp}
\rB_{10}\rB_{01}\cong \rB_{101}\langle 1\rangle \oplus \rB_{101}\langle -1\rangle \oplus \rB_1 \langle 1\rangle \oplus \rB_1 \langle -1\rangle.
\end{equation}
Any non-zero morphism in $\hom_{\Sext_2}(\mathrm{Z}, \rB_{10})$, where $\mathrm{Z}$ is one of the four indecomposables in \eqref{eq:indecomp}, induces a map of complexes $ \mathrm{Z}\to \rT_1^{-1} \rT_0^{-1}$, which can be null-homotopic or not. By Soergel's hom formula in~\eqref{eq:Soergelhomformula}, the morphism spaces 
$\hom_{\Sext_2}(\rB_{101}\langle -1\rangle, \rB_{10})$ and $\hom_{\Sext_2}(\rB_{1}\langle -1\rangle, \rB_{10})$ are both one-dimensional, with 
respective generators given by 
\begin{equation}\label{eq:homgens2}
\xy (0,-.9)*{
\tikzdiagc[scale=0.4,yscale=1]{
  \draw[ultra thick,blue] (-1,-2)node[below]{\tiny $1$} -- (-1,-.75);
  \draw[ultra thick,blue] (-1,.75) -- (-1,2)node[above]{\tiny $1$};
  \draw[ultra thick,blue] (1,-2)node[below]{\tiny $1$} -- (1,-.75);
   \draw[ultra thick,blue] (1,.75) -- (1,2)node[pos=1, tikzdot]{};
 \draw[ultra thick,violet] (0,-2)node[below]{\tiny $0$}  --  (0,-.75); 
 \draw[ultra thick,violet] (0,.75) --  (0,2)node[above]{\tiny $0$};
  \draw [draw=black] (-1.75,-.75) rectangle (1.75,.75);
  \node at (0,0) {\tiny $\mathrm{JW}_{(1,0,1)}$};
}}\endxy\;
\quad \text{and} \quad 
\xy (0,1)*{
\tikzdiagc[scale=0.4,yscale=1]{
\draw[ultra thick,blue] (-1,-1.5)-- (-1,1.5)node[above]{\tiny $1$};
  \draw[ultra thick,violet] (0,-1.25)-- (0,1.5)node[above]{\tiny $0$}node[pos=0, tikzdot]{};
}}\endxy
,
\end{equation}
and the other two possible morphism spaces are both zero. The map of complexes induced by the non-zero morphism in $\hom_{\Sext_2}(\rB_{1}\langle -1\rangle, \rB_{01})$ is null-homotopic, with the homotopy being induced by the identity on $\rB_1 \langle -1\rangle$. The map of complexes induced by the non-zero morphism in $\hom_{\Sext_2}(\rB_{101}\langle -1\rangle, \rB_{10})$ is not null-homotopic, 
because Soergel's hom formula implies that $\hom_{\Sext_2}(\rB_{101}\langle -1\rangle, \rB_{0}\langle -1\rangle)=
\hom_{\Sext_2}(\rB_{101}\langle -1\rangle, \rB_{1}\langle -1\rangle)=0$, so there are no morphisms which could constitute a homotopy. 

Note that this implies that $\mathrm{hom}_{\rY}(\rB_{01}\rY, \rB_{01}\rY)$ contains a non-zero morphism which is not 
a multiple of an idempotent. 
\end{rem}

\begin{thm}\label{thm:mainresultexample} The following holds in $\mathbf{U}$.
\begin{enumerate}[a)]
\item We have 
\[
\rB_\rho \rB_{\raltseq{k}{0}}\rY\cong \rB_{\raltseq{k}{1}} \rB_\rho \rY \cong \rB_{\raltseq{k}{1}}\rY\langle -1\rangle, \quad \forall\, k\in\mathbb{Z}_{>0}.
\]
\item We have 
\[
\rB_{\raltseq{k}{0}}\rY\cong \rB_\rho \rB_{\raltseq{k}{1}}\rY\langle -1\rangle, \quad \forall\, k\in\mathbb{Z}_{>0}.
\]
\item 
The indecomposables 
\[
\rB_{\raltseq{k}{1}}\rY\langle s\rangle,\; \rB_\rho \rB_{\raltseq{m}{1}}\rY\langle t\rangle, \quad k,m\in \mathbb{Z}_{\geq 0}; s,t\in \mathbb{Z}
\] 
are pairwise non-isomorphic.
\end{enumerate}
\end{thm}
\begin{proof}
\begin{enumerate}[wide,labelindent=0pt,label={(\alph*)}]
\item The isomorphism $\rB_1 \rT_1\cong \rB_1\langle -1\rangle$ in $K^b(\Sext_2)$ and the isomorphism $\rB_\rho \rY\cong \rT_1\rY$ in $\mathbf{U}$ 
imply that 
\[
\rB_\rho \rB_{0}\rY\cong \rB_{1} \rB_\rho \rY\cong \rB_1 \rT_1 \rY \cong \rB_1 \rY \langle -1 \rangle\;\,\text{in}\;\, \mathbf{U}. 
\]
The general result, as stated in the theorem, follows by induction and uniqueness of decomposition into indecomposables in $\mathbf{U}$, because  
\[
\rB_{0}\rB_{1} \cong \rB_{01}\quad\text{and}\quad
\rB_{\raltseq{k-1}{0}} \rB_1 \cong \rB_{\raltseq{k}{1}} \oplus \rB_{\raltseq{k-2}{1}}\;\,\text{in}\;\, \Sext_2,
\]
for all $k\in \mathbb{Z}_{> 2}$.
\item The isomorphism in this case can be obtained from the one in the previous case by tensoring with $\rB_\rho$ on the left and using that 
$\rB_\rho^2$ acts as the identity on $\mathbf{U}$. 

\item There are three different cases: we have to show that  
\begin{enumerate}[i)]
\item $\rB_{\raltseq{k}{1}}\rY\langle s \rangle\not\cong \rB_{\raltseq{m}{1}}\rY\langle t \rangle$, for $k,m\in \mathbb{Z}_{\geq 0}$ with $k\ne m$ and $s,t\in \mathbb{Z}$; 
\item $\rB_{\raltseq{k}{1}}\rY\langle s \rangle\not\cong \rB_\rho\rB_{\raltseq{m}{1}}\rY\langle t \rangle$, for $k,m\in \mathbb{Z}_{\geq 0}$ and $s,t\in \mathbb{Z}$;
\item $\rB_\rho\rB_{\raltseq{k}{1}}\rY\langle s \rangle\not\cong\rB_\rho\rB_{\raltseq{m}{1}}\rY\langle t \rangle$, for $k,m\in \mathbb{Z}_{\geq 0}$ with $k\ne m$ and $s,t\in \mathbb{Z}$.
\end{enumerate}
The first case implies the third one, because 
\[
\mathrm{hom}_{\rY}(\rB_\rho \rB_{\raltseq{m}{1}}\rY\langle t\rangle ,\rB_\rho \rB_{\raltseq{k}{1}}\rY)\cong \mathrm{hom}_{\rY}(\rB_{\raltseq{m}{1}} \rY\langle t\rangle ,\rB_{\raltseq{k}{1}}\rY) 
\]
by adjunction, so it suffices to prove the first two cases.

i) Suppose that there is an isomorphism in $\mathbf{U}$ 
\[
f\colon \rB_{\raltseq{k}{1}}\rY\langle t\rangle \to \rB_{\raltseq{m}{1}}\rY
\]
for some $k,m\in \mathbb{Z}_{\geq 0}$ with $k\ne m$ and some $t\in \mathbb{Z}$, which we may assume to be non-negative without loss of generality. As before, 
$f$ is determined by its non-zero components in $K^b(\Sext_2)$
\[
f^{r,s}\colon \rB_{\raltseq{k}{1}}\langle t\rangle \to \rB_{\raltseq{m}{1}}\rY^{r,s}
\]
for a finite number of integers $r,s$. By \cite[Theorem 2.5]{mackaay-thiel} and \autoref{prop:homformula}, this implies that $r+s=0$ and $r\leq 0$. Since $t\geq 0$ and $k\ne m$, both by assumption, 
Soergel's hom formula \eqref{eq:Soergelhomformula} implies that $f^{-s,s}=0$ if $s=0$. Thus, we must have $s>0$ for all non-zero $f^{-s,s}$. Recall that 
$\rY^{-s,s}\simeq (\rT_1^{-1} \rT_0^{-1})^s$ and that the minimal complex $[(\rT_1^{-1} \rT_0^{-1})^s]^{\mathrm{min}}$ is given by 
\[
\cdots \xlongrightarrow{}\rB_{\raltseq{s-1}{0}}\langle -1 \rangle\oplus \rB_{\raltseq{s-1}{1}}\langle -1\rangle
\xlongrightarrow{\mathrm{d}^{-1}}\rB_{\raltseq{s}{0}}.
\]
By adjunction, we also have 
\[
\mathrm{hom}_{K^b(\Sext_2)}(\rB_{\raltseq{k}{1}}\langle t\rangle ,\rB_{\raltseq{m}{1}}\rY^{-s,s})\cong \mathrm{hom}_{K^b(\Sext_2)}(\rB_{\laltseq{m}{1}}\rB_{\raltseq{k}{1}}\langle t\rangle ,\rY^{-s,s}), 
\]
and \eqref{eq:multconst} and \autoref{thm:extcategorification} imply that the indecomposables in the decomposition of $\rB_{\laltseq{m}{1}}\rB_{\raltseq{k}{1}}\langle t\rangle$ are of the form 
$\rB_{\lraltseq{n}{1}{1}}\langle t'\rangle$, for certain integers $n\geq \vert k-m\vert$ and certain $t'\in \{t-1,t,t+1\}$. Since $k\ne m$, we must have $n>0$. Therefore, 
Soergel's hom formula \eqref{eq:Soergelhomformula} implies that 
\[
\mathrm{hom}_{\Sext_2}(\rB_{\lraltseq{n}{1}{1}}\langle t' \rangle,\rB_{\raltseq{s}{0}})=0
\]
for all integers $n,s>0$ and $t'\geq 0$. Since $t'<0$ can only hold if $t\leq 0$, our initial assumption that $t\geq 0$ implies that we must have $t=0$ and $t'=-1$. However, this leads to a contradiction.  
Since $f^{-s,s}$ is zero unless $s>0$, the inverse of $f$ must have a non-zero component $(f^{-1})^{s,-s}\colon \rB_{\raltseq{m}{1}}\to \rB_{\raltseq{k}{1}}\rY^{s,-s}$ for some $s>0$, which is 
impossible according to \autoref{prop:homformula}. This completes the proof that the $\rB_{\raltseq{k}{1}}\rY\langle t \rangle$ 
are all pairwise non-isomorphic. 

ii) It suffices to show that $\rB_\rho \rB_{\raltseq{k}{1}}\rY\not\cong\rB_{\raltseq{m}{1}}\rY\langle t\rangle$, for any $k,m\in \mathbb{Z}_{\geq 0}$ and any $t\in \mathbb{Z}$. First assume that $k=m=0$ and suppose, on the contrary, that there is an isomorphism 
\[
f\colon \rB_\rho \rY \to \rY\langle t\rangle
\]
for some $t\in \mathbb{Z}$. By tensoring on the left with $\rB_\rho$ and using the fact that $\rB_\rho^2$ acts as the identity, we see that 
$t$ has to be zero. Further, $f$ is determined by its non-zero components $f^{r,s}\colon \rB_\rho \to \rY^{r,s}$. 
By \cite[Theorem 2.5]{mackaay-thiel}, such a component can only be non-zero if $r+s=1$, so assume this and assume also that $r\geq 1$. Then 
$\rY^{r,s}\simeq \rB_\rho \rT_{\raltseq{2r-1}{1}}$. 
By adjunction, we have 
\begin{equation}\label{eq:zero-case}
\mathrm{hom}_{K^b(\Sext_2)}(\rB_\rho, \rB_\rho \rT_{\raltseq{2r-1}{1}})\cong 
\mathrm{hom}_{K^b(\Sext_2)}(R, \rT_{\raltseq{2r-1}{1}}).
\end{equation}
The minimal complex $[\rT_{\raltseq{2r-1}{1}}]^{\min}$ is given by $\rB_{\raltseq{2r-1}{1}}$ in homological degree zero, so Soergel's hom 
formula in \eqref{eq:Soergelhomformula} implies that the second morphism space in \eqref{eq:zero-case} is zero because $2r-1>0$. The case 
when $r<1$ can be proved similarly. This implies that $f^{r,s}=0$ for all $r,s\in \mathbb{Z}$, which proves that $\rB_\rho \rY$ can not 
be isomorphic to $\rY\langle t\rangle$ for any $t\in \mathbb{Z}$.

Now assume that $k>0$. Item b) implies that the proof is complete if we can show that $\rB_{\raltseq{k}{0}}\rY$ and $\rB_{\raltseq{m}{1}}\rY\langle t\rangle$ are not isomorphic for any $m\in \mathbb{Z}_{\geq 0}$ and any $t\in \mathbb{Z}$. Note that the case when $m>0$ reduces to the case when $k>0$ by tensoring both complexes on the left with $\rB_\rho$ and using that $\rB_\rho^2$ acts as the identity, so we can indeed assume 
that $k>0$ without loss of generality. Now, suppose in that case that there exists an isomorphism 
\begin{equation}\label{eq:lastcase}
f\colon \rB_{\raltseq{k}{0}}\rY\langle t\rangle\to \rB_{\raltseq{m}{1}}\rY
\end{equation}
in $\mathbf{U}$, for some $k,m\in \mathbb{Z}_{\geq 0}$ and $t\in \mathbb{Z}$. 

We will first show that $t=1$ must hold. Then we will show that $\rB_{\raltseq{k}{0}}\langle 1\rangle$ must be a direct summand of $\rB_{\raltseq{m}{1}}\rY^{-s,s}$ for some $s>0$. 
Finally, we will show that that is impossible. To arrive at that contradiction, we will first show that there is only a monomorphism $\rB_{\raltseq{k}{0}}\langle 1\rangle \to 
\rB_{\raltseq{m}{1}}\rY^{-s,s}$ for $k=1$, but that there is no epimorphism $\rB_{\raltseq{m}{1}}\rY^{-s,s}\to \rB_{0}\langle 1\rangle$. To prove that $k=1$ must hold, we use {\em localization}, which we briefly recall in the beginning of the relevant paragraph.

Note that
\begin{eqnarray*}
\mathrm{hom}(\rB_{\raltseq{k}{0}}\rY\langle t\rangle , \rB_{\raltseq{m}{1}}\rY) & \cong & \mathrm{hom}(\rB_\rho\rB_{\raltseq{k}{0}}\rY\langle t\rangle , \rB_\rho \rB_{\raltseq{m}{1}}\rY) \\
& \cong & \mathrm{hom}(\rB_{\raltseq{k}{1}}\rY\langle t-1\rangle,\rB_{\raltseq{m}{0}}\rY\langle 1\rangle) \\
& \cong & \mathrm{hom}(\rB_{\raltseq{k}{1}}\rY, \rB_{\raltseq{m}{0}}\rY\langle 2-t\rangle).
\end{eqnarray*}
Hence, the inverse of $f$ induces an isomorphism $\rB_{\raltseq{m}{0}}\rY\langle 2-t\rangle\to \rB_{\raltseq{k}{1}}\rY$ in $\mathbf{U}$. Since $t\leq 1$ if and only if $2-t\geq 1$, we can assume that 
$t\geq 1$ in \eqref{eq:lastcase} without loss of generality. Thus, let $t\geq 1$. As before, the isomorphism $f$ is determined by its components 
\[
f^{r,s}\colon \rB_{\raltseq{k}{0}}\langle t\rangle\to \rB_{\raltseq{m}{1}}\rY^{r,s}.
\]
By \cite[Theorem 2.5]{mackaay-thiel} and \autoref{prop:homformula}, we have $f^{r,s}=0$ unless $r+s=0$ and $r\leq 0$. Soergel's hom formula in \eqref{eq:Soergelhomformula} implies that $f^{0,0}=0$, thus 
all non-zero components $f^{r,-r}$ satisfy $r<0$ (which is equivalent to saying that $f^{-s,s}\ne 0$ can only hold if $s>0$, of course). 

Now, suppose that $s=-r>0$. We claim that $t=1$ must hold. Note that $\rY^{-s,s}\simeq (\rT_1^{-1} \rT_0^{-1})^s$ and the $0$-cochain object of 
$[(\rT_1^{-1} \rT_0^{-1})^s]^{\min}$ is $\rB_{\raltseq{2s}{0}}$. By adjunction, we have 
\[
\mathrm{hom}_{K^b(\Sext_2)}(\rB_{\raltseq{k}{0}}\langle t\rangle, \rB_{\raltseq{m}{1}}\rY^{-s,s})\cong 
\mathrm{hom}_{K^b(\Sext_2)}(\rB_{\laltseq{m}{1}}\rB_{\raltseq{k}{0}}\langle t\rangle,\rY^{-s,s}),
\]
and the indecomposables in the decomposition of $\rB_{\laltseq{m}{1}}\rB_{\raltseq{k}{0}}\langle t\rangle$ are all of the form 
$\rB_{\raltseq{n}{0}}\langle t'\rangle$ for certain $n\geq \vert k-m\vert$ and $t'\in \{t-1,t,t+1\}$. Soergel's hom formula in \eqref{eq:Soergelhomformula} implies that 
\[
\mathrm{hom}_{\Sext_2}(\rB_{\raltseq{n}{0}}\langle t'\rangle,\rB_{\raltseq{2s}{0}})
\]
is zero for all $t'>0$ (and any $n\geq 0$) and is one-dimensional for $t'=0$ (in which case $n=2s$). This shows that $f^{-s,s}\ne 0$ implies that $t'=0$, which can only hold if $t=1$ since 
$t\geq 1$ by assumption. This completes the proof of the claim. 

Now suppose that $g=f^{-1}$. Then there must be an integer $s>0$ such that the following 
composite is non-zero: 
\[
\rB_{\raltseq{k}{0}}\langle 1\rangle\xrightarrow{f^{-s,s}} \rB_{\raltseq{m}{1}}\rY^{-s,s} 
\xrightarrow{g^{s,-s}\circ\, \mathrm{id}_{\rY^{-s,s}}} \rB_{\raltseq{k}{0}}\langle 1\rangle 
\rY^{s,-s} \rY^{-s,s}
\xrightarrow{\mathrm{id}_{\rB_{\raltseq{k}{0}}\langle 1\rangle\circ \mu}} \rB_{\raltseq{k}{0}}\langle 1\rangle, 
\]
where we have used the canonical isomorphism $\rB_{\raltseq{k}{0}}\rY^{0,0}\langle 1\rangle
\cong \rB_{\raltseq{k}{0}}\langle 1\rangle$ for the final codomain. Soergel's hom formula in 
\eqref{eq:Soergelhomformula} implies that that composite is a non-zero multiple of the 
identity, hence $\rB_{\raltseq{k}{0}}\langle 1\rangle$ is a direct summand of 
$\rB_{\raltseq{m}{1}}\rY^{-s,s}$. However, that is impossible, as we will show below. 

First identify $\Sext_2$ with the algebraic category of Soergel bimodules over $R$ (the polynomial ring generated by the dumbbells). Let $Q$ be the field of fractions of $R$ and recall the {\em localization functor} $\mathrm{Loc}\colon \Sext_2\to \Sext_{2,Q}$, defined by 
forgetting the grading of Soergel bimodules and tensoring them on the right with $Q$ (over $R$), see \cite[Section 5.4]{e-m-t-w}. 
By \cite[Lemma 5.20]{e-m-t-w}, this is a monoidal functor. For any $u\in \Syaff_2$, let $Q_u$ be the {\em standard} $Q$-$Q$-bimodule, 
which is the one-dimensional $Q$-vector space with left and right $Q$-actions defined by
\[
x \cdot q \cdot y:= xqu(y),
\]
for $q,x,y\in Q$. For any $u,v\in \Syaff_2$, we have 
$Q_u Q_v \cong Q_{uv}$, see e.g. \cite[(5.5)]{e-m-t-w}, and the morphism spaces between the standard $Q$-$Q$-bimodules are very easy 
to describe (as follows from e.g. \cite[Lemma 5.2]{e-m-t-w}):
\[
\mathrm{hom}_{Q\text{-}Q-\mathrm{bimod}}(Q_u,Q_v)\cong 
\begin{cases}
Q, & \mathrm{if}\; u=v; \\
0, & \mathrm{else}.
\end{cases}
\]
As indicated in the commutative diagram \cite[(5.25)]{e-m-t-w} and proved in \cite[Theorem 18.22]{e-m-t-w}, for every $w\in \Syaff_2$, we have 
\[
B_w\otimes_R Q\cong \bigoplus_{u\preceq w} Q_u,
\]
because in affine type $A_1$ we have $p_{u,w}(1)=1$ for all $u\preceq w$ and zero else (where $p_{u,w}$ is the Kazhdan-Lusztig polynomial for $u,w\in \Syaff_2$). Concretely, this means that 
\[
\mathrm{Loc}\left(B_{\raltseq{n}{i}}\right)\cong Q_{\raltseq{n}{i}}\oplus \bigoplus_{0< l<n} \left(Q_{\raltseq{l}{0}} \oplus Q_{\raltseq{l}{1}}\right) \oplus Q,
\]
for any $n\in \mathbb{Z}_{\geq 0}$ and $i\in \{0,1\}$ (note that $Q_{\raltseq{0}{i}}=Q$). Localization can also be applied to Rouquier complexes and \cite[Lemma 5.21]{e-m-t-w} implies that, for every $w\in \Syaff_2$, we have 
\[
\mathrm{Loc}(\rT_w)\simeq Q_w \quad\text{and}\quad \mathrm{Loc}(\rT^{-1}_w)\simeq Q_{w^{-1}}.
\]
The results in the previous paragraph imply that $\mathrm{Loc}(\rB_{\raltseq{k}{0}})$ must be a 
direct summand of $\mathrm{Loc}\left(\rB_{\raltseq{m}{1}}\rY^{-s,s}\right)$, which imposes conditions on $k$, $m$ and $s$. 
On the one hand, we get    
\[
\mathrm{Loc}(\rB_{\raltseq{k}{0}})\cong 
Q_{\raltseq{k}{0}}\oplus \bigoplus_{0<l<k} \left(Q_{\raltseq{l}{0}} \oplus Q_{\raltseq{l}{1}}\right) \oplus Q
\]
and, on the other hand, we get 
\begin{gather*}
\mathrm{Loc}\left(\rB_{\raltseq{m}{1}}\rY^{-s,s}\right)\cong \\
\left(Q_{\raltseq{m}{1}}\oplus \bigoplus_{0<n<m} \left(Q_{\raltseq{n}{0}} \oplus Q_{\raltseq{n}{1}}\right) 
\oplus Q\right) Q_{\lraltseq{2s}{1}{0}}\cong \\
Q_{\raltseq{\vert m-2s\vert}{0}}\oplus \bigoplus_{0<n<m} 
\left(Q_{\raltseq{n+2s}{0}} \oplus Q_{\raltseq{\vert n-2s\vert}{0}}\right) \oplus Q_{\lraltseq{2s}{1}{0}}.
\end{gather*} 
By the Krull-Schmidt property of $\mathrm{add}(\bigoplus_{w\in \Syaff_2} Q_w)$, where the additive closure is taken in the category of all $Q$-$Q$-bimodules, we see that 
$\mathrm{Loc}(\rB_{\raltseq{k}{0}})$ can only be a direct summand of 
$\mathrm{Loc}\left(\rB_{\raltseq{m}{1}}\rY^{-s,s}\right)$ if $k=1$ and $m\geq 2s$. 

Thus we have an embedding $f^{-s,s}\colon \rB_0\langle 1\rangle\to \rB_{\raltseq{m}{1}}\rY^{-s,s}$ in $K^b(\Sext_2)$, with $m\geq 2s$. By adjunction, this implies that there is a non-zero morphism $\rB_{\laltseq{m}{1}} \rB_0\langle 1\rangle\to \rY^{-s,s}$. Since $[\rY^{-s,s}]^{\min}$ is given by 
$\rB_{\lraltseq{2s}{1}{0}}$ in homological degree zero, such a non-zero morphism can only exist if $m=2s$, thanks to \autoref{thm:extcategorification} and equations \eqref{eq:Soergelhomformula} and \eqref{eq:multconst}. In particular, we have  
\[
\rB_{\lraltseq{2s}{0}{1}}\rB_{\lraltseq{2s}{1}{0}}\cong \left(\rB_{\raltseq{4s-1}{0}}\oplus \cdots \oplus \rB_{0}\right)^{q+q^{-1}},  
\]
so the embedding $f^{-s,s}\colon \rB_0\langle 1\rangle\to \rB_{\lraltseq{2s}{0}{1}}\rY^{-s,s}$ is induced by the identity on $\rB_0\langle 1\rangle$. However, we claim that that embedding does not split. Recall that $[\rY^{-s,s}]^{\min}\cong [(\rT_1^{-1}\rT_0^{-1})^s]^{\mathrm{min}}$ is given by 
\[
\cdots \xlongrightarrow{}\rB_{\lraltseq{2s-1}{0}{0}}\langle -1 \rangle\oplus \rB_{\lraltseq{2s-1}{1}{1}}\langle -1\rangle
\xlongrightarrow{\mathrm{d}^{-1}}\underline{\rB_{\lraltseq{2s}{1}{0}}}
\]
where 
\[
\mathrm{d}^{-1}=(\mathrm{d}^{-1}_{\mathrm{L}}, \mathrm{d}^{-1}_{\mathrm{R}})= 
\left(
\xy (0,.05)*{
\tikzdiagc[scale=0.4,yscale=1]{
  \draw[ultra thick,blue] (-1.5,-1.75) -- (-1.5,-.75)node[pos=0,tikzdot]{};
  \draw[ultra thick,blue] (-1.5,.75) -- (-1.5,2)node[above]{\tiny $1$};
  \draw[ultra thick,violet] (1.5,-2)node[below]{\tiny $0$} -- ( 1.5,-.75);
   \draw[ultra thick,violet] (1.5,.75) -- ( 1.5,2)node[above]{\tiny $0$};
 \draw[ultra thick,violet] (-1,-2)node[below]{\tiny $0$}  --  (-1,-.75); 
 \draw[ultra thick,violet] (-1,.75) --  (-1,2)node[above]{\tiny $0$};
  \draw[ultra thick,blue] (1,-2)node[below]{\tiny $1$}  --  (1,-.75); 
    \draw[ultra thick,blue] (1,.75)--  (1,2)node[above]{\tiny $1$}; 
    \node at (0,1.5) {$\cdots$};
     \node at (0,-1.5) {$\cdots$};
  \draw [draw=black] (-2.25,-.75) rectangle (2.25,.75);
  \node at (0,0) {\tiny $\mathrm{JW}_{(1,\ldots,0)}$};
}}
\endxy
\;,\; 
\xy (0,.05)*{
\tikzdiagc[scale=0.4,yscale=1]{
  \draw[ultra thick,blue] (-1.5,-2)node[below]{\tiny $1$} -- (-1.5,-.75);
  \draw[ultra thick,blue] (-1.5,.75) -- (-1.5,2)node[above]{\tiny $1$};
  \draw[ultra thick,violet] (1.5,-1.75)-- ( 1.5,-.75)node[pos=0, tikzdot]{};
   \draw[ultra thick,violet] (1.5,.75) -- ( 1.5,2)node[above]{\tiny $0$};
 \draw[ultra thick,violet] (-1,-2)node[below]{\tiny $0$} -- (-1,-.75); 
 \draw[ultra thick,violet] (-1,.75)-- (-1,2)node[above]{\tiny $0$};
  \draw[ultra thick,blue] (1,-2)node[below]{\tiny $1$} -- (1,-.75); 
    \draw[ultra thick,blue] (1,.75) -- (1,2)node[above]{\tiny $1$}; 
    \node at (0,1.5) {$\cdots$};
     \node at (0,-1.5) {$\cdots$};
  \draw [draw=black] (-2.25,-.75) rectangle (2.25,.75);
  \node at (0,0) {\tiny $\mathrm{JW}_{(1,\ldots,0)}$};
}}
\endxy
\right)
\]
Tensoring this complex on the left with $\rB_{\lraltseq{2s}{0}{1}}$ (over $R$) yields a complex containing the morphism  
\[
\rB_{\lraltseq{2s}{0}{1}}\rB_{\lraltseq{2s-1}{1}{1}}\langle -1\rangle
\xlongrightarrow{\mathrm{id}_{\rB_{\lraltseq{2s}{0}{1}}}\circ\, \mathrm{d}^{-1}_{\mathrm{R}}}\underline{\rB_{\lraltseq{2s}{0}{1}}\rB_{\lraltseq{2s}{1}{0}}}
\]
When we compose this morphism with the projection of $\rB_{\lraltseq{2s}{0}{1}}\rB_{\lraltseq{2s}{1}{0}}$ onto $\rB_{0}\langle 1\rangle$, we get a non-zero multiple of 
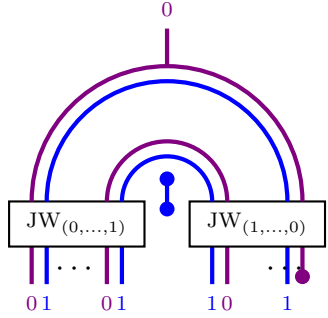
\[
\xy (0,.05)*{
\tikzdiagc[scale=0.4,yscale=1]{
  \draw[ultra thick,violet] (-4.5,-2)node[below]{\tiny $0$} -- (-4.5,-.75);
  \draw[ultra thick,blue] (-1.5,-2)node[below]{\tiny $1$}-- (-1.5,-.75);
   \draw[ultra thick,blue] (-4,-2)node[below]{\tiny $1$}  --  (-4,-.75); 
  \draw[ultra thick,violet] (-2,-2)node[below]{\tiny $0$}  --  (-2,-.75);  
     \node at (-3,-1.5) {$\cdots$};
  \draw [draw=black] (-5.25,-.75) rectangle (-.75,.75);
  \node at (-3,0) {\tiny $\mathrm{JW}_{(0,\ldots,1)}$};
  \draw[ultra thick,blue] (1.5,-2)node[below]{\tiny $1$} -- (1.5,-.75);
  \draw[ultra thick,violet] (4.5,-1.75)-- (4.5,-.75)node[pos=0, tikzdot]{};
 \draw[ultra thick,violet] (2,-2)node[below]{\tiny $0$}  --  (2,-.75); 
  \draw[ultra thick,blue] (4,-2)node[below]{\tiny $1$}  --  (4,-.75); 
     \node at (4,-1.5) {$\cdots$};
  \draw [draw=black] (.75,-.75) rectangle (5.25,.75);
  \node at (3,0) {\tiny $\mathrm{JW}_{(1,\ldots,0)}$};
  \draw[ultra thick, violet] (4.5,.75) arc(0:180:4.5);
   \draw[ultra thick, blue] (4,.75) arc(0:180:4);
    \draw[ultra thick, violet] (2,.75) arc(0:180:2);
     \draw[ultra thick, blue] (1.5,.75) arc(0:180:1.5);
    \draw[ultra thick, violet] (0,5.25) --(0,6.5)node[above]{\tiny $0$};
     \draw[ultra thick, blue] (0,.5) -- (0,1.5)node[pos=0, tikzdot]{}node[pos=1, tikzdot]{};
}}\endxy
\]
It is easy to see that the latter diagram is non-zero. For example, attaching a dot to the top 
$0$-strand and cups to the $2s$ bottom left strands, and using adjunction to straighten the resulting diagram, yields 
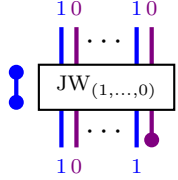
\[
\xy (0,.05)*{
\tikzdiagc[scale=0.4,yscale=1]{
  \draw[ultra thick,blue] (1.5,-2)node[below]{\tiny $1$} -- (1.5,-.75);
  \draw[ultra thick,violet] (4.5,-1.75)-- (4.5,-.75)node[pos=0, tikzdot]{};
 \draw[ultra thick,violet] (2,-2)node[below]{\tiny $0$}  --  (2,-.75); 
  \draw[ultra thick,blue] (4,-2)node[below]{\tiny $1$}  --  (4,-.75); 
     \node at (3,1.5) {$\cdots$};
     \draw[ultra thick,blue] (1.5,.75)-- (1.5,2)node[above]{\tiny $1$} ;
  \draw[ultra thick,violet] (4.5,.75) -- (4.5,2)node[above]{\tiny $0$};
 \draw[ultra thick,violet] (2,.75) -- (2,2)node[above]{\tiny $0$}; 
  \draw[ultra thick,blue] (4,.75) -- (4,2)node[above]{\tiny $1$}; 
     \node at (3,-1.5) {$\cdots$};
  \draw [draw=black] (.75,-.75) rectangle (5.25,.75);
  \node at (3,0) {\tiny $\mathrm{JW}_{(1,\ldots,0)}$}; 
     \draw[ultra thick, blue] (0,-.5) -- (0,.5)node[pos=0, tikzdot]{}node[pos=1, tikzdot]{};
}}
\endxy
\]
which is non-zero, as we know. This implies that the projection of $\rB_{\lraltseq{2s}{0}{1}}\rB_{\lraltseq{2s}{1}{0}}$ onto $\rB_{0}\langle 1\rangle$, which is unique up to a 
non-zero scalar, does not induce a morphism of complexes $\rB_{\lraltseq{2s}{0}{1}}\rY^{-s,s}\to \rB_0\langle 1\rangle$, so the embedding $f^{-s,s}$ does not split. 

This shows that $\rB_{\raltseq{k}{0}}\rY\langle t\rangle$ and $\rB_{\raltseq{m}{1}}\rY$ are not isomorphic for any $k,m\in \mathbb{Z}_{\geq 0}$ such that $k>0$ and any $t\in \mathbb{Z}$, which was the last case we had to prove.  \qedhere
\end{enumerate}
\end{proof}

Let $\left[\mathbf{U}\right]_{\oplus}$ be the split Grothendieck group of $\mathbf{U}$. Since $[\Sext_2]_{\oplus}\cong \eah{2}$, the free 
$\mathbb{Z}[q,q^{-1}]$-module $\left[\mathbf{U}\right]_{\oplus}$ carries the structure of an $\eah{2}$-module. 

\begin{cor}\label{cor:decat}
The $\mathbb{Z}[q,q^{-1}]$-linear map $\gamma_U\colon U\to \left[\mathbf{U}\right]_{\oplus}$, defined by 
\[
\gamma_U(u_k):=\left[\rB_{\raltseq{k}{1}}\rY\right] \quad\text{and}\quad \gamma_U(u'_k):=\left[\rB_\rho\rB_{\raltseq{k}{1}}\rY\right], \quad 
k\in \mathbb{Z}_{\geq 0}, 
\]
is an isomorphism of $\eah{2}$-modules. 
\end{cor}
\begin{proof}
This corollary is an immediate consequence of \autoref{prop:basisofU} and \autoref{thm:mainresultexample}.    
\end{proof}

\subsubsection{A triangulated $\Sext_2$ birepresentation $\mathbf{W}$}\label{sec:orbitcats}

Unfortunately, there is no obvious triangulated structure on the category $\mathrm{mod}_{K^b(\Sextdiamond_2)}(\rY)$, e.g., there is no obvious way to define the cone of a morphism between two objects in that category. This is a well-known problem for categories of 
modules over algebra objects in triangulated monoidal categories, see e.g.~\cite{balmer2011}. We will therefore use a different strategy to obtain a triangulated birepresentation $\mathbf{W}$.

This strategy is motivated by the observation that horizontal composition on the right with $\rY^{r,s}$, for $(r,s)\in \mathbb{Z}^2$, defines a $\mathbb{Z}^2$-action on $K^b(\Sext_2)$. To be precise, this is a strong but not a strict action, because $\rY^{r,s}\rY^{r',s'}$ is isomorphic to $\rY^{r+r',s+s'}$ but not equal, in general. The non-trivial isomorphism is given by $\zeta$, defined in \autoref{sec:defzeta}, and the $\mathbb{Z}^2$-action is strong precisely because $\zeta$ 
is natural in both entries and satisfies the hexagon identities. 

Using this action, we can define the {\em orbit category} $\widehat{\mathbf{\Omega}}:=K^b(\Sext_2)/\rY^{\mathbb{Z},\mathbb{Z}}$, which has the same objects as $K^b(\Sext_2)$ and its morphism spaces are defined by 
\[
\mathrm{hom}_{\widehat{\mathbf{\Omega}}}(\mathrm{A},\mathrm{B}):=\bigoplus_{(r,s)\in \mathbb{Z}^2}
\mathrm{hom}_{K^b(\Sext_2)}(\mathrm{A},\mathrm{B}\rY^{r,s}),
\]
for $\mathrm{A},\mathrm{B}\in K^b(\Sext_2)$. Composition is defined as in \eqref{eq:free-forget-composition}. By construction, the left regular $K^b(\Sext_2)$ action on $K^b(\Sext_2)$ induces a 
left $K^b(\Sext_2)$-action on $\widehat{\mathbf{\Omega}}$. The natural functor $\Pi_{\widehat{\mathbf{\Omega}}}\colon K^b(\Sext_2)\to \widehat{\mathbf{\Omega}}$, which is the identity on objects and sends any morphism 
$f\colon \mathrm{A}\to \mathrm{B}$ in $K^b(\Sext_2)$ to 
$f\colon \mathrm{A}\to \mathrm{B}\rY^{0,0}$ in $\widehat{\mathbf{\Omega}}$, is a morphism of wide finitary $\Sext_2$-bireprentations. In the following proposition, we do not consider any triangulated structures yet.

\begin{prop}\label{prop:orbitequiv}
There is an equivalence of wide finitary $\Sext_2$ birepresentations 
\[
\mathrm{add}\left\{\mathrm{A}\rY\colon \mathrm{A}\in K^b(\Sext_2)\right\}\cong \widehat{\mathbf{\Omega}}, 
\]
where the additive closure on the left-hand side is taken in $K^b(\Sextdiamond_2)$.
\end{prop}
\begin{proof}
This is exactly the content of the free forgetful adjunction in \eqref{eq:free-forget} and 
\eqref{eq:free-forget-composition}.    
\end{proof}

\begin{cor}\label{cor:orbitequiv}
The equivalence in \autoref{prop:orbitequiv} restricts to an equivalence of wide finitary 
$\Sext_2$ birepresentations 
\[
\mathbf{U}\cong \widehat{\mathbf{\Omega}}',
\]
where $\widehat{\mathbf{\Omega}}'$ is the full $\Sext_2$ subbirepresentation of $\widehat{\mathbf{\Omega}}$ whose objects are those of $\Sext_2$, which can be identified with the complexes in $K^b(\Sext_2)$ concentrated in homological degree zero via the usual embedding.  
\end{cor}

In order to construct a triangulated $\Sext_2$ birepresentation, we proceed in two steps. We consider the $\mathbb{Z}^2$-action on $K^b(\Sext_2)$ as generated by $\rY^{1,1}\cong \rB_{\rho}^2$ and $\rY^{0,1}\cong \rT_1^{-1}\rB_\rho$. Instead of considering all orbits, we first use results of \cite[Section 3.4]{elias2018} to construct a new wide finitary category $\fSext_2$ which already contains a canonical isomorphism $\rB_\rho^2\cong R$, and then apply \cite[Theorem/Definition 1.1]{FKQ} to construct a triangulated orbit category from $K^b(\fSext_2)$ under the action of the second generator.

Recall from \eqref{eq:I} and \autoref{prop:projection} that $W$ is isomorphic to the quotient of $\eah{2}$ by the left ideal generated by 
\[
\rho^2-1\quad\text{and}\quad \rho-T_1.
\]
The first generator is central, so the ideal it generates is two-sided and 
\begin{equation}\label{eq:feah}
\feah{2}:=\eah{2}/\langle \rho^2-1\rangle
\end{equation}
is a quotient algebra. Of course, we have 
\begin{equation}\label{eq:Wfeahquotient}
W\cong \feah{2}/\feah{2}(\rho-T_1),
\end{equation}
where $\feah{2}(\rho-T_1)$ is the left $\feah{2}$ ideal generated by $\rho-T_1$. Elias~\cite[Section 3.4]{elias2018} showed that $\feah{2}$ can be categorified by equipping $\BSext_2$ with an additional morphism and its "inverse", imposing some additional relations, 
and defining the new wide finitary monoidal category $\fSext_2$ as the Karoubi envelope of the additive envelope of this "enhanced" $\BSext_2$. Concretely, we add the generators 
\begin{equation}\label{eq:Unewgensprime}
\xy (0,0)*{
\tikzdiagc[scale=1]{
\begin{scope}[yscale=-.5,xscale=.5,shift={(5,-2)}] 
  \draw[ultra thick] (-1,-1) -- (-1, 1)node[pos=0.5, tikzdot]{};
  \draw[ultra thick,-to] (-1,-.6) to (-1,-.5);
  \draw[ultra thick,-to] (-1,.6) to (-1,.5);
\end{scope}
\begin{scope}[yscale=.5,xscale=.5,shift={(9,2)}] 
  \draw[ultra thick] (-1,-1) -- (-1, 1)node[pos=0.5, tikzdot]{};
  \draw[ultra thick,to-] (-1,-.85) to (-1,-.75);
  \draw[ultra thick,to-] (-1,.85) to (-1,.75);
\end{scope}
\node at (0,0) {Degree};
\node at (2,0) {$0$};
\node at (4,0) {$0$};
}}\endxy
\end{equation}  
and the relations
\begingroup\allowdisplaybreaks
\begin{gather}\label{eq:Unewrelsaprime} 
\xy (0,0)*{
\tikzdiagc[scale=.75]{
  \draw[ultra thick] (-1,-1) -- (-1, 1)node[pos=0.3, tikzdot]{}node[pos=0.66, tikzdot]{};
  \draw[ultra thick,-to] (-1,-.8) to (-1,-.7);
  \draw[ultra thick,to-] (-1,-.15) to (-1,-.1);
  \draw[ultra thick,-to] (-1,.7) to (-1,.8);
}}\endxy
\mspace{8mu}=\mspace{12mu}
\xy (0,0)*{
\tikzdiagc[scale=.75]{
  \draw[ultra thick] (-1,-1) -- (-1, 1);
  \draw[ultra thick,-to] (-1,.1) to (-1,.15);
}}\endxy
\mspace{80mu}
\xy (0,0)*{
\tikzdiagc[scale=.75]{
  \draw[ultra thick] (-.75,1) to[out=-90,in=180] (0,0) to[out=0,in=-90] (.75,1);
\node[tikzdot] at (-.55,.3) {}; 
  \draw[ultra thick,-to] (-.72,.7) to (-.74,.8);
    \draw[ultra thick,-to] (.72,.7) to (.74,.8);
}}\endxy
\mspace{7mu}=\mspace{12mu}
\xy (0,0)*{
\tikzdiagc[scale=.75]{
  \draw[ultra thick] (-.75,1) to[out=-90,in=180] (0,0) to[out=0,in=-90] (.75,1);
\node[tikzdot] at (.55,.3) {}; 
  \draw[ultra thick,-to] (-.72,.7) to (-.74,.8);
    \draw[ultra thick,-to] (.72,.7) to (.74,.8);
}}\endxy
\mspace{80mu}
\xy (0,0)*{
\tikzdiagc[scale=.75]{
  \draw[ultra thick] (-.75,1) to[out=-90,in=180] (0,0) to[out=0,in=-90] (.75,1);
\node[tikzdot] at (-.55,.3) {}; 
  \draw[ultra thick,to-] (-.70,.6) to (-.72,.7);
  \draw[ultra thick,to-] (.70,.6) to (.72,.7);
}}\endxy
\mspace{7mu}=\mspace{12mu}
\xy (0,0)*{
\tikzdiagc[scale=.75]{
  \draw[ultra thick] (-.75,1) to[out=-90,in=180] (0,0) to[out=0,in=-90] (.75,1);
\node[tikzdot] at (.55,.3) {}; 
  \draw[ultra thick,to-] (-.70,.6) to (-.72,.7);
  \draw[ultra thick,to-] (.70,.6) to (.72,.7);
}}\endxy
\\ \label{eq:Unewrelsaaprime}
\xy (0,0)*{
\tikzdiagc[scale=.75]{
\draw[ultra thick,blue] (-1,-1) to[out=60,in=180] (0,-.4);
\draw[ultra thick,violet] (0,-.4) to[out=0,in=-90] (1,1);
  \draw[ultra thick] (0,-1) -- (0, 1)node[pos=0.5, tikzdot]{};
  \draw[ultra thick,-to] (0,-.8) to (0,-.7);
  \draw[ultra thick,to-] (0,.65) to (0,.75);
}}\endxy
\mspace{5mu}=\mspace{5mu}
\xy (0,0)*{
\tikzdiagc[scale=.75]{
\draw[ultra thick,blue] (-1,-1) to[out=90,in=180] (0,.4);
\draw[ultra thick,violet] (0,.4) to[out=0,in=-120] (1,1);
  \draw[ultra thick] (0,-1) -- (0, 1)node[pos=0.5, tikzdot]{};
  \draw[ultra thick,-to] (0,-.8) to (0,-.7);
  \draw[ultra thick,to-] (0,.65) to (0,.75);
}}\endxy
\mspace{60mu}
\xy (0,0)*{
\tikzdiagc[scale=.75]{
\draw[ultra thick,violet] (-1,-1) to[out=60,in=180] (0,-.4);
\draw[ultra thick,blue] (0,-.4) to[out=0,in=-90] (1,1);
  \draw[ultra thick] (0,-1) -- (0, 1)node[pos=0.5, tikzdot]{};
  \draw[ultra thick,to-] (0,-.9) to (0,-.8);
  \draw[ultra thick,to-] (0,.9) to (0,.8);
}}\endxy
\mspace{5mu}=\mspace{5mu}
\xy (0,0)*{
\tikzdiagc[scale=.75]{
\draw[ultra thick,violet] (-1,-1) to[out=90,in=180] (0,.4);
\draw[ultra thick,blue] (0,.4) to[out=0,in=-120] (1,1);
  \draw[ultra thick] (0,-1) -- (0, 1)node[pos=0.5, tikzdot]{};
  \draw[ultra thick,to-] (0,-.9) to (0,-.8);
  \draw[ultra thick,to-] (0,.9) to (0,.8);
}}\endxy
\end{gather}
It is easy to see that our relations in \eqref{eq:Unewrelsaprime} and \eqref{eq:Unewrelsaaprime} are equivalent to \cite[(3.22a-c), (3.23)]{elias2018}, and that they imply the relations below. 

\begin{lem}\label{lem:Unewrelsprime}
In $\fSext_2$, we have 
\begin{gather}\label{eq:Unewrelsa2prime}
\xy (0,0)*{
\tikzdiagc[scale=.75]{
  \draw[ultra thick] (-1,-1) -- (-1, 1)node[pos=0.3, tikzdot]{}node[pos=0.66, tikzdot]{};
  \draw[ultra thick,to-] (-1,-.9) to (-1,-.8);
  \draw[ultra thick,-to] (-1,0) to (-1,.05);
  \draw[ultra thick,to-] (-1,.7) to (-1,.8);
}}\endxy
\mspace{8mu}=\mspace{12mu}
\xy (0,0)*{
\tikzdiagc[scale=.75]{
  \draw[ultra thick] (-1,-1) -- (-1, 1);
  \draw[ultra thick,to-] (-1,.1) to (-1,.15);
}}\endxy
\mspace{80mu}
\xy (0,0)*{
\tikzdiagc[scale=.75, rotate=180]{
  \draw[ultra thick] (-.75,1) to[out=-90,in=180] (0,0) to[out=0,in=-90] (.75,1);
\node[tikzdot] at (-.55,.3) {}; 
  \draw[ultra thick,-to] (-.72,.7) to (-.74,.8);
    \draw[ultra thick,-to] (.72,.7) to (.74,.8); 
}}\endxy
\mspace{7mu}=\mspace{12mu}
\xy (0,0)*{
\tikzdiagc[scale=.75, rotate=180]{
  \draw[ultra thick] (-.75,1) to[out=-90,in=180] (0,0) to[out=0,in=-90] (.75,1);
\node[tikzdot] at (.55,.3) {}; 
  \draw[ultra thick,-to] (-.72,.7) to (-.74,.8);
    \draw[ultra thick,-to] (.72,.7) to (.74,.8);
}}\endxy
\mspace{80mu}
\xy (0,0)*{
\tikzdiagc[scale=.75,rotate=180]{
  \draw[ultra thick] (-.75,1) to[out=-90,in=180] (0,0) to[out=0,in=-90] (.75,1);
\node[tikzdot] at (-.55,.3) {}; 
  \draw[ultra thick,to-] (-.70,.6) to (-.72,.7);
  \draw[ultra thick,to-] (.70,.6) to (.72,.7);
}}\endxy
\mspace{7mu}=\mspace{12mu}
\xy (0,0)*{
\tikzdiagc[scale=.75,rotate=180]{
  \draw[ultra thick] (-.75,1) to[out=-90,in=180] (0,0) to[out=0,in=-90] (.75,1);
\node[tikzdot] at (.55,.3) {}; 
  \draw[ultra thick,to-] (-.70,.6) to (-.72,.7);
  \draw[ultra thick,to-] (.70,.6) to (.72,.7);
}}\endxy
\end{gather}
\end{lem}

By \cite[Lemma 3.25 and Theorem 3.28]{elias2018}, there is an isomorphism $\rB_\rho^2\cong R$ 
in $\fSext_2$, the split Grothendieck group of $\fSext_2$ is isomorphic 
to $\feah{2}$ and the natural functor $E\colon \Sext_2\to \fSext_2$ induces the projection $\eah{2}\cong [\Sext_2]_{\oplus}\to \feah{2}\cong [\fSext_2]_{\oplus}$.

Since $\rY^{1,1}=\rB_\rho^2$, the above implies that $\rY^{r,s}\cong \rY^{0,s-r}$ in $K^b(\fSext_2)$, for any $(r,s)\in \mathbb{Z}^2$. Hence, it suffices to consider a $\mathbb{Z}$-action on $K^b(\fSext_2)$, with $s\in \mathbb{Z}$ acting by horizontal composition on the right with $\rY^{0,s}=(\rT_1^{-1}\rB_\rho)^s$. Let $\overline{\mathbf{\Omega}}:=K^b(\fSext_2)/\rY^{0,\mathbb{Z}}$ be the corresponding orbit category and let $\overline{\mathbf{\Omega}'}$ be its full $\Sext_2$ subbirepresentation whose objects are those of $\fSext_2$, which can be identified with the objects in $K^b(\fSext_2)$ concentrated in homological degree zero. 

\begin{lem}\label{lem:E}
The natural $\mathbb{R}$-linear monoidal functor $E\colon K^b(\Sext_2)\to K^b(\fSext_2)$ 
induces a morphism of $\Sext_2$-birepresentations $E_{\Omega}\colon \widehat{\mathbf{\Omega}}\to \overline{\mathbf{\Omega}}$, which restricts and corestricts to a morphism of 
$\Sext_2$-birepresentations $E_{\Omega'}\colon \widehat{\mathbf{\Omega}}'\to \overline{\mathbf{\Omega}'}$.
\end{lem}

\begin{proof} On objects, the functor $E_\Omega$ is defined by the identity map, of course. Given a morphism $f\colon \mathrm{A}\to \mathrm{B}$ in $\widehat{\mathbf{\Omega}}$ with only one 
non-zero component $f^{r,s}\colon \mathrm{A}\to \mathrm{B}\rY^{r,s}$ in $K^b(\Sext_2)$, for some $r,s\in \mathbb{Z}$, define $E_\Omega(f)$ as the morphism in $\overline{\mathbf{\Omega}}$ whose only one non-zero component $E_\Omega(f)^{0,s-r}$ is given by postcomposing $E(f)$ with the isomorphism $\rY^{r,s}\cong \rY^{0,s-r}$ in $K^b(\fSext_2)$. This definition can be extended to all morphisms in $\mathrm{hom}_{\widehat{\mathbf{\Omega}}}(\mathrm{A},\mathrm{B})$ by additivity. 

The definition implies immediately that $E_\Omega$ is $\mathbb{R}$-linear and sends identity morphisms to identity morphisms, and the hexagon identities and binaturality of $\zeta$ guarantee that it also preserves composition. Finally, it is a morphism of left $\Sext_2$ birepresentations because the isomorphisms $\rY^{r,s}\cong \rY^{0,s-r}$ in the definition of $E_\Omega$ are applied to the rightmost factor of the objects.  

The last claim of the lemma follows immediately. 
\end{proof}

The following result is an immediate consequence of \autoref{cor:orbitequiv} and \autoref{lem:E}.

\begin{cor}\label{cor:E}
The functor $E_{\Omega'}$ induces a morphism of $\Sext_2$ birepresentations $E_{\Omega'}\colon \mathbf{U}\to \overline{\mathbf{\Omega}'}$ (for which we use the same notation).
\end{cor}
We do not know whether $E_{\Omega'}$ is an equivalence. It is certainly essentially surjective, because it is the identity on objects, but it is unclear whether it is fully faithful.  

The above fits exactly the conditions of \cite[Theorem/Definition 1.1]{FKQ}: the triangulated category $\mathcal{T}:=K^b(\fSext_2)$ is naturally endowed with the dg-enhancement $\mathcal{A}:=C^b_{dg}(\fSext_2)$ (the dg-category of bounded complexes in $\fSext_2$), and the complex $F:=\rY^{0,1}=\rT_1\rB_\rho$ naturally yields a dg-bimodule such that the induced endofunctor 
$H^0(F)$ on $H^0(\mathcal{A})\cong K^b(\Sext_2)$ is an equivalence (with inverse induced by $\rY^{0,-1}$). This means that there is a canonical triangulated orbit category $\mathbf{W}$ of $K^b(\fSext_2)$ under the action of $\rY^{0,1}$, which in the 
notation of \cite[Theorem/Definition 1.1]{FKQ} is defined by 
\[ 
\mathbf{W}:=H^0\left(\mathrm{pretr}\left(C^b_{dg}(\fSext_2)/\rY^{0,\mathbb{Z}}\right)\right).
\]
The left $\Sext_2$-action clearly commutes with the right action by $\rY^{0,1}$, hence 
$\mathbf{W}$ is a triangulated left $\Sext_2$-birepresentation. 

By \cite[Theorem 3.17]{FKQ}, there is a natural dg-functor 
\begin{equation}\label{eq:Q}
Q\colon C^b_{dg}\left(\fSext_2\right)\to \mathrm{pretr}\left(C^b_{dg}\left(\fSext_2\right)/\rY^{0,\mathbb{Z}}\right),
\end{equation}
which induces a triangle functor $Q^{\triangle}:=H^0(Q)$ between the homotopy categories 
\begin{equation}\label{eq:Qtriangle}
Q^{\triangle}\colon K^b\left(\fSext_2\right)\to \mathbf{W}.
\end{equation}
As $Q^{\triangle}$ is a morphism of triangulated $\Sext_2$-birepresentations, it induces a morphism of $[\Sext_2]_{\oplus}$-modules between the triangulated Grothendieck groups
\begin{equation}\label{eq:QtriangleGrothendieck}
\left[Q^{\triangle}\right]_{\triangle}\colon \left[K^b\left(\fSext_2\right)\right]_{\triangle}\to \left[\mathbf{W}\right]_{\triangle}.
\end{equation}
A well-known general result on Grothendieck groups, see e.g. \cite{rose2011}, implies 
that $[K^b(\fSext_2)]_{\triangle}\cong [\fSext]_{\oplus}\cong \feah{2}$. Thus \eqref{eq:QtriangleGrothendieck} yields a morphism of $\eah{2}$-modules  
\[
q\colon \feah{2}\to \left[\mathbf{W}\right]_{\triangle}.
\]
We claim that this morphism in an epimorphism. Indeed, consider the dg functor $Q$ in \eqref{eq:Q}.  The objects of $C^b_{dg}(\fSext_2)/Y^{0,\mathbb{Z}}$ 
are the same as those of $C^b_{dg}(\fSext_2)$. In taking the pretriangulated closure, we add cones, but those do not generate new elements in the Grothendieck group, proving the claim.

By construction, horizontal composition with $\rY^{0,1}$ is naturally isomorphic to the identity functor on 
$\mathbf{W}$, so $q$ factors through $W$ by \eqref{eq:Wfeahquotient}, resulting in a commutative triangle 
\[
 \begin{tikzcd}
 \feah{2} \arrow{r}{}  \arrow{rd}{q} 
   &  W \arrow{d}{\gamma_W} \\
     & \left[\mathbf{W}\right]_{\triangle}
\end{tikzcd}
 \]
Surjectivity of $q$ implies surjectivity of $\gamma_W$. Recall also that $W^{\mathbb{C}(q)}$ is simple, see \autoref{prop:simplicityW}. However, we do not know that $\left[\mathbf{W}\right]_{\triangle}^{\mathbb{C}(q)}$ is non-zero, which is why we have to 
settle for a conjecture. 
\begin{conj}\label{conj:example}
The morphism  
\[
\gamma_W^{\mathbb{C}(q)}\colon W^{\mathbb{C}(q)}\to \left[\mathbf{W}\right]_{\triangle}^{\mathbb{C}(q)}
\]
is an isomorphism. 
\end{conj}

Finally, recall the morphism of $\Sext_2$ birepresentations $E_{\Omega'}\colon \mathbf{U} \to 
\overline{\mathbf{\Omega}'}$ from \autoref{cor:E}. By \cite[Remark 3.12]{FKQ}, there is an equivalence of wide finitary $\Sext_2$ birepresentations
\[
\overline{\mathbf{\Omega}}\cong H^0(C^b_{dg}(\fSext_2)/Y^{0,\mathbb{Z}}).
\]
Precomposition with $E_{\Omega'}$ and postcomposition with the canonical functor $H^0(C^b_{dg}(\fSext_2)/Y^{0,\mathbb{Z}})\to H^0(\mathrm{pretr}(C^b_{dg}(\fSext_2)/Y^{0,\mathbb{Z}}))$ yields 
a morphism of $\Sext_2$ birepresentations 
\begin{equation}\label{eq:catprojection}
\Pi^{\mathbf{U}}_{\mathbf{W}}\colon \mathbf{U}\to \mathbf{W}.
\end{equation}
We conjecture that $\Pi^{\mathbf{U}}_{\mathbf{W}}$ categorifies the projection $\pi^U_W$ in 
\autoref{prop:projection}, meaning that the following diagram commutes
\[
\begin{tikzcd}
U \arrow{d}{\gamma_U} \arrow{rr}{\pi^U_W} & & W \arrow{d}{\gamma_W} \\
\left[\mathbf{U}\right]_{\oplus} \arrow{r}{\left[\Pi^{\mathbf{U}}_{\mathbf{W}}\right]} & \left[\mathbf{W}\right]_{\oplus} 
\arrow{r} & \left[\mathbf{W}\right]_{\triangle} 
\end{tikzcd}
\]



\end{document}